\documentclass[11pt]{amsart}

\usepackage[T1]{fontenc}
\usepackage[utf8]{inputenc}

\usepackage[english]{babel}
\usepackage{geometry}
\usepackage{graphicx}

\usepackage{xcolor}
\usepackage{amsmath,amssymb,amsthm,amsfonts,stmaryrd,mathtools}
\usepackage{bbm}
\usepackage{eurosym}
\usepackage{array,dcolumn,mathdots,multirow}
\usepackage{rotating}
\usepackage{fancyvrb}
\usepackage[activate={true,nocompatibility},spacing,kerning]{microtype}

\graphicspath{%
{./Images/}
}

\definecolor{darkgreen}{rgb}{0,0.5,0}
\definecolor{darkred}{rgb}{0.5,0,0}
\definecolor{darkblue}{rgb}{0.004,0.396,0.741}
\definecolor{gcolor}{rgb}{0.004,0.396,0.741}	
\definecolor{warning}{rgb}{1.0,0.6,0.0}	
\definecolor{gray}{rgb}{0.6,0.6,0.6}

\usepackage[pdftex,
            colorlinks,
            hyperfootnotes=false,
            pdffitwindow=true,
            plainpages=false,
            pdfpagelabels=true,
            pdfpagemode=UseOutlines,
            pdfpagelayout=TwoPageRight,
            pdftitle={Asymptotic Expansions Relating to Longest Monotone Subsequences},
            pdfauthor={Folkmar Bornemann, TU Munich},
            hyperindex,
            setpagesize=false,
            filecolor=purple,
            urlcolor=darkred,
            citecolor=gcolor,
            linkcolor=gcolor]{hyperref}
\usepackage{remreset}
\usepackage{bm}

\usepackage{listings}
\definecolor{mygreen}{rgb}{0,0.6,0}
\definecolor{mygray}{rgb}{0.5,0.5,0.5}
\definecolor{lightgray}{rgb}{0.925,0.925,0.925}
\definecolor{mymauve}{rgb}{0.58,0,0.82}
\definecolor{gcolor}{rgb}{0.004,0.396,0.741}
\usepackage{caption}
\theoremstyle{plain}
\newtheorem{conjecture}{Conjecture}[section]
\newtheorem{theorem}{Theorem}[section]
\newtheorem{lemma}{Lemma}[section]
\newtheorem{corollary}{Corollary}[section]
\theoremstyle{definition}
\newtheorem{definition}{Definition}[section]
\theoremstyle{remark}
\newtheorem{remark}{Remark}[section]

\newtheorem{example}{Example}[section]

\input{macros.sty}
\VerbatimFootnotes
\makeindex
\begin{document}

\title[Asymptotic Expansions Relating to Longest Monotone Subsequences]{Asymptotic Expansions Relating to the Lengths of Longest Monotone Subsequences of Involutions}
\author{Folkmar Bornemann}
\address{Department of Mathematics, TU München, Germany}
\email{bornemann@tum.de}


\begin{abstract}
We study the distribution of the length of longest monotone subsequences in random (fixed-point free) involutions of $n$ integers as $n$ grows large, establishing asymptotic expansions in powers of $n^{-1/6}$ in the general case and in powers of $n^{-1/3}$ in the fixed-point free cases. Whilst the limit laws were shown by Baik and Rains to be one of the Tracy--Widom distributions $F_\beta$ for $\beta=1$ or $\beta=4$, we find explicit analytic expressions of the first few expansions terms as linear combinations of higher order derivatives of $F_\beta$ with rational polynomial coefficients. Our derivation is based on a concept of generalized analytic de-Poissonization and is subject to the validity of certain hypotheses for which we provide compelling (computational) evidence. In a preparatory step expansions of the hard-to-soft edge transition laws of L$\beta$E are studied, which are lifted into expansions of the generalized Poissonized length distributions for large intensities. (This paper continues our work \cite{arxiv:2301.02022}, which established similar results in the case of general permutations and $\beta=2$.)
\end{abstract}

\maketitle

\section{Introduction}

We denote by $S_n$ the symmetric group of permutations $\sigma$ on the set $\{1,2,\ldots,n\}$. The subgroup of involutions is given by
\[
S^\boxdot_n := \{\sigma \in S_n : \sigma = \sigma ^{-1} \};
\]
and, since fixed-point free involutions exist only for $n$ even, we write their subset in the form
\[
S_{n}^\boxtimes := \{\sigma \in S^\boxdot_{2n} : \sigma(x) \neq x \text{ for all $x=1,2,\ldots,2n$}\}.
\]
We study the length $L_n^\boxdot(\sigma)$ of longest increasing subsequences\footnote{Defined as the maximum of all $k$ for which there are
$1\leq i_1 < i_2 < \cdots < i_k \leq n$ with $\sigma_{i_1} < \sigma_{i_2} < \cdots < \sigma_{i_k}$.}  of $\sigma \in S^\boxdot_n$
as well as the lengths $L_n^\boxslash(\sigma)$, $L_n^\boxbslash(\sigma)$ of longest increasing and of longest decreasing subsequences of $\sigma \in S^\boxtimes_n$.\footnote{The set of involutions $S^\boxdot_n$ is invariant under reversal of
permutations, so that longest increasing and longest decreasing subsequences have the same enumerative combinatorics. In contrast, this symmetry is broken for the set of fixed-point free involutions  $S^\boxtimes_n$ and we have to consider both cases separately.}

Drawing the involutions $\sigma$ uniformly from their respective sets, the functions $L_n^\varoast$, where we write $\varoast\in\{\boxslash,\boxbslash,\boxdot\}$, become discrete random variables. The present paper studies asymptotic expansions of their distributions as $n$ grows large. It continues our work \cite{arxiv:2301.02022} on the general permutation case, where we established asymptotic expansions of the distribution of the length $L_n(\sigma)$ of longest increasing subsequences of $\sigma \in S_n$. 

\subsection*{Constructive Combinatorics} Using the Robinson–Schensted correspondence \cite{MR121305} and a result of Schützenberger \cite[p.~127]{MR190017}, which characterizes the fixed-point free involutions as corresponding to Young diagrams  with all columns having even length, we get the following formulae (see, e.g., \cite[§5]{MR2334203}):\footnote{In the general permutation case the corresponding formula is $\big|\{L_n = l\}\big|= \sum_{\lambda \,\vdash n\,:\, l(\lambda)=l} d_\lambda^2 = \sum_{\lambda \,\vdash n\,:\, \lambda_1=l} d_\lambda^2$.}
\begin{equation}\label{eq:SYT}
\begin{gathered}
\big|\{L_n^\boxdot = l\}\big| = \sum_{\lambda \,\vdash n\,:\, l(\lambda)=l} d_\lambda = \sum_{\lambda \,\vdash n\,:\, \lambda_1=l} d_\lambda \\*[1mm]
\big|\{L_n^\boxslash = l\}\big| = \sum_{\lambda \,\vdash n\,:\, l(\lambda)=l} d_{2\lambda}, \qquad \big|\{L_n^\boxbslash = 2l\}\big| = \sum_{\lambda \,\vdash n\,:\, \lambda_1=l} d_{2\lambda}.
\end{gathered}
\end{equation}
Here $\lambda \,\vdash n$ denotes a partition $\lambda_1\geq \lambda_2 \geq \cdots \geq \lambda_{l(\lambda)} > 0$ of the integer $n = \sum_{j=1}^{l(\lambda)} \lambda_j$ and $d_\lambda$ is the number of standard Young tableaux of shape $\lambda$.
In particular, we see that the random variables $L_n^\varoast$ take their values in the following ways:
\begin{itemize}\itemsep=6pt
\item $L_n^\varoast$ ($\varoast=\boxslash,\boxdot$) in the set $\{1,2,\ldots,n\}$;
\item $L_n^\boxbslash$ in the set $\{2,4,\ldots,2n\}$ of even integers.
\end{itemize}
Therefore, we subsume the discrete probabilities and their distributions into a single notation by writing, for the discrete probabilities which are all vanishing for $l > n$,
\begin{equation}\label{eq:prob_dens}
p^*_\boxslash(n;l):=\prob(L_{n}^\boxslash = l), \quad p^*_\boxbslash(n;l):=\prob(L_{n}^\boxbslash = 2l), \quad p^*_\boxdot(n;l):=\prob(L_{n}^\boxdot = l),
\end{equation}
and, for the cumulative probabilities which are all unity for $l\geq n$,
\begin{equation}\label{eq:prob_dist}
p_\boxslash(n;l):=\prob(L_{n}^\boxslash \leq l), \quad p_\boxbslash(n;l):=\prob(L_{n}^\boxbslash \leq 2l), \quad p_\boxdot(n;l):=\prob(L_{n}^\boxdot \leq l).
\end{equation}

\subsection*{Generalized Poissonization} Asymptotic methods in enumerative combinatorics and in probability theory profit from an explicit identification of certain generating functions. One such approach, which turned out to be key in the general permutation case (for a review of that case see \cite[§1]{arxiv:2301.02022}), is the use of Poisson generating functions, or Poissonization: the index $n$ is replaced by an independent random 
variable $N_r \in \N_0:=\{0,1,2,\ldots \}$ with Poisson distribution of intensity $r\geq 0$, that is,
\[
\prob(N_r = n) = e^{-r} \cdot \frac{r^n}{n!}.
\]
In the fixed-point free cases, by the work of Rains \cite[Thm.~3.4]{Rains98}, the distribution of the combined random variables $L_{N_r}^\boxslash, L_{N_r}^\boxbslash$, briefly written as
\begin{subequations}\label{eq:gen_poisson}
\begin{align}
\prob(L_{N_r}^\boxslash \leq l) &= e^{-r} \sum_{n=0}^\infty p_\boxslash(n;l) \frac{r^n}{n!} =: P_\boxslash(r;l),\\*[1mm]
\prob(L_{N_r}^\boxbslash \leq 2l) &= e^{-r} \sum_{n=0}^\infty p_\boxbslash(n;l) \frac{r^n}{n!}=:P_\boxbslash(r;l),
\end{align}
can be expressed in terms of certain group integrals (see Sect.~\ref{sect:groupintegrals}). Up to an exponential factor these two Poisson generating functions are basically the exponential generating functions of the corresponding enumerative problem (see Sect.~\ref{sect:comb_gen}).

In the general involution case $\varoast=\boxdot$ no useful explicit formula for the distribution of $L_{N_r}^\boxdot$ has been found so far. 
Therefore, in their work on establishing limit laws for the involution cases, Baik and Rains \cite[§4]{MR1845180} suggested to introduce a second independent Poisson random variable taking the role of the number of fixed-points. The limit law is then obtained, after a multivariate de-Poissonization, only indirectly by identifying the dominant range of fixed-points \cite[§8]{MR1845180}. 

Instead, our suggestion is to stick with the combinatorial exponential generating function and to generalize the concept of Poissonization itself. To this end we denote the number of involutions by $I_n:=|S^\boxdot_n|$ and recall  (see, e.g., \cite[Example~II.13]{MR2483235})
 that the exponential generating function of this sequence is given by
 \[
e^{z+z^2/2} = \sum_{n=0}^\infty I_n \frac{z^n}{n!}.
\]
Thus, for $r\geq 0$, there is an independent random variable $N_r^* \in \N_0$ with distribution 
\[
\prob(N_r^* = n) = e^{-r-\frac{r^2}{2}} \cdot \frac{I_n r^n}{n!},
\]
which is an example of the broader concept of a {\em generalized Poisson distribution} of in\-ten\-sity~$r$ induced by an entire function $f$ which we discuss in Sect.~\ref{sect:gen_de_poisson}. Now, the distribution of the thus combined random variable $L_{N_r^*}^\boxdot$, briefly written as
\begin{equation}
\prob(L_{N_r^*}^\boxdot \leq l) = e^{-r-\frac{r^2}{2}} \sum_{n=0}^\infty p_\boxdot(n;l) \frac{I_n r^n}{n!} =: P_\boxdot(r;l),
\end{equation}
\end{subequations}
can also be expressed in terms of a group integral (see Sect.~\ref{sect:groupintegrals}). All three (generalized) Poisson generating functions can be continued as entire functions $P_\varoast(z;l)$ to $z\in\C$. 

In the fixed-point free cases $\varoast=\boxslash,\boxbslash$, Baik and Rains obtained limit laws for the Poissonized distributions $P_\varoast(r;l)$ when $r \to\infty$ while $l$ is kept near the mode of the distribution. To this end they 
represented the group integrals in terms of Hankel and Toeplitz determinants of modified Bessel functions \cite[Thm.~2.5]{MR1844203},\footnote{By a different method, such a representation was first obtained by Gessel \cite{Gessel90}
for the combinatorial exponential generating function corresponding to $P_\boxdot(r;l)$ and by 
Goulden \cite{MR1158782} for the one corresponding to $P_\boxbslash(r;l)$.} subsequently analyzing the double-scaling limits by the machinery of Riemann--Hilbert problems \cite{MR1845180}. 

Borodin and Forrester \cite{MR1986402} identified
the Poissonized distributions of the fixed-point free cases as yet another well studied probability distribution in random matrix theory:\footnote{In fact, by using the results of \cite{MR2066104}, Borodin and Forrester were restricted to state $P_\boxslash(r;l)=E_4^{\text{hard}}(8r;l)$ for $l$ even only. That this formula is valid independent of the parity of $l$ was recently established by Forrester and Mays \cite[p.~15]{arxiv.2205.05257}; a different proof is given in Sect.~\ref{sect:hard-edge}.}
\begin{subequations}\label{eq:Poisson}
\begin{equation}
P_\boxslash(r;l) = E_4^{\text{hard}}(8r;l), \qquad
P_\boxbslash(r;l) = E_1^{\text{hard}}(8r;l).
\end{equation}
Here\footnote{Throughout the paper, we will use $l$ as an integer $l\geq 1$ and $\nu$ as a corresponding real variable that is used whenever an expression of $l$ generalizes to non-integer arguments.} $E_\beta^\text{hard}(s;\nu)$ denotes the probability that, in the hard-edge scaling limit, the scaled smallest eigenvalue of the Laguerre $\beta$ ensemble\footnote{We choose scalings of the weights of the Gaussian and Laguerre ensembles as in \cite{MR1842786} and \cite[§§2.1/2.4]{MR2895091}, which differs for $\beta=4$ from the choices made at other places in the literature: e.g., the original $\beta=4$ Tracy–Widom distribution, as introduced in \cite{MR1385083}, is the function $F_4(\sqrt{2} s)$ here (cf., e.g., \cite[Eq.~(2.10)]{MR2895091} and \cite[Eq.~(35)]{MR2165698}).}  (L$\beta$E) with real parameter $\nu >-1$ is bounded below by $s\geq 0$. For the general involution case $\varoast=\boxdot$, we will add in Sect.~\ref{sect:gen} (see Eqs.~\eqref{eq:groupintegrals_fdot} and \eqref{eq:E1_group}) a similar representation of the generalized Poissonized distribution to the picture:
\begin{equation}
P_\boxdot(r;l) = E_1^{\text{hard}}\left(4r^2;\frac{l-1}{2}\right).
\end{equation}
\end{subequations}
Next, Borodin and Forrester reclaimed the Poissonized Baik--Rains limit laws for $\varoast = \boxslash,\boxbslash$ 
by establishing the hard-to-soft edge transition of L$\beta$E ($\beta=1,2,4$), as $\nu\to\infty$, in form of the 
limit law \cite[Thm.~1/Cor.~2]{MR1986402}
\begin{subequations}\label{eq:BF2003}
\begin{equation}
\lim_{\nu\to\infty} E_\beta^{\text{hard}}\left(\big(\nu-t (\nu/2)^{1/3}\big)^2;\nu_\beta\right) = F_\beta(t),
\end{equation}
where we write briefly
\begin{equation}
\nu_\beta := 
\begin{cases}
(\nu-1)/2, & \beta = 1,\\*[0.5mm]
\nu, &\beta = 2,\\*[0.5mm]
\nu+1, &\beta = 4.
\end{cases}
\end{equation}
\end{subequations}
Here $F_\beta$ is the Tracy--Widom distribution of the Gaussian $\beta$ ensemble (G$\beta$E), i.e., the probability that in the soft-edge scaling limit the scaled largest eigenvalue is bounded from above by $-\infty < t < \infty$.

\subsection*{Expansions of Poissonized Distributions}

In Sect.~\ref{sect:hard-to-soft} we prove asymptotic expansions for the $\beta=1$ and $\beta=4$ hard-to-soft edge transition limits \eqref{eq:BF2003} (the case $\beta=2$ was already dealt with in \cite[§3]{arxiv:2301.02022}), which then lift to expansions of the (generalized) Poissonized distributions in the form (see Cor.~\ref{cor:gen_Poisson_expan})
\begin{equation}\label{eq:PvaroastIntro}
P_\varoast(r;l) = F_\beta(t) + \sum_{j=1}^m F_{\beta,j}(t) \cdot r_\varoast^{-j/3} + O\big(r_\varoast^{-(m+1)/3}\big)\bigg|_{t=t_{l^\varoast}(r_\varoast),\; \beta = \beta(\varoast)},
\end{equation}
uniformly valid for bounded $t$. Here the scaling is
\begin{equation}\label{eq:tnu_vanilla}
t_\nu(r):= \frac{\nu-2\sqrt{r}}{r^{1/6}} \qquad (r>0)
\end{equation}
and we use the abbreviations (note that $l = (l^\varoast)_{\beta(\varoast)}$ for $\varoast = \boxslash,\boxbslash$)
\begin{subequations}\label{eq:lr_varoast}
\begin{equation}
l^\varoast := 
\begin{cases}
l-1, & \varoast = \boxslash,\\*[0.5mm]
2l+1, & \varoast = \boxbslash,\\*[0.5mm]
l, & \varoast = \boxdot,
\end{cases}
\qquad
r_\varoast := 
\begin{cases}
2r, & \varoast = \boxslash,\boxbslash,\\*[1mm]
r^2, & \varoast = \boxdot,
\end{cases}
\end{equation}
and
\begin{equation}
\beta(\boxslash) := 4,\quad \beta(\boxbslash) := \beta(\boxdot) := 1.
\end{equation}
\end{subequations}

The leading order term in \eqref{eq:PvaroastIntro} (i.e.,  $m=0$) reclaims in the fixed-point free cases $\varoast=\boxslash,\boxbslash$ the Poissonized Baik--Rains limit laws \cite[Prop.~7.3]{MR1845180}. For $\varoast=\boxdot$ we get the generalized Poissonized limit law
\[
\lim_{r\to\infty} \prob\left(\frac{L_{N_r^*}^\boxdot - 2 r}{r^{1/3}}\leq t \right) = F_1(t),
\]
which appears to be new.

The first order terms in \eqref{eq:PvaroastIntro} (i.e., $m=1$) were studied, for some of the problem cases, prior to our current work. 
In the case $\varoast=\boxbslash$ such a term was obtained by Baik and Jenkins \cite[Thm.~1.2]{MR3161478}, who framed the combinatorial problem in terms of the Poissonized distribution of maximal nesting in random matchings. Using the machinery of Riemann--Hilbert problems and the Painlevé~II representation of the Tracy--Widom distribution $F_1$, they proved that
\begin{subequations}\label{eq:BaikJenkins}
\begin{equation}\label{eq:BaikJenkinsBeta1}
F_{1,1}(t) = -\frac{t^2}{60} F_1'(t) - \frac{1}{5} F_1''(t),
\end{equation}
with a suboptimal error of $O(r^{-1/2})$, though. In their recent study of finite-size effects, Forrester and Mays \cite[§3]{arxiv.2205.05257} found,\footnote{See, however,  Fn.~\ref{fn:FM23} below for a discussion of the error terms that are stated in their results.} for the cases $\varoast=\boxslash,\boxbslash$, expressions for $F_{\beta,1}$ in terms of operator traces and determinants and, alternatively, in terms of Painlevé~III$'$.

In Sect.~\ref{sect:generalform} we show that the expansion terms $F_{\beta,j}(t)$ can generally be expressed in terms of operator traces and determinants with kernels such as \eqref{eq:K}, a functional form that is rather unwieldy to handle. Combining the Tracy--Widom theory relating $F_\beta$ to Painlevé~II with factorizations of the $\beta=2$ hard- and soft-edge distributions, we prove in Sect.~\ref{sect:functionalform} that the Baik--Jenkins result \eqref{eq:BaikJenkinsBeta1} implies\footnote{This settles a question suggested by Forrester and Mays \cite[Rem.~3.7]{arxiv.2205.05257}.}
\begin{equation}\label{eq:BaikJenkinsBeta4}
F_{4,1}(t) = -\frac{t^2}{60} F_4'(t) - \frac{1}{5} F_4''(t).
\end{equation}
\end{subequations}
Now, based on (\ref{eq:BaikJenkins}a/b) and the compelling numerical evidence detailed in Appendix~\ref{app:evidence}, we hypothesize (see the ``linear form hypothesis'' in Sect.~\ref{sect:higherorder}) that the $F_{\beta,j}$ share,  in general and not just for $j=1$,  exactly the same structure as their $\beta=2$ counterparts in \cite[§4]{arxiv:2301.02022}: 

\smallskip

\begin{quote}
{\em The $F_{\beta,j}$ are always linear combinations of higher order derivatives of the limit distribution $F_\beta$ (up to order $2j$) with certain rational polynomial coefficients.}
\end{quote}

\smallskip

\noindent
Subject to that hypothesis, based on the technique that led us to the proof of \eqref{eq:BaikJenkinsBeta4}, we can explicitly calculate the polynomial coefficients. We did so for up to  $j=10$; the formulae are displayed in \eqref{eq:FbetaP} and in the supplementary material mentioned in Fn.~\ref{fn:suppl}. Note that in all cases inspected the polynomial coefficients share their sparsity pattern (that is, the pattern of non-zero rational coefficients) with the case $\beta=2$. Such a far-reaching similarity of functional form points to some common underlying structure that is, with current methodology, only poorly understood.

\subsection*{Generalized de-Poissonization}

Extracting the asymptotics of the distributions $p_\varoast(n;l)$ for large $n$ from their (generalized) Poissonizations $P_\varoast(r;l)$ for large intensities $r$ requires Tauberian-type (generalized) de-Poissonization techniques. 

First, by generalizing Johansson's de-Poissonization lemma \cite{MR1618351} in Sect.~\ref{sect:johansson}, based on monotonicity and sandwiching, we are able to prove from \eqref{eq:PvaroastIntro} in Thm.~\ref{thm:leading_de_poisson} that, in the fixed-point free cases,
\[
p_\varoast(n;l) = F_{\beta(\varoast)}\big(t_{l^\varoast}(2 n)\big) + O\big(n^{-1/6}\sqrt{\log n}\,\big) \qquad (\varoast =\boxslash,\boxbslash), 
\]
and, in the general involution case,
\[
p_\boxdot(n;l) = F_1\big(t_{l}(n)\big) + O\big(n^{-1/6}\sqrt{\log n}\,\big),
\]
uniformly for bounded $t$. These are the Baik--Rains limit laws \cite[Thms.~3.1/3.4]{MR1845180} with an error term added. 

Second, by establishing asymptotic expansions we see that the $O\big(n^{-1/6}\sqrt{\log n}\,\big)$ error estimates of the first approach are suboptimal. To this end, we generalize in Sect.~\ref{sect:jasz} the analytic de-Poissonization and associated Jasz expansions of Jacquet–Szpankowski \cite{MR1625392}. In the fixed-point free cases, which are subject to standard analytic de-Poissonization, we get expansions in powers of $n^{-1/3}$,
\[
p_\varoast(n;l) = F_{\beta(\varoast)}(t) + \sum_{j=1}^m F_{\varoast,j}(t) \cdot (2n)^{-j/3} + O\big(n^{-(m+1)/3}\big)\bigg|_{t=t_{l^\varoast}(2n)} \qquad (\varoast =\boxslash,\boxbslash),
\]
which are uniformly valid for bounded $t$. For technical reasons which we were not able to resolve so far, as for the general permutation case studied in \cite[§5]{arxiv:2301.02022} the proof has to assume a certain ``tameness hypothesis'' regarding the $l$-dependent families of the finitely many zeros of $P_\varoast(z;l)$ in certain sectors of the complex plane. In the general involution case $\varoast=\boxdot$, leaving the detailed analytical estimates to future work, we proceed in a purely formal fashion which leads us to an expansion in powers of $n^{-1/6}$ of the form
\[
p_\boxdot(n;l) = F_1(t) + \sum_{j=1}^m F_{\boxdot,j}(t) \cdot n^{-j/6} + O\big(n^{-(m+1)/6}\big)\bigg|_{t=t_{l+1}(n)}.
\]
Here, because of $F_{\boxdot,1}=0$, the sum starts effectively with the order $n^{-1/3}$ term. 

Subject to the linear form hypothesis, the $F_{\varoast,j}$ can be explicitly calculated (see Eqs.~\eqref{eq:FbetaD} and \eqref{eq:Fdot}) as linear combinations of higher-order derivatives of $F_{\beta(\varoast)}$ with rational polynomial coefficients: in the fixed-point free cases we give formulae for $j=1,2,3$, and in the general involution case for $1\leq j \leq 7$. In Figs.~\ref{fig:FbetaD} and \ref{fig:Fdot} those functional forms are checked against  tables  of exact values of $p_\varoast(n;l)$ for up to $n=1000$ as compiled in Sect.~\ref{sect:exact}.

Interestingly, in all three cases the expansions start with structurally the same first order term (see Cor.~\ref{cor:limit_law_finite_size}): in the fixed-point free cases they start with ($\beta=\beta(\varoast)$, $N=2n$)
\[
p_\varoast(n;l) = F_{\beta}(t) - \Big(\frac{t^2}{60} F_\beta'(t) + \frac{6}{5} F_\beta''(t) \Big)N^{-1/3} + O\big(N^{-2/3}\big)\bigg|_{t = t_{l^\varoast}(N)} \quad (\varoast=\boxslash,\boxbslash),
\]
and in the general involution case with
\[
p_\boxdot(n;l) = F_{1}(t) - \Big(\frac{t^2}{60} F_1'(t) + \frac{6}{5} F_1''(t) \Big)n^{-1/3} + O\big(n^{-1/2}\big)\bigg|_{t = t_{l+1}(n)}.
\]

\subsection*{Organization of the paper} In Sect.~\ref{sect:gen} we discuss the combinatorial exponential generating functions and their relation to the (generalized) Poissonized length distributions. We recall Rains' representations in terms of group integrals. By relating them to certain $\tau$-functions we prove the representations \eqref{eq:Poisson} in terms of the probabilities $E^\text{hard}_\beta(s;\nu)$ and get simple expressions in terms of a Chazy I equation along the way, which we have used to compile tables of the exact length  distributions for up to $n=1000$. Finally, we establish the $H$-admissibility (a new criterion is introduced in Appendix~\ref{app:hayman}) of the generating functions and discuss Stirling- and Regev-type formulae.

In Sect.~\ref{sect:hard-to-soft} we study the asymptotic expansion of the Borodin–Forrester hard-to-soft edge transition law \eqref{eq:BF2003}. Here we lay the foundational work for the concrete functional form of all the expansion terms in this paper. Whereas the general structure of the expansions, which gives their uniformity and differentiability, is found by expanding operator determinants, the appealing functional form of the expansion terms, as displayed in \eqref{eq:F_beta_j}, is established by combining the Tracy--Widom theory with factorizations of the $\beta=2$ 
hard- and soft-edge distributions. By going beyond the first order terms, the proof of the functional form is subject to the linear form hypothesis, the evidence of which is detailed in Appendix~\ref{app:evidence}. 

In Sect.~\ref{sect:poissondist} we apply the results of Sect.~\ref{sect:hard-to-soft} to the (generalized) Poissonized length distributions, generalizing the $m=1$ results of Baik and Jenkins \cite{MR3161478} and of Forrester and Mays \cite{arxiv.2205.05257}. 

In Sect.~\ref{sect:gen_de_poisson} we study the generalized de-Poissonization of a broad class of generalized Poisson generating functions. First, Johansson's de-Poissonization lemma \cite{MR1618351},  based on monotonicity and sandwiching, is generalized. Second, we formally derive a generalized Jasz expansion and calculate concrete expansion coefficients for the choice of the generalized Poisson distribution which underlies the general involution case. 

In Sect.~\ref{sect:depoisson} we establish the main results of the paper: rigorous error estimates for the Baik–Rains limit laws of the length distributions $p_\varoast(n;l)$ and, subject to some hypotheses, their asymptotic expansions. Additionally we discuss the modifications that apply to the discrete densities $p_\varoast^*(n;l)$.

Finally, in Sect.~\ref{sect:mean} we study the asymptotic expansions of the expected value and variance. Assuming some tail bounds, we are able to get many more concrete terms of such expansions than previously put forward in the literature. These terms from our theory are checked against model fits to data sets obtained from the combinatorial tables compiled in Sect.~\ref{sect:exact}.

\section{The Generating Functions}\label{sect:gen}

\subsection{The combinatorial exponential generating functions}\label{sect:comb_gen}

The cardinalities of the sets of (fixed-point free) involutions are
\[
 |S_{n}^\boxtimes | = (2n-1)!! = 1 \cdot 3\cdot 5 \, \cdots \, (2n-1)
\]
and, by definition, $|S^\boxdot_n| = I_n$.
Thus, by observing that $(2n)! = (2n-1)!! \cdot 2^n n!$, the exponential generating functions of the enumerations underlying the length distributions, that is,\footnote{With the understanding that the $n=0$ term is always $1$.}
\begin{equation}\label{eq:expon_gen}
\begin{gathered}
f^{\boxslash}_l(z) := \sum_{n=0}^\infty  \big|\{L_{n}^\boxslash\leq l\}\big| \, \frac{z^{2n}}{(2n)!}, \qquad
f^{\boxbslash}_{l}(z) := \sum_{n=0}^\infty  \big|\{L_{n}^\boxbslash\leq 2l\}\big| \, \frac{z^{2n}}{(2n)!},\\*[1mm]
f^{\boxdot}_{l}(z) := \sum_{n=0}^\infty  \big|\{L_{n}^\boxdot\leq l\}\big| \, \frac{z^n}{n!},
\end{gathered}
\end{equation}
are related to the (generalized) Poissonized distributions introduced in \eqref{eq:gen_poisson} by
\begin{equation}\label{eq:Poisson_gen}
P_\boxslash(z;l) = e^{-z} f_l^\boxslash(\sqrt{2z}),\quad
P_\boxbslash(z;l)  = e^{-z} f_{l}^\boxbslash(\sqrt{2z}), \quad
P_\boxdot(z;l) = e^{-z-z^2/2} f_{l}^\boxdot(z).
\end{equation}
For the fixed-point free cases we also consider the (entire) exponential generating functions of the discrete probability distributions themselves, which can be expressed in the form
\begin{equation}\label{eq:expon_gen_prob}
f^{\varoast}_l(\sqrt{2z}) = \sum_{n=0}^\infty p_\varoast(n;l) \frac{z^n}{n!} \qquad (\varoast = \boxslash,\boxbslash).
\end{equation}
\subsection{Group integrals}\label{sect:groupintegrals}

Using representation theory of the symmetric group and of the classical matrix groups, Rains \cite[Thm.~3.4]{Rains98} (cf. also \cite[Eqs.~(1.34/35)]{MR1844203} and \cite[Prop.~1.2]{MR1794352}) related the first two of the generating functions \eqref{eq:expon_gen} to group integrals, namely
\begin{subequations}\label{eq:groupintegrals_fpf}
\begin{align}
f_l^\boxslash(z) &= \E_{U\in O(l)} e^{z \tr U} = \frac{1}{2}\left(\E_{U\in O^+(l)} e^{z \tr U} + \E_{U\in O^-(l)} e^{z \tr U}\right),\\*[2.5mm] 
f_{l}^\boxbslash(z) &= \E_{U \in O^{-}(2l+2)} e^{z \tr U},
\end{align}
where the expectation is taken with respect to the Haar measure on (components of) the classical groups. Note that we have identified the group $\text{Sp}(2l)$ with $O^-(2l+2)$, since both share the same eigenvalue distribution except for the two additional eigenvalues $\pm 1$ in the latter, which cancel in the trace. For the third generating function, using Schur function identities, Baik and Rains \cite[Eqs.~(4.6/12)]{MR1844203} (cf. also \cite[Prop.~1.2]{MR1794352}) obtained
\begin{equation}\label{eq:groupintegrals_fdot}
f^{\boxdot}_{l}(z) = e^z \, \E_{U \in O^{-}(l+1)} e^{z \tr U}.
\end{equation}
\end{subequations}
\begin{remark} We thus get $f^\boxbslash_{l}(z) = e^{-z} f_{2l+1}^\boxdot(z)$,
or equivalently, by comparing coefficients, the combinatorial formula%
\[
\big|\{L_{n}^\boxbslash\leq 2l\}\big| = 1+ \sum_{k=1}^{2n} (-1)^k \binom{2n}{k}\, \big|\{L_{k}^\boxdot\leq 2l+1\}\big|, 
\]
which appears to be simple enough to suggest a purely combinatorial (``bijective'') proof.
\end{remark}

\subsection{Chazy I equation}\label{sect:ChazyI}
In terms of the  group integrals 
\begin{equation}\label{eq:g_pm}
g^{\pm}_l(z) := \E_{U\in O^\pm(l+1)} e^{z \tr U}
\end{equation}
we can rewrite \eqref{eq:groupintegrals_fpf} in the form 
\begin{equation}\label{eq:gen_from_g}
f_{l+1}^\boxslash(z) = (g^-_l(z) + g^+_l(z))/2, \quad f_{l}^\boxbslash(z) = g^-_{2l+1}(z), \quad f_l^\boxdot(z) = e^z g^-_l(z),
\end{equation}
supplemented (since $l$ is meant to be a positive integer) by the obvious case (cf. \eqref{eq:E4cosh} below)
\[
f_1^\boxslash(z) = \cosh(z).
\]
By relating $g_l^\pm(s)$
to the $\tau$-functions of the Toda lattice, Adler and van Moerbeke \cite[Prop.~3.4]{MR1794352} proved a representation in terms of solutions of a particular Chazy I equation: upon writing
\[
g^{\pm}_l(s) = \exp\left(\int_0^s v_l^\pm(x) \frac{dx}{x}\right),
\]
the function $v=v_l^\pm$ is the {\em unique} (analytic) solution of the third order differential equation\footnote{In fact, this equation corresponds to the particular choice $c_1=c_2=c_4=c_6=c_9=0$, $c_3=1$, $c_5=-4$, $c_7=-l^2/4$, $c_8=l^2-1$ of parameters in the full Chazy I equation as provided in \cite[Eq.~(A3)]{MR1752309}.}%
\begin{subequations}\label{eq:ChazyI}
\begin{equation}\label{eq:ChazyIeq}
v''' + \frac{1}{x} v'' + \frac{6}{x} v'^2 - \frac{4}{x^2} v v' - \frac{16x^2+l^2}{x^2} v' + \frac{16}{x}v + \frac{2(l^2-1)}{x} = 0,
\end{equation}
subject to the initial condition, as $x\to 0^+$,
\begin{equation}\label{eq:ChazyIini}
v_l^\pm(x) = x^2 \pm \frac{x^{l+1}}{l!}  + O(x^{l+2}).
\end{equation}
\end{subequations}
In fact, uniqueness follows from substituting the power series expansion $v_l(x) = \sum_{n=2}^\infty a_n x^n$
into the Chazy I equation \eqref{eq:ChazyIeq}, yielding the recursion \cite[Eq.~(3.13)]{MR1794352}, for $n=2,3,\ldots$\, ,
\begin{equation}\label{eq:ChazyIrec}
(n+1) (n^2-l^2) a_{n+1} - 16(n-2) a_{n-1} + 2\sum_{m=2}^{n-1}  m a_m \cdot (3(n-m)+1) a_{n+1-m} = 0,
\end{equation}
with the starting value $a_2$ taken from \eqref{eq:ChazyIini}. Conforming with the full form of the initial condition \eqref{eq:ChazyIini}, we note that if $l\geq 3$, the recursion implies $a_3=\cdots =a_l=0$ and, if $l\geq 2$, the coefficient $a_{l+1}$ becomes a free parameter of the recursion, to be taken from \eqref{eq:ChazyIini}.

\subsection{Compiling tables of exact values}\label{sect:exact}
Using the recursion \eqref{eq:ChazyIrec}, it is a simple exercise in computing with truncated power series expansions in a modern computer algebra system to expand the functions $g^\pm_l(s)$ and thus the generating functions as in \eqref{eq:gen_from_g}. In this fashion we compiled tables\footnote{The tables are available for download at \url{https://box-m3.ma.tum.de/f/c7e9a3c608554fc19be5/}. All three cases $\varoast=\boxslash,\boxbslash,\boxdot$ were checked against the combinatorial formulae \eqref{eq:SYT} for $n=80$ (choosing larger $n$ quickly becomes infeasible) and the case $\boxdot$ additionally, for $l=2,\ldots,5$, against recurrences in terms of Catalan numbers given in \cite[p.~559]{MR2334203}, as well as, for $n-l=0,\ldots,30$, against an explicit formula by Goulden \cite[Cor.~3.4(a)]{MR1080995}---note the restriction on $l$ for it to hold true:
\[
\big|\{L_n^\boxdot=l\}\big| =  (-1)^{n-l}\sum_{i,j \geq 0, 2i+j\leq n-l} \frac{(-1)^{i+j}n! \cdot I_j}{i!j! (n-i-j)!} \qquad (l \geq (n-1)/2).
\]} 
of the exact integer values ($l=1,\ldots,n$ and $n=1,\ldots, 1000$)
\[
|\{L_{n}^\boxslash = l\}|, \quad |\{L_{n}^\boxbslash = 2l\}|, \quad |\{L_{n}^\boxdot = l\}|.
\]
\begin{remark} Forrester and Mays \cite[§4.5]{arxiv.2205.05257} report having compiled tables for the fixed-point free cases up to $n=200$, based on truncated power series computations with representations in terms of the Okamoto $\sigma$-form of Painlevé III. However, as noted in \cite[§3.2]{arxiv.2206.09411} for similar computations relating to the general permutation case, the use of a Chazy I representation is much more cost efficient: whereas the $\sigma$-form of Painlevé III is \emph{quadratic} in the highest order derivative and \emph{cubic} in the lower orders, a Chazy I equation of the form given in \eqref{eq:ChazyIeq} is \emph{linear} in the highest order and \emph{quadratic} in the lower orders,  Hence, switching from Painlevé to Chazy reduces the complexity of evaluating the corresponding recursion significantly (by a factor of $n_{\max}$ if the table is compiled up to $n=n_{\max}$).
\end{remark}

\begin{remark}
To compare  asymptotic expansions  with numerical data of  probabilities for values of $n$ which are substantially larger than the $n\leq 1000$  covered by the tables, we  have to resort to Monte Carlo simulations. To draw  uniformly from the involutions $S_n^\boxdot$, an algorithm of optimal complexity $O(n)$  was developed by Arndt \cite[§9.3]{Arndt10}, cf. \cite[§2.5.6]{Arndt11} --  switching off the randomized branching between fixed points and $2$-cycles allows us to draw uniformly from the fixed-point free involutions $S_n^\boxtimes$, instead. The length of a longest increasing, or decreasing, subsequence  can be  computed in optimal $O(n \log n)$ comparisons by determining  just the first row of the  Young tableau in the Robinson–Schensted correspondence (or, equivalently, the top entries of the piles in patience sorting \cite{MR1694204}), see the algorithm discussed in \cite[§5.1.4, p.~52]{Knuth3}, or \cite[§2]{MR0354386} where optimality is established.\footnote{An efficient Julia implementation of those algorithms comes with the source files at  \url{https://arxiv.org/abs/2306.03798}. For example, with $n=10^5$ on a modern core, it allows us to draw elements from $S_n^\boxdot$ or $S_n^\boxtimes$ in less than half a millisecond and to  calculate the length of a longest increasing, or decreasing, subsequence in just about  one millisecond. Using $48$  cores in parallel, we can easily go for $10^8$ samples in less than an hour computing time. Using a  C++ implementation, Forrester and Mays \cite[p.~30 and Fig.~12]{arxiv.2205.05257} report simulation data in the fixed-point free cases for $n=10^4$ and $n=5\times 10^4$, with $5\times 10^6$ samples; cf. also Fig.~\ref{fig:FbetaD} below.} 
\end{remark}

\subsection{Hard-edge scaling limits}\label{sect:hard-edge}

To establish the connection between the group integrals and the hard-edge scaling limits, the following theorem generalizes a  result of Forrester and Witte \cite[Eqs.~(5.44/51)]{MR2066104} from the groups $O(2m+2)$ and $\text{Sp}(2m)=O^{-}(2m+2)$ to the general cases $O(m)$ and $O^-(m)$. If combined with the formulae in \eqref{eq:Poisson_gen} and \eqref{eq:groupintegrals_fpf}, this provides us with the starting point of our study, namely the representations stated in \eqref{eq:Poisson}.

\begin{theorem}\label{thm:E_group} For $m \in \N:=\{1,2,3,\ldots\}$ and $s\geq 0$ there holds
\begin{subequations}\label{eq:E_group}
\begin{align}
E_1^{\text{\rm hard}}\left(4s^2; \frac{m-2}{2}\right) &= e^{-s^2/2} \,\E_{U \in O^-(m)} e^{s \tr U},\label{eq:E1_group}\\*[2mm]
E_4^{\text{\rm hard}}\left(4s^2; m\right) &= e^{-s^2/2} \,\E_{U \in O(m)} e^{s \tr U}.\label{eq:E4_group}
\end{align}
\end{subequations}
Note that for $m$ odd the analytic continuation of the first function is not even. 
\end{theorem}

\begin{proof} The case $m=1$ is proved by inspection. In fact, \cite[Cor.~3.1]{MR964668} (see also \cite[p.~613]{MR2641363}, adjusting scaling)
and \cite[p.~294]{MR1266485} yield the particular evaluations
\begin{subequations}
\begin{equation}
E_1^{\text{\rm hard}}\left(4s^2; -1/2\right) = e^{-s-s^2/2}, \qquad E_2^{\text{\rm hard}}\left(4s^2; 0\right) = e^{-s^2},
\end{equation}
so that by \cite[Prop. 5.6]{MR1842786}
\begin{equation}\label{eq:E4cosh}
E_4^{\text{\rm hard}}\left(4s^2;1\right) = \frac{1}{2}\left( E_1^{\text{\rm hard}}\left(4s^2; -1/2\right) + \frac{E_2^{\text{\rm hard}}\left(4s^2; 0\right)}{E_1^{\text{\rm hard}}\left(4s^2; -1/2\right)}\right) = e^{-s^2/2} \cosh(s).
\end{equation}
\end{subequations}
Since $O^-(1) = \{-1\}$ and $O(1) = \{-1,1\}$, the group integrals evaluate to the same results.

For $l = m-1 \geq 1$ we proceed by comparing $\tau$-function representations of $E_\beta^{\text{\rm hard}}$ with those of the group integrals. Forrester and Witte \cite[Eqs.~(2.75--79, 3.18/19)]{MR1901115} obtained that
\begin{align*}
E_1^{\text{\rm hard}}\left(4s^2; \frac{l-1}{2}\right) &= e^{-s^2/2} \exp\left(\int_0^s w_l^-(x)\frac{dx}{x}\right),\\*[2mm]
E_4^{\text{\rm hard}}\left(4s^2; l+1 \right) &= \frac{e^{-s^2/2}}{2} \left(
\exp\left(\int_0^s w_l^-(x)\frac{dx}{x}\right) + \exp\left(\int_0^s w_l^+(x)\frac{dx}{x}\right)\right),
\end{align*}
where 
\[
w_l^\pm(x) = \sigma_{l}^\pm(4x) + (l-1)x - \frac{l(l-1)}{4}
\]
such that $\sigma=\sigma^\pm_l$ solves the Okamoto $\sigma$-form of Painlevé V \cite[Eq.~(2.9)]{MR1901115}, 
\begin{equation}\label{eq:Okamoto}
(t\ddot\sigma)^2 - \left(\sigma-t \dot\sigma + 2\dot\sigma^2 + (\nu_1+\nu_2+\nu_3)\dot\sigma \right)^2 + 4\dot \sigma(\nu_1+ \dot\sigma)(\nu_2+ \dot\sigma)(\nu_3+ \dot\sigma)=0,
\end{equation}
denoting derivatives in $t$ by a dot,
with parameters $\nu_1 = -1/2$, $\nu_2 = l/2$, $\nu_3=(l-1)/2$, subject to the initial condition
\[
\sigma_l^\pm(t) = \frac{l(l-1)}{4}- \frac{l-1}{4}t + \frac{t^2}{16} \pm \frac{1}{l!} \left(\frac{t}{4}\right)^{l+1} + O(t^{l+2})
\quad (t \to 0^+).
\]
In terms of $w_{l}^\pm$ we thus get, on the one hand, the initial conditions
\begin{equation}\label{eq:w_ini}
w_l^\pm(x) = x^2 \pm \frac{x^{l+1}}{l!} + O(x^{l+2}) \quad (x\to 0^+),
\end{equation}
and, on the other hand, after inserting $\sigma = \sigma_l^\pm$ in terms of $w=w_l^\pm$ into \eqref{eq:Okamoto}, writing $t=4x$ and differentiating in $x$, the differential equation 
\[
\frac{1}{8} x^2 w''(x) \left(w''' + \frac{1}{x} w'' + \frac{6}{x} w'^2 - \frac{4}{x^2} w w' - \frac{16x^2+l^2}{x^2} w' + \frac{16}{x}w + \frac{2(l^2-1)}{x}\right) = 0.
\]
Here, we recognize the term in the brackets as the Chazy I equation \eqref{eq:ChazyIeq}. Since the initial condition \eqref{eq:w_ini} implies that $x^2 w_l^{\pm\,\prime\prime}(x) \neq 0$ for small $x$, we see by continuation that $w_l^\pm$ is, in fact, a solution of this Chazy I equation. Furthermore, the initial conditions of $w_l^\pm$ and $v_l^\pm$, as given in \eqref{eq:w_ini} and \eqref{eq:ChazyIini}, are the same, so that by uniqueness of the induced solution (cf. the discussion of the result of Adler and van Moerbeke  \cite[Prop.~3.4]{MR1794352} in Sect.~\ref{sect:ChazyI}) 
\[
w_l^\pm(x) = v_l^\pm(x) = x \frac{d}{dx} \log \E_{U \in O^\pm(l+1)} e^{s \tr U},
\]
which finishes the proof.
\end{proof}
\begin{remark} 
In \cite[§3.1]{arxiv.2205.05257} it is noted that \eqref{eq:E4_group} can also be established independent of the parity of $m$ by expressing both, the expectation over $U \in O(m)$ and $E_4^{\text{\rm hard}}(4s^2;m)$, in terms of a generalized hypergeometric function of $m$ variables based on zonal polynomials.
\end{remark}

\subsection{\boldmath$H$\unboldmath-admissibility\protect\footnote{For a review of $H$-admissibility see Appendix~\ref{app:hayman}.}}

We start with an asymptotics of the integrals over $O^\pm(m)$ as $z\to\infty$ in the complex plane, 
following the proof of a similar result \cite[Thm.~2.2]{arxiv.2206.09411} for the unitary group: first, we write the integrals as multidimensional integrals over the joint eigenvalue density of $O^\pm(m)$ as established by Weyl \cite[Eqs.~(7.9.7/15)]{Weyl46} (see also \cite[Eqs.~(2.62–66)]{MR2641363}), next we apply the multidimensional Laplace method \cite[Cor.~A.1]{arxiv.2206.09411} and, finally, we evaluate a remaining multidimensional integral in terms of a variant of the Selberg integral \cite[Eq.~(2.5.10)]{MR2760897}. 

We thus get, skipping the details, for any $m\in \N$ and any fixed $0< \delta\leq \pi/2$,%
\begin{subequations}\label{eq:gen_large_z}
\begin{align}
\E_{U \in O^-(m)} e^{z \tr U}  &=  \frac{c_{m-1}\cdot e^{(m-2)z}}{\pi^{(m-1)/2}z^{(m-1)(m-2)/4}}  \big(1+ O(z^{-1})\big), \\*[2mm]
\E_{U \in O(m)} e^{z \tr U} &= \frac{c_m \cdot e^{mz}}{2^m \pi^{m/2}z^{m(m-1)/4} } \big(1+ O(z^{-1})\big),
\end{align}
\end{subequations}
uniformly as $z\to\infty$ while $|\arg z\,| \leq \frac{\pi}{2} -\delta$.\footnote{By Thm.~\ref{thm:E_group}, this shows the asymptotics of $E_\beta^{\text{\rm hard}}(4s^2;a)$ ($\beta=1,4$) given \cite[Eq.~(13.52)]{MR2641363} for $a=0,1,2,\ldots$ (scaling adjusted for $\beta=4$) to not only hold for real arguments $s$ but to continue uniformly into sectors of the complex plane.} Here, we briefly write $c_j := \prod_{k=1}^j \Gamma(k/2)$.

It follows from (\ref{eq:groupintegrals_fpf}a/b) and \eqref{eq:g_pm} that the entire functions $f^\boxslash_l(\sqrt{2z})$, $f^\boxbslash_l(\sqrt{2z})$ and $g_l^\pm(\sqrt{z})$, written as $f(z)$, 
enjoy asymptotic expansions of the form
\begin{subequations}
\begin{equation}
f(z^2) = c z^{\nu} e^{\tau z} \big(1+O(z^{-1})\big)
\end{equation}
with certain parameters $\nu$ and $\tau>0$.
Hence, Thm.~\ref{thm:born} applies and gives that the expansions of the associated auxiliary functions are, as $r\to\infty$,
\begin{equation}\label{eq:aux_gen}
a(r) = \frac{\tau}{2}\sqrt{r} + \frac{\nu}{2} + O(r^{-1/2}),\qquad b(r) = \frac{\tau}{4}\sqrt{r} + O(r^{-1/2}),
\end{equation}
\end{subequations}
combined with the following result.

\begin{theorem}\label{thm:gen_fun_zeros} The functions $f^\boxslash_l(\sqrt{2z})$, $f^\boxbslash_l(\sqrt{2z})$ and $g_l^\pm(\sqrt{z})$ are entire functions of genus zero and have, for any fixed $\epsilon< \pi/2$ at most finitely many zeros in the sector $|\arg z\,|\leq \pi/2+\epsilon$. In particular, they are $H$-admissible.
\end{theorem}

Therefore, by \eqref{eq:aux_gen}, applying Thm.~\ref{thm:criterion} to  $f_l^\boxdot(z)=e^z f(z^2)$ with $f(z)=g_l^-(\sqrt{z})$ gives:

\begin{corollary}\label{cor:f_boxdot_Hadmiss}
The entire function $f^\boxdot_l(z)$ is $H$-admissible.
\end{corollary}

\subsection{Stirling- and Regev-type formulae} As an immediate application of the $H$-ad\-mis\-si\-bi\-li\-ty of the generating functions $f^\boxslash_l(\sqrt{2z})$, $f^\boxbslash_l(\sqrt{2z})$ and $f_l^\boxdot(z)$, following the steps in our previous work \cite{arxiv.2206.09411} on the general permutation case we get Stirling-type formulae \eqref{eq:stirling} for the discrete probability distributions $p_\varoast(n;l)$.
Accurate numerical evaluations of such formulae are obtained from rewriting the generating functions in terms of the hard-edge scaling limits:
\[
\begin{aligned}
f_l^\boxslash(\sqrt{2r}) &= e^{r} E_4^\text{hard}(8r;l),\\*[1mm]
f_l^\boxbslash(\sqrt{2r}) &= e^{r} E_1^\text{hard}(8r;l),\\*[1mm]
f_l^\boxdot(r) &= e^{r+r^2/2} E_1(4r^2;(l-1)/2).
\end{aligned}
\]
We skip the details since the Stirling-type formulae are now superseded, in terms of ease of practical use and accuracy, by the asymptotic expansions of Sect.~\ref{sect:depoisson}. 

We content ourselves with stating, analogously to the discussion in \cite[§2.3]{arxiv.2206.09411}, derived from the normal approximation \eqref{eq:CLT} by using the leading terms \eqref{eq:aux_gen} of the asymptotic expansions of the auxiliary functions, the following Regev-type formulae as $n\to \infty$ for {\em fixed}~$l$:%
\begin{subequations}\label{eq:regev}
\begin{align}
\big|\{L_{n}^\boxslash \leq l\}\big| &= 2 \cdot \frac{\Gamma(1/2)\cdot \Gamma(2/2) \cdots \Gamma(l/2) \cdot l^{2n+l(l-1)/4}}{\pi^{l/2} 2^{l} (2n)^{l(l-1)/4}} (1+ o(1)),\\*[2mm]
\big|\{L_{n}^\boxbslash \leq 2l\}\big| &= 2\cdot\frac{1!\cdot 3! \cdots (2l-1)! \cdot (2l)^{2n+l^2+l/2}}{\pi^{l/2} 2^{l^2} (2n)^{l^2+l/2}} (1+ o(1)),\\*[2mm]
\big|\{L_{n}^\boxdot \leq l\}\big| &= \frac{\Gamma(1/2)\cdot \Gamma(2/2) \cdots \Gamma(l/2) \cdot l^{n+l(l-1)/4}}{\pi^{l/2} n^{l(l-1)/4}} (1+ o(1)).
\end{align}
\end{subequations}
For the same reasons as in \cite[§2.3]{arxiv.2206.09411}, these formulae are accurate for $l \ll n^{1/4}$ only,
a range which belongs well to the left tail of the discrete distributions $p_\varoast(n;l)$. Whereas the first two  appear to be new, the third one was previously obtained by Regev \cite[Eq.~(F.4.5.1)]{MR625890}, cf. also~\cite[Fn.~21, Rem.~2.4]{arxiv.2206.09411}.

\section{Expansion of the Hard-to-Soft Edge Transition}\label{sect:hard-to-soft}

In this section we prove expansions of the hard-to-soft edge transition law \eqref{eq:BF2003} in the cases $\beta=1$ and $\beta=4$. As in \cite{arxiv.2205.05257}, we start with a particularly convenient representation of the hard- and soft-edge probabilities in terms of Fredholm determinants,\footnote{To establish the hard-to-soft edge transition limit  \eqref{eq:BF2003}, Borodin and Forrester \cite[Prop.~5]{MR1986402} had used a different kernel for the $\beta=1$ case and dealt with $\beta=4$ by using formulae similar to \eqref{eq:E4cosh}.} namely \cite[Eqs.~(3.11/13) and Cor.~1]{MR2229797}
\begin{equation}\label{eq:Ebeta}
\begin{gathered}
E_\beta^{\text{hard}}(s;\nu_\beta)\big|_{\beta=1} = \det(I - V_\nu)\big|_{L^2(0,\sqrt{s})}, \quad V_\nu(x,y) = \frac12 J_\nu(\sqrt{xy}),\\*[2mm]
E_ \beta^{\text{hard}}(s;\nu_\beta)\big|_{\beta=4} = \frac{1}{2}\left(\det(I - V_\nu)\big|_{L^2(0,\sqrt{s})} + \det(I + V_\nu)\big|_{L^2(0,\sqrt{s})} \right),
\end{gathered}
\end{equation}
and \cite[Eqs.~(33/35)]{MR2165698}
\begin{equation}\label{eq:Fbeta}
\begin{gathered}
F_1(s) = \det(I - V_{\Ai})\big|_{L^2(s,\infty)}, \quad V_{\Ai}(x,y) = \frac12 \Ai\left(\frac{x+y}{2}\right),\\*[2mm]
F_4(s) = \frac{1}{2}\left(\det(I - V_{\Ai})\big|_{L^2(s,\infty)} +  \det(I + V_{\Ai})\big|_{L^2(s,\infty)}\right).
\end{gathered}
\end{equation}
Following the discussion of the case $\beta=2$ in \cite[§3]{arxiv:2301.02022}, we introduce the quantity
\begin{equation}\label{eq:hnu}
h_\nu := 2^{-1/3} \nu^{-2/3}
\end{equation}
and study expansions in powers of $h_\nu$ as $h_\nu \to 0^+$. The transform $s=\phi_\nu(t)$ used in the transition limit can briefly be written as
\begin{equation}\label{eq:phitrans}
\phi_\nu(t) = \omega_\nu(t)^2, \quad \omega_\nu(t) = \nu(1-h_\nu t).
\end{equation}

\begin{theorem}\label{thm:hard2soft} Let be $\beta = 1$ or $\beta = 4$. There holds the expansion
\begin{equation}\label{eq:hard2soft}
E_\beta^{\text{\rm hard}}(\phi_\nu(t);\nu_\beta) = F_\beta(t) + \sum_{j=1}^m E_{\beta,j}(t) h_\nu^j + h_\nu^{m+1}\cdot O(e^{-3t/4})
\end{equation}
which is uniformly valid when $t_0\leq t < h_\nu^{-1}$ as $h_\nu\to0^+$, $m$ being any fixed non-negative integer and $t_0$ any fixed real number. Preserving uniformity, the expansion can be repeatedly differentiated w.r.t. the variable $t$.
The $E_{\beta,j}$ are certain smooth functions; the first three are
\begin{subequations}\label{eq:F_beta_j}
\begin{align}
E_{\beta,1}(t) &= \frac{3t^2}{10} F_\beta'(t) - \frac{2}{5} F_\beta''(t)\label{eq:F1}\\
\intertext{and, subject to the linear form hypothesis below,}
E_{\beta,2}(t) &= \Big(\frac{9}{175} + \frac{32t^3}{175}\Big) F_\beta'(t) + \Big(-\frac{32t}{175} + \frac{9t^4}{200}\Big) F_\beta''(t) - \frac{3t^2}{25} F_\beta'''(t) + \frac{2}{25} F_\beta^{(4)}(t),\label{eq:F2}\\*[2mm]
E_{\beta,3}(t) &=  \Big(\frac{268 t}{7875} + \frac{1037 t^4}{7875}\Big)  F_\beta'(t)
+\Big( -\frac{33t^2}{350}  + \frac{48t^5}{875} \Big) F_\beta''(t)  \\*[1mm]
&\quad +\Big(-\frac{578}{7875} -\frac{16 t^3}{125} + \frac{9 t^6}{2000}\Big) F_\beta'''(t)
 +\Big(\frac{64t}{875} -\frac{9 t^4}{500} \Big)F_\beta^{(4)}(t)\notag\\*[1mm]
&\quad +\frac{3 t^2 }{125}F_\beta^{(5)}(t)
-\frac{4}{375} F_\beta^{(6)}(t).\notag
\end{align}
\end{subequations}
\end{theorem}

\begin{remark}
A similar (but unconditional) result \cite[Thm.~3.1]{arxiv:2301.02022} holds for $\beta = 2$. It turns out that the polynomial coefficients of the cases $\beta=1,4$ and of $\beta=2$, though generally different, share a common sparsity pattern. Algebraic relations between them are given in \eqref{eq:hyposol}.
\end{remark}

\begin{remark}\label{rem:thinning}  Thm.~\ref{thm:hard2soft} and its $\beta=2$ counterpart \cite[Thm.~3.1]{arxiv:2301.02022} admit some  generalizations:
\begin{itemize}\itemsep=5pt
\item  
In a forthcoming study, we will address the hard-to-soft edge transition with a thinning rate as in \cite[§3.3]{MR3647807}. In fact, the expansion terms keep being linear combinations of higher order derivatives of the respective limit laws with rational polynomial coefficients -- which are, in addition to it, {\em independent} of the thinning rate. For $\beta=1,4$, the rationale of such an independence result is sketched in Appendix~\ref{app:thinning}, with supporting numerical evidence provided in Appendix \ref{app:num_evidence}.

\item In \cite{arXiv:2403.07628}, we have established a similar structure for the large-matrix limits of the $\beta=1,2,4$ Gaussian and Laguerre (Wishart) ensembles at the soft edge.
\end{itemize}
\end{remark}

\begin{figure}[tbp]
\includegraphics[width=0.325\textwidth]{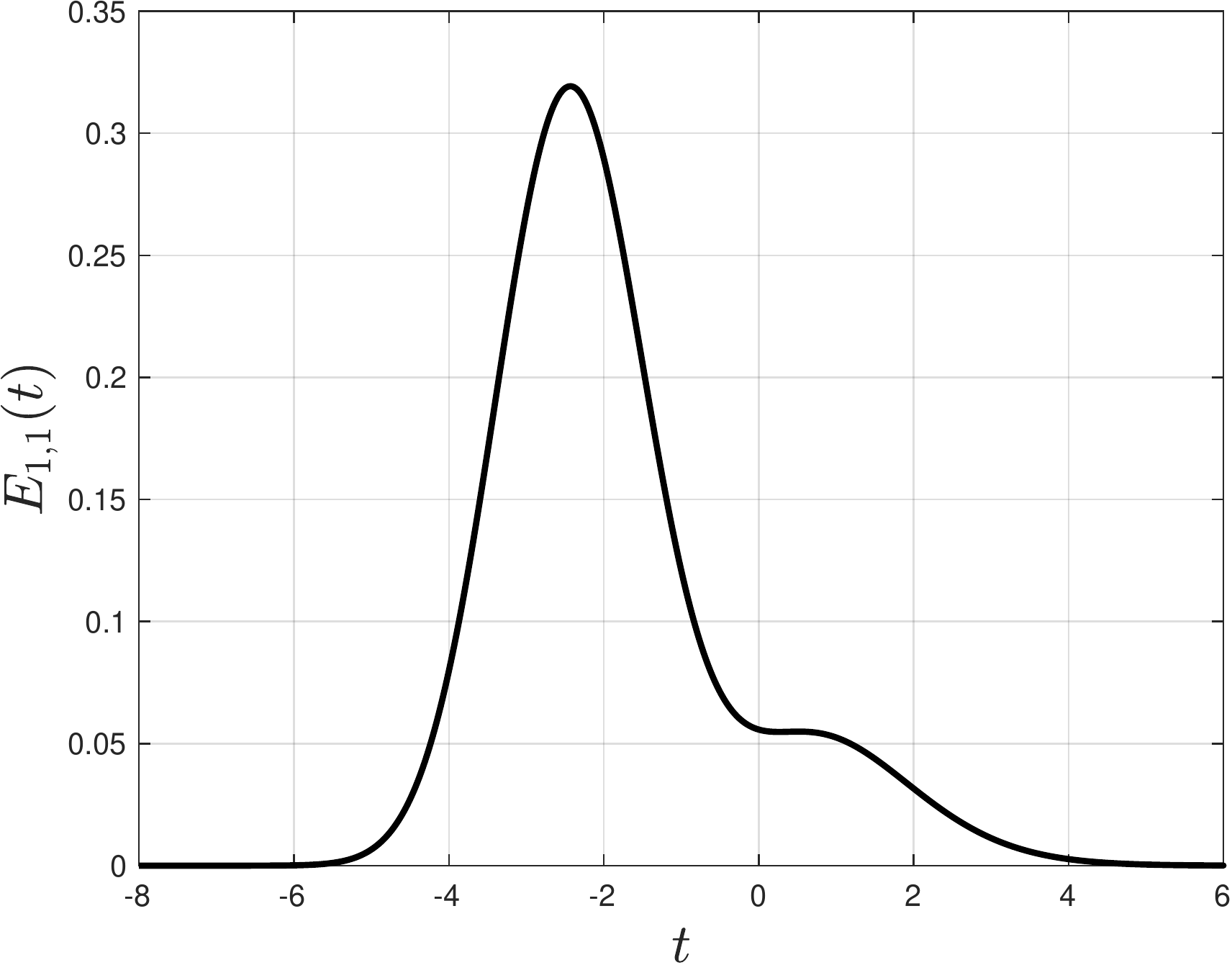}\hfil\,
\includegraphics[width=0.325\textwidth]{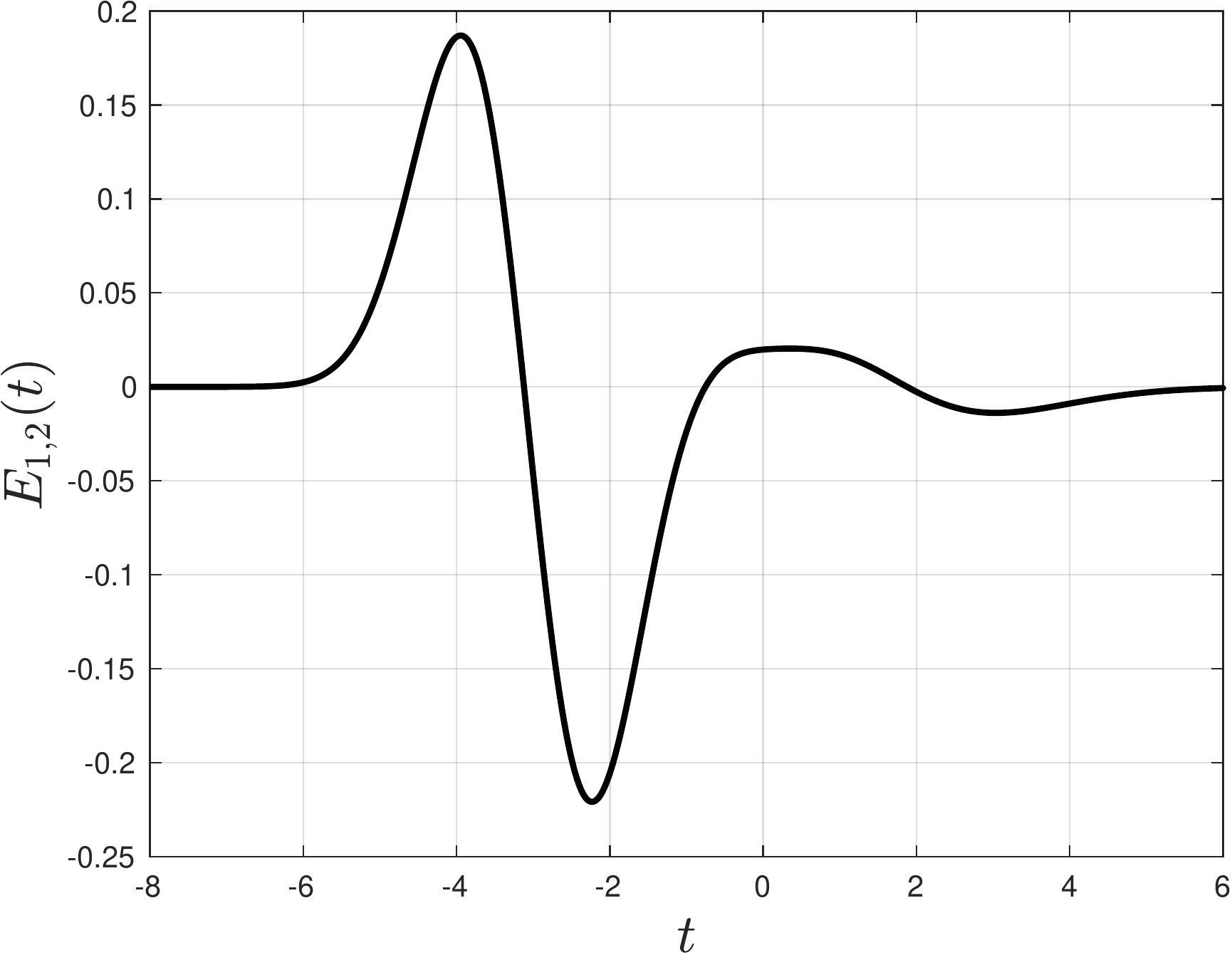}\hfil\,
\includegraphics[width=0.325\textwidth]{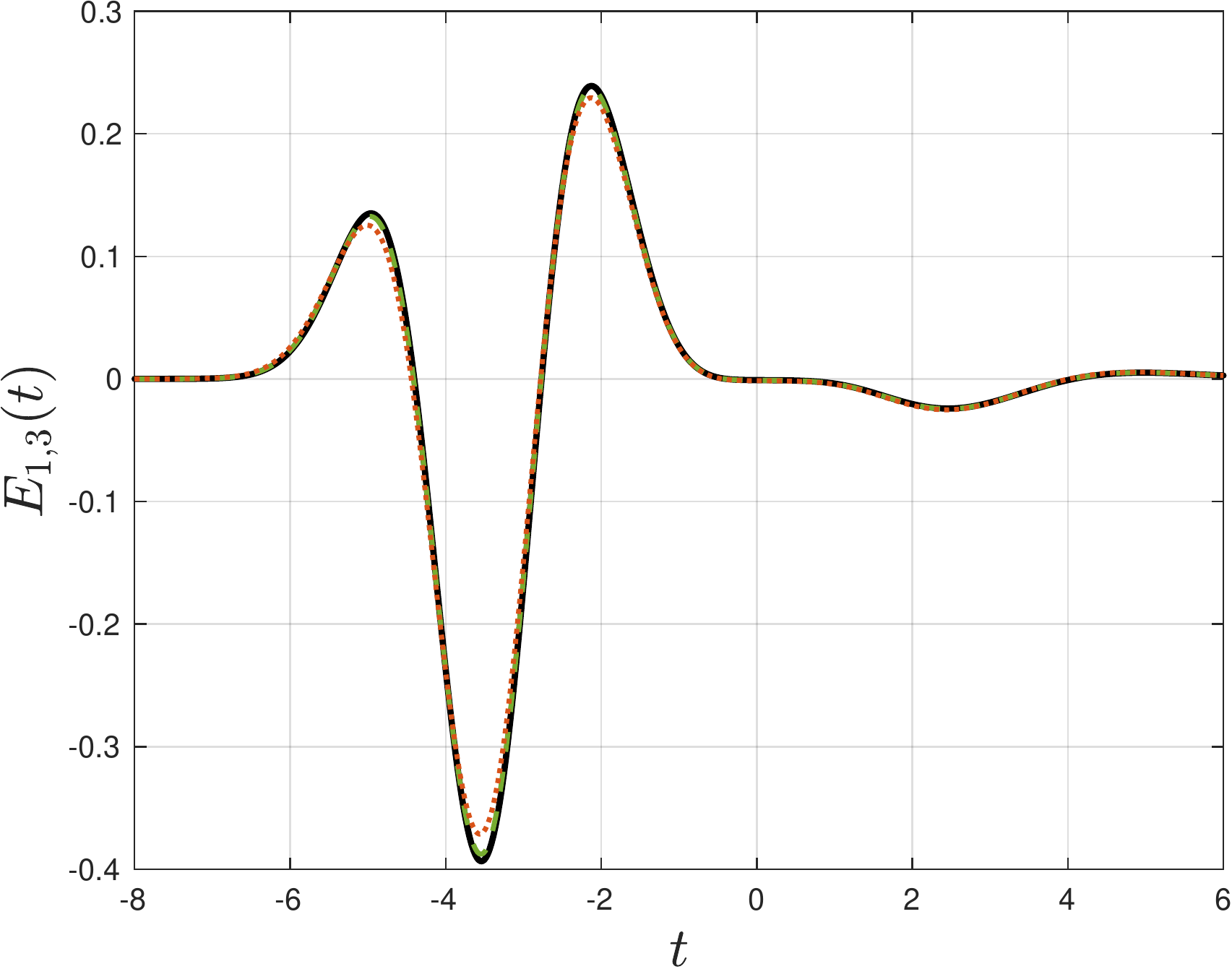}\\
\includegraphics[width=0.325\textwidth]{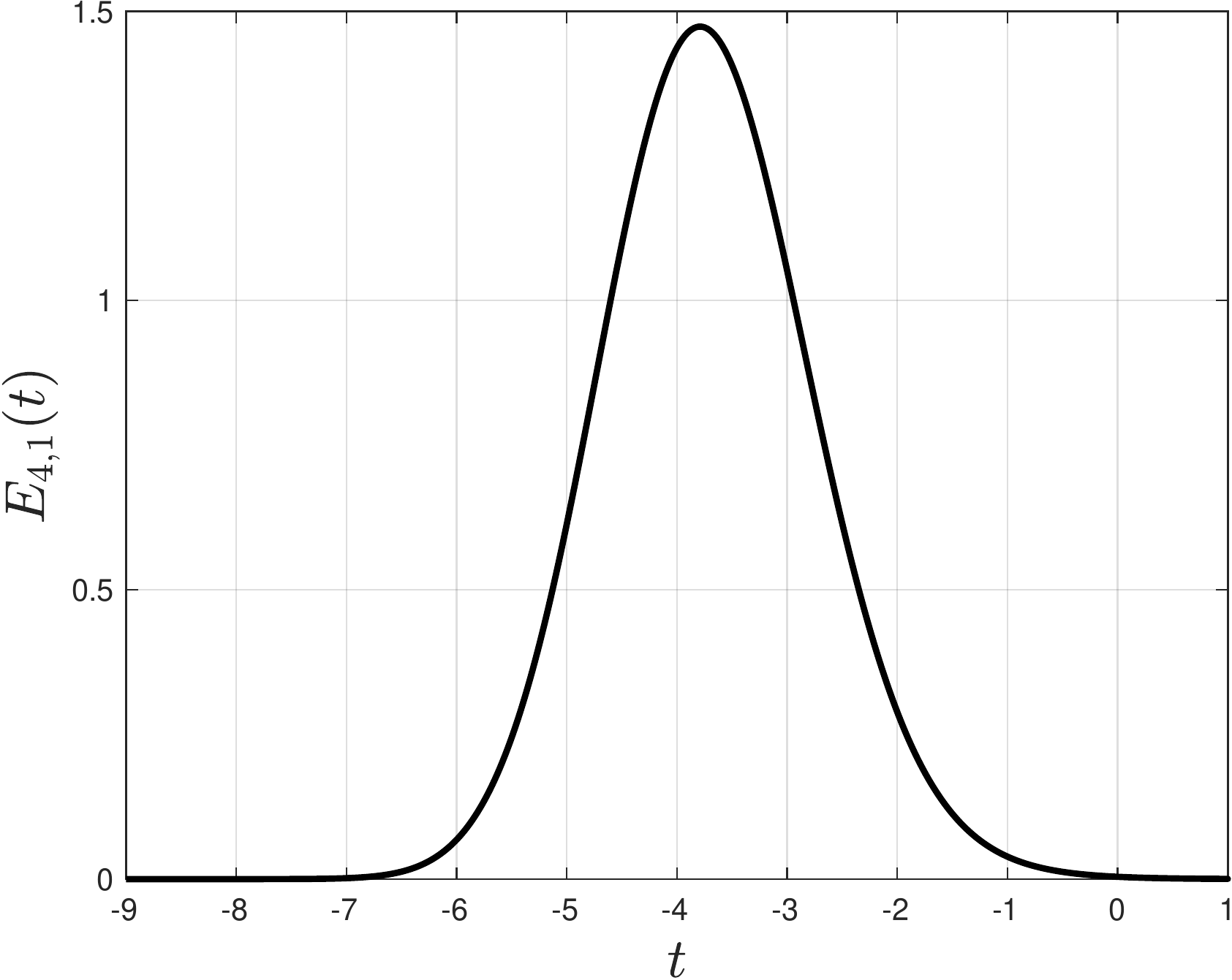}\hfil\,
\includegraphics[width=0.325\textwidth]{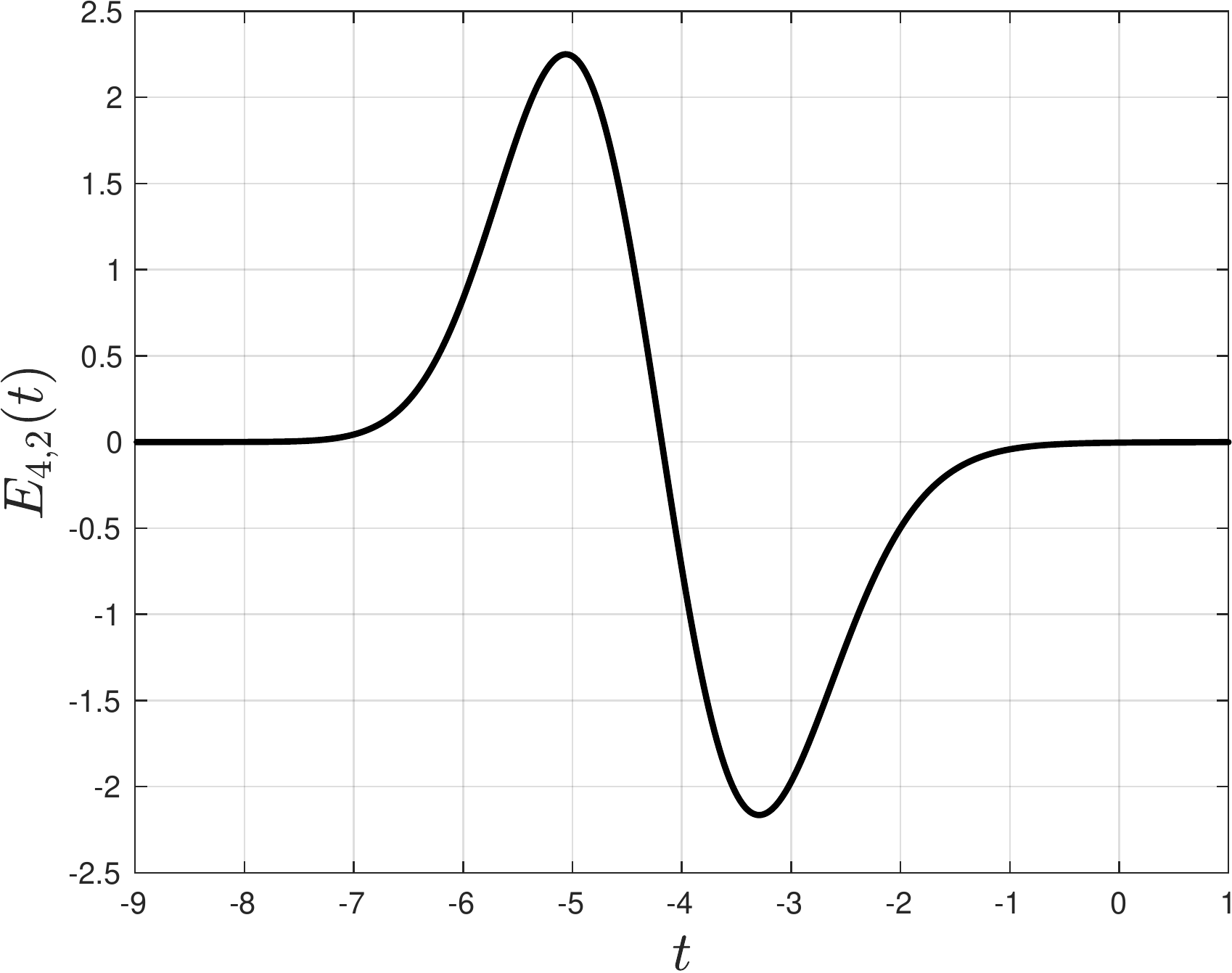}\hfil\,
\includegraphics[width=0.325\textwidth]{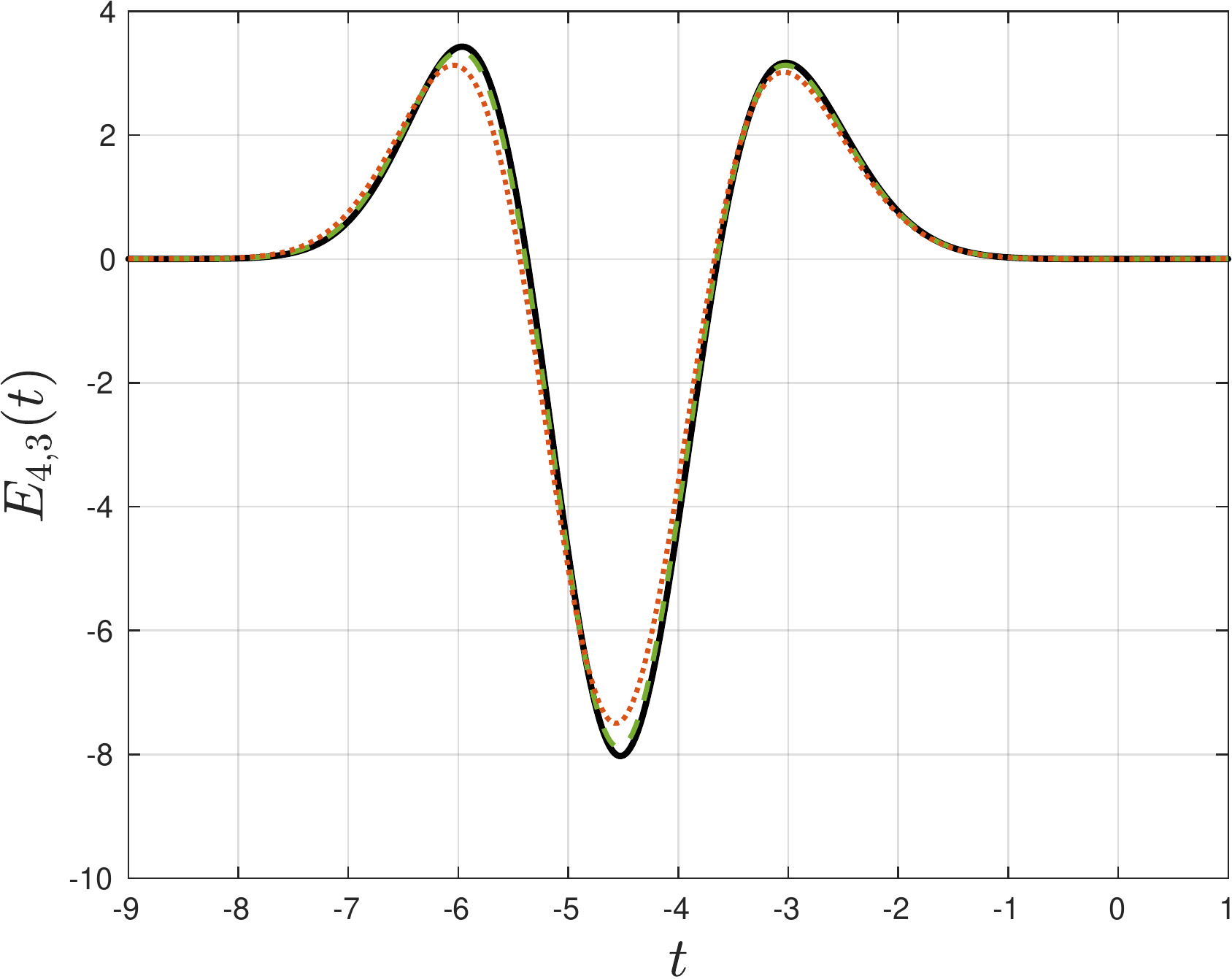}
\caption{{\footnotesize Top row $\beta=1$; bottom row $\beta=4$. Plots of $E_{\beta,1}(t)$ (left panels) and 
$E_{\beta,2}(t)$ (middle panels) as in \eqref{eq:F_beta_j}.
The right panels show $E_{\beta,3}(t)$ as in \eqref{eq:F_beta_j} (black solid line) with the approximation \eqref{eq:Fbeta3} for $\nu=100$ (red dotted line) and $\nu=800$ (green dashed line): the close agreement validates the functional forms given in \eqref{eq:F_beta_j}. Details about the numerical method can be found in \cite{MR2895091, MR2600548,arxiv.2206.09411,MR3647807}.}}
\label{fig:hard2soft}
\end{figure}

The intricate formulae in \eqref{eq:F_beta_j} (based on the linear form hypothesis) can be validated numerically: Fig.~\ref{fig:hard2soft} plots the functions $E_{\beta,j}(t)$ ($j=1,2,3$) next to the approximation
\begin{equation}\label{eq:Fbeta3}
E_{\beta,3}(t)\approx h_\nu^{-3}\cdot \left(E_\beta^\text{hard}(\phi_\nu(t);\nu_\beta) - F_\beta(t) - E_{\beta,1}(t)h_\nu- E_{\beta,2}(t)h_\nu^2 \right)
\end{equation}
for $\nu=100$ and $\nu=800$; the close matching is a strong testament of the correctness of \eqref{eq:F_beta_j}.

\medskip

The proof of Thm.~\ref{thm:hard2soft} is split into several steps and will be concluded in Sects.~\ref{sect:generalform} and \ref{sect:functionalform}.

\subsection{Kernel expansions} We start with an auxiliary result.

\begin{lemma}\label{lemma:Psi} Define for $h>0$ and $x,y < h^{-1}$ the function
\begin{equation}\label{eq:Psi}
\Psi(x,y;h) := h^{-1} \left(1- \sqrt{(1-h x)(1-h y)}\,\right).
\end{equation}
It satisfies the bound
\begin{equation}\label{eq:PsiBound}
\frac{x+y}{2} \leq \Psi(x,y;h) < h^{-1}
\end{equation}
and has a convergent power series expansion of the form
\begin{subequations}\label{eq:PsiExpan}
\begin{equation}
\Psi(x,y;h) = \frac{x+y}{2} + (x-y)^2 \sum_{k=1}^\infty r_k(x,y) h^k,
\end{equation}
the $r_k(x,y)$ being homogeneous symmetric rational polynomials of degree $k-1$, starting with
\begin{equation}
r_1(x,y) = \frac{1}{8},\quad r_2(x,y) = \frac{1}{16}(x+y), \quad r_3(x,y) = \frac{1}{128}\big(5(x^2+y^2)+6xy\big).
\end{equation}
\end{subequations}
The series converges uniformly for $x,y<(1-\delta)h^{-1}$, $\delta$ being any fixed real positive number.
\end{lemma}
\begin{proof} We write
\[
\Psi(x,y;h) = h^{-1} \left(1- \Phi(x,y;h) \left(1- h \frac{x+y}{2}\right)\right),
\]
where, by \cite[Lemma~3.2]{arxiv:2301.02022}, the term $\Phi(x,y;h)$ satisfies the bound $0 < \Phi(x,y;h) \leq 1$ and has a convergent power series expansion of the form
\[
\Phi(x,y;h) = 1 - (x-y)^2 \sum_{k=2}^\infty \hat r_k(x,y) h^k,
\]
the $\hat r_k(x,y)$ being homogeneous symmetric rational polynomials of degree $k-2$ starting with
\[
\hat r_2(x,y) = \frac{1}{8},\quad \hat r_3(x,y) = \frac{1}{8}(x+y), \quad \hat r_4(x,y) = \frac{1}{128}\big(13(x^2+y^2)+22xy\big).
\]
The lemma follows from observing $r_k(x,y) = \hat r_{k+1}(x,y) - \frac{1}{2}(x+y) \hat r_k(x,y)$.
\end{proof}

\begin{lemma}\label{lem:Vkernexpan} The change of variables $s=\omega_\nu(t)$, mapping $t< h_\nu^{-1}$ monotonically decreasing to $s> 0$, induces the symmetrically transformed kernel  
\begin{subequations}\label{eq:Vkernexpan}
\begin{align}
\hat V_\nu(x,y) &:=\sqrt{\omega_\nu'(x) \omega_\nu'(y)} \, V_\nu(\omega_\nu(x), \omega_\nu(y)). \\
\intertext{There holds the kernel expansion}
\hat V_\nu(x,y) &= V_{\Ai}(x,y) + \sum_{j=1}^m K_j(x,y) h_\nu^j + h_\nu^{m+1}\cdot O\big(e^{-(x+y)/2}\big),
\end{align}
\end{subequations}
which is uniformly valid when $t_0 \leq x,y < h_\nu^{-1}$ as $h_\nu\to 0^+$, $m$ being any fixed non-negative integer and $t_0$ any fixed real number. Here the $K_j$, $j=1,2,\ldots$, are kernels of the form
\begin{equation}\label{eq:KjForm}
K_j(x,y) = p_j(x,y) \Ai\Big(\frac{x+y}{2}\Big)+q_j(x,y) \Ai'\Big(\frac{x+y}{2}\Big)
\end{equation}
where $p_{j}(x,y)$ and $q_{j}(x,y)$ are certain symmetric rational polynomials. If written in terms of the (scaled) elementary symmetric polynomials
\[
u = \frac{x+y}{2},\qquad v = xy,
\] 
the first three kernels are
\begin{subequations}\label{eq:K}
\begin{align}
K_1(x,y) &= \frac{u}{10}\Ai(u) + \left(\frac{2u^2}{5}-\frac{v}{4}\right)\Ai'(u),\label{eq:K1}\\*[1mm]
K_2(x,y) &= \left( \frac{4u^5}{25}-\frac{u^3v}{5}+\frac{u v^2}{16}+\frac{13u^2}{140}-\frac{v}{20}\right) \Ai(u) 
   +\left(\frac{4u^3}{7}-\frac{9uv}{20}+\frac{1}{70}\right) \Ai'(u),\label{eq:K2}\\*[1mm]
K_3(x,y) &= \left(\frac{1228 u^6}{2625}-\frac{24 u^4
   v}{35}+\frac{803 u^3}{6300}+\frac{21 u^2
   v^2}{80}-\frac{v^3}{96}-\frac{u v}{10}-\frac{1}{225}\right)\Ai(u)\label{eq:K3}\\*[1mm]
&    +\left(\frac{16 u^7}{375}-\frac{2 u^5
   v}{25}+\frac{u^3v^2}{20}+\frac{197 u^4}{225}-\frac{uv^3}{96}-\frac{253 u^2 v}{280}+\frac{v^2}{8}+\frac{37 u}{3150}\right) \Ai'(u).\notag
\end{align}
\end{subequations}
Preserving uniformity, the kernel expansion \eqref{eq:Vkernexpan} can repeatedly be differentiated w.r.t. $x$, $y$.
\end{lemma}

\begin{proof} Upon using the function $\Psi(x,y;h)$ as defined in \eqref{eq:Psi} we write
\[
2\hat V_\nu(x,y) =\frac{1}{\sqrt{2h_\nu}} J_\nu(\omega_\nu(t)), \quad t = \Psi(x,y;h_\nu).
\]
By \cite[Eq.~(3.12)]{arxiv:2301.02022} we have 
\begin{equation}\label{eq:Olver52}
\frac{1}{\sqrt{2h}} J_\nu(\omega_\nu(t)) = \Ai(t) \sum_{k=0}^m p^*_k(t)h_\nu^k + \Ai'(t) \sum_{k=0}^m q^*_k(t)h_\nu^k + h_\nu^{m+1} O(e^{-t}),
\end{equation}
uniformly for $t_0\leq t < h_\nu^{-1}$ as $h_\nu\to 0^+$. Here, the $p_k^*$ and $q_k^*$ are certain rational polynomials (cf.~\cite[Remark A.8]{arxiv:2301.02022} and \cite[Eqs.~(10.19.10/11)]{MR2723248}) starting with
\[
p_0^*(t) = 1, \quad p_1^*(t) = \frac{t}{5},\quad p_2^*(t) = \frac{9 t^5}{200}+\frac{3 t^2}{35},\quad p_3^*(t) = \frac{957 t^6}{14000}+\frac{173
   t^3}{3150}-\frac{2}{225},
\]
and
\[
q_0^*(t) = 0, \quad q_1^*(t) = \frac{3 t^2}{10},\quad q_2^*(t) =\frac{17 t^3}{70}+\frac{1}{35}, \quad q_3^*(t) = \frac{9 t^7}{2000}+\frac{611 t^4}{3150}+\frac{37t}{1575}.
\]
First, if $x$ or $y$ is between $\tfrac{3}{4}\cdot h_\nu^{-1}$ and $h_\nu^{-1}$, then using the bound \eqref{eq:PsiBound}  we see that
\[
h_\nu^{-1} > t \geq \frac{x+y}{2} \geq \frac{3h_\nu^{-1}}{8}  + \frac{t_0}{2},
\]
so that, by the superexponential decay of the Airy function and its derivative, $\hat V_\nu(x,y)$ and all of its expansion terms in \eqref{eq:Vkernexpan} get completely absorbed into the error term
\[
h_\nu^{m+1} \cdot O(e^{-(x+y)/2}).
\]
Second, if $t_0 \leq x,y \leq \tfrac{3}{4}\cdot h_\nu^{-1}$, then, by Lemma~\ref{lemma:Psi}, the power series expansion \eqref{eq:PsiExpan} converges uniformly; written in terms of the scaled elementary symmetric polynomials $u$ and $v$ we  get
\[
t = u + \sum_{k=1}^\infty r_k^*(u,v) h_\nu^k
\]
with certain rational polynomials $r_k^*$ of degree $k+1$ starting with
\[
r_1^*(u,v)= \frac{1}{2} \big(u^2-v \big),\quad r_2^*(u,v)=\frac{1}{2}\big(u^3-uv \big),\quad r_3^*(u,v) = \frac{1}{8} \big(5u^4-6u^2v+v^2 \big).
\]
Inserting the power series for $t$ into \eqref{eq:Olver52}, Taylor expanding and using the Airy differential equation $\Ai''(\xi) = \xi \Ai(\xi)$ yields the asserted form of the expansion \eqref{eq:Vkernexpan}, its uniformity and its differentiability (as well as the first three concrete terms).
\end{proof}

\subsection{Proof of the general form of the expansion}\label{sect:generalform} 

Lemma~\ref{lem:Vkernexpan} and Thm.~\ref{thm:detexpan} yield
\begin{equation}\label{eq:Gexpan}
\begin{aligned}
E_\pm(t;\nu)&:=\det(I \mp V_\nu)\big|_{L^2(0,\omega_\nu(t))} = \det(I \mp \hat V_\nu)\big|_{L^2(t,h_\nu^{-1})}\\*[1mm] 
&= F_\pm(t) + \sum_{j=1}^m E_{\pm,j}(t) h_\nu^j + h_\nu^{m+1}O(e^{-t}) + e^{-h_\nu^{-1}/2} O(e^{-t/2}),
\end{aligned}
\end{equation}
uniformly for $t_0\leq t < h_\nu^{-1}$ as $h_\nu\to 0^+$; preserving uniformity, this expansion can be repeatedly differentiated w.r.t. the variable $t$. Thm.~\ref{thm:detexpan} gives
\begin{equation}\label{eq:Gz}
F_\pm(t) = \det(I \mp V_{\Ai})\big|_{L^2(t,\infty)}
\end{equation}
and the $E_{\pm,j}(t)$ as products of $F_\pm(t)$ with certain operator traces (a form which is unwieldy to handle; much easier and more appealing expressions are given in Sect.~\ref{sect:functionalform}). Observing
\[
e^{-h_\nu^{-1}/2} < e^{-h_\nu^{-1}/4}e^{-t/4} = h_\nu^{m+1} O(e^{-t/4})\qquad (h_\nu\to 0^+)
\]
we can combine the two error terms as $h_\nu^{m+1} O(e^{-3t/4})$. Using the determinantal representations of $E_\beta^\text{hard}(\phi_\nu(t);\nu_\beta)$ given in \eqref{eq:Ebeta} and of $F_\beta(t)$ in \eqref{eq:Fbeta}, we finally obtain the general form  of the expansion \eqref{eq:hard2soft} with
\begin{subequations}\label{eq:FfromG}
\begin{gather}
F_1(t) = F_{+}(t),\qquad F_4(t) = \frac{1}{2}\big( F_{+}(t) + F_{-}(t)\big),\\*[1mm]
E_{1,j}(t) = E_{+,j}(t),\qquad E_{4,j}(t) = \frac{1}{2}\big( E_{+,j}(t) + E_{-,j}(t)\big).
\end{gather}
\end{subequations}

\subsection{Functional form of the \boldmath $E_{\beta,j}(t)$\unboldmath}\label{sect:functionalform} 

\subsubsection{Factorization of the $\beta=2$ hard- and soft-edge distributions}

For the case $\beta=2$ we discovered in \cite[Thm.~3.1]{arxiv:2301.02022} that the corresponding hard-to-soft edge transition expansion
\begin{equation}\label{eq:hard2soft2}
E_2^{\text{\rm hard}}(\phi_\nu(t);\nu) = F_2(t) + \sum_{j=1}^m E_{2,j}(t) h_\nu^j + h_\nu^{m+1}\cdot O(e^{-3t/2})
\end{equation}
enjoys a remarkable structure: the expansion terms $E_{2,j}$ are linear combinations of higher order derivatives of the limit distribution $F_2$ with rational polynomial coefficients,\footnote{This structure was explicitly established for $1\leq j\leq 10$: see \cite[Eqs.~(3.4a--c)]{arxiv:2301.02022}  and the supplementary Mathematica notebook coming with the sources of \cite{arxiv:2301.02022}. We conjecture it to be generally true.}
\begin{equation}\label{eq:E2j_Structure}
E_{2,j} = \sum_{k=1}^{2j} p_{2,jk} F_2^{(k)},\quad p_{2,jk} \in \Q[t],
\end{equation}
starting with the first order term\footnote{Formula \eqref{eq:E21} was proved in \cite[Eq.~(3.4a)]{arxiv:2301.02022} 
without harnessing the power of the Tracy--Widom theory. Other proofs, which are based on that theory, can be found in  \cite[Appendix B]{arxiv:2301.02022} or, by inference from the Poissonized length distributions, in \cite[Thm.~1.3]{MR3161478} and \cite[Eq.~(2.29)]{arxiv.2205.05257}.}
\begin{equation}\label{eq:E21}
E_{2,1}(t) = \frac{3t^2}{10}F_2'(t) - \frac{1}{5}F''_2(t).
\end{equation}

Now, utilizing the factorizations (see \cite[Eqs.~(3.12) and Cor.~1]{MR2229797} and \cite[Eq.~(34)]{MR2165698})
\begin{equation}\label{eq:detfactor}
E_2^\text{hard}(\phi_\nu(t);\nu) = E_+(t;\nu)E_-(t;\nu),\qquad F_2(t) = F_+(t)F_-(t)
\end{equation}
will help us to understand the functional form of the $E_{\pm,j}$ in relation to the $E_{2,j}$. To begin with, if we insert  the expansions \eqref{eq:Gexpan} and  \eqref{eq:hard2soft2} into  \eqref{eq:detfactor} and equate the coefficients of $h^m$, we get after a devision by the strictly positive $F_2 = F_+ F_-$ that
\begin{equation}\label{eq:EpmRec}
\frac{E_{+,m}}{F_+} + \frac{E_{-,m}}{F_-} = \frac{E_{2,m}}{F_2} - \sum_{k,l\geq 1:\, k+l=m} \frac{E_{+,k}}{F_+} \cdot \frac{E_{-,l}}{F_-}.
\end{equation}

\subsubsection{Tracy--Widom theory as an algebraic tool}\label{sect:TracyWidom} Tracy--Widom theory \cite{MR1257246, MR1266485, MR1385083} provides us with the means to use formulae such as
\eqref{eq:EpmRec} to validate, or even establish, specific functional forms. Borrowing from the set of ideas that we developed in \cite[Appendix B]{arxiv:2301.02022} to establish the functional form of the first few concrete cases of the $E_{2,j}$, we reformulate that theory in a more algebraic fashion: first, \cite[Eqs.~(1.15)]{MR1257246} and \cite[Eq.~(52--54)]{MR1385083}) give
\begin{equation}\label{eq:TWtheory}
\frac{F_2'(t)}{F_2(t)} = q'(t)^2 - t q(t)^2 - q(t)^4, \qquad F_\pm(t) = e^{\mp\frac12\int_{t}^\infty q(x)\,dx} \sqrt{F_2(t)}
\end{equation}
in terms of the Hastings--McLeod solution $q(t)\sim \Ai(t)$ $(t\to\infty)$ of the Painlevé II equation
\begin{equation}\label{eq:PII}
q''(t) = t q(t) + 2 q(t)^3.
\end{equation}
Next, from these formulae we obtain inductively that (with explicit expressions)\footnote{For instance, there is
\[
\frac{F_\pm'(t)}{F_\pm(t)} = \frac12\big(q'(t)^2 \pm q(t) -t q(t)^2 -q(t)^4 \big).
\]
~\\*[-0.4cm]}
\[
q^{(n)},\,\frac{F^{(n)}_2}{F_2},\, \frac{F_\pm^{(n)}}{F_\pm} \in \Q[t][q,q'],
\]
where $\Q[t][q,q']$ denotes the space, closed under differentiation, of polynomials in $q$ and $q'$ with coefficients in $\Q[t]$. Moreover, since by \cite[Thms.~20.2/21.1]{MR1960811} the non-zero solutions of \eqref{eq:PII} do not satisfy an algebraic differential equation of first order, the functions $q, q'$ are algebraically independent over the ring $\Q[t]$. Therefore, all terms in the space $\Q[t][q,q']$ are {\em uniquely} expanded to $\Q[t]$-linear combinations of $q$-monomials.\footnote{That is, terms of the form
$q(t)^{\alpha_0}q'(t)^{\alpha_1}$ with $\alpha_0,\alpha_1 \in \N_0$.}

Now, for concretely given terms $T_0,\ldots,T_n \in \Q[t][q,q']$ we can validate (as in Sect.~\ref{sect:firstorder}) a linear equation of the form
\begin{equation}\label{eq:Tequation}
T_0 = p_1 T_1 + \cdots + p_n T_n, 
\end{equation}
for given coefficients $p_1,\ldots, p_n \in \Q[t]$, by a simple term-wise comparison of the resulting expansion in $\Q[t][q,q']$. On the other hand, if the coefficients are yet to be determined (as in Sect.~\ref{sect:higherorder}), comparing the $\Q[t]$-coefficients of all $q$-monomials in \eqref{eq:Tequation} forces 
\[
(p_1,\ldots,p_n)^T \in \Q[t]^n
\]
to be the solution of a $N\times n$ linear system of equations over the ring $\Q[t]$, where $N$ denotes the number of different $q$-monomials that appear in the representations of $T_0,\ldots,T_n$. Explicit examples of such $\Q[t]$-linear systems can be found in \eqref{eq:exampleSystem} below and in \cite[Appendix B]{arxiv:2301.02022}.

For applying this methodology to \eqref{eq:EpmRec} we note that by \eqref{eq:E2j_Structure} there is 
\begin{equation}\label{eq:E2Qstructure}
E_{2,j}/F_2 \in \Q[t][q,q'],
\end{equation}
at least for the $j$ up to $10$ for which we cared to explicitly validate \eqref{eq:E2j_Structure} in the supplementary material of \cite{arxiv:2301.02022}. However, to extract new formulae from \eqref{eq:EpmRec}, we have to prove (as in Sect.~\ref{sect:firstorder}), or  assume (as in Sect.~\ref{sect:higherorder}), that a similar structure holds for  the $E_{\pm,j}/F_\pm$.

\subsubsection{The first order terms}\label{sect:firstorder}

Here, we start with \cite[Thm.~1.2]{MR3161478} of Baik and Jenkins, which provides us with formula \eqref{eq:BaikJenkinsBeta1} for the term $F_{1,1}(t)$ in the expansion \eqref{eq:PvaroastIntro}. By the simple relation  between $E_{1,1}$ and $F_{1,1}$ in \eqref{eq:Ebeta1Fbeta1}, and by the relation between $\beta$-terms and $\pm$-terms in~\eqref{eq:FfromG}, we thus get
\begin{subequations}\label{eq:EpmOne}
\begin{equation}\label{eq:EplusOne}
E_{+,1}(t) =  \frac{3t^2}{10} F_+'(t) - \frac{2}{5} F_+''(t).
\end{equation}
Hence, $E_{+,1}/F_+ \in \Q[t][q,q']$ so that by \eqref{eq:EpmRec} and \eqref{eq:E2Qstructure} there is also
\[
\frac{E_{-,1}}{F_-} = \frac{E_{2,1}}{F_2} - \frac{E_{+,1}}{F_+} \in \Q[t][q,q'].
\]
Now, this gives proof of the formula\footnote{See \eqref{eq:plusMinusSame} for the rationale of choosing the same $\Q[t]$-coefficients here as in \eqref{eq:EplusOne}.}
\[
\frac{E_{-,1}(t)}{F_-(t)}  = \frac{3t^2}{10} \frac{F_-'(t)}{F_-(t)} - \frac{2}{5} \frac{F_-'(t)}{F_-(t)}
\]
since a routine calculation exhibits that both sides expand in $\Q[t][q,q']$ to 
\begin{multline*}
\frac{1}{5} q(t)^4 q'(t)^2+\frac{t}{5} q(t)^2 q'(t)^2-\frac{1}{5} q(t) q'(t)^2 -\frac{1}{10} q'(t)^4 + \frac{3t^2}{20}  q'(t)^2 -\frac{1}{5}q'(t)\\*[1mm]
-\frac{1}{10} q(t)^8-\frac{ t }{5}q(t)^6+\frac{1}{5}q(t)^5-\frac{t^2}{4}  q(t)^4+\frac{t}{5} q(t)^3+\left(\frac{1}{10}-\frac{3t^3}{20}\right) q(t)^2+\frac{3t^2}{20}  q(t).
\end{multline*}
Hence, multiplying with $F_-$, we have proof for the following counterpart to \eqref{eq:EplusOne}:
\begin{equation}\label{eq:EminusOne}
E_{-,1}(t) =  \frac{3t^2}{10} F_-'(t) - \frac{2}{5} F_-''(t).
\end{equation}
\end{subequations}
By the linearity of \eqref{eq:FfromG}, and the fact that the $E_{\pm,1}(t)$ share the same $\Q[t]$-coefficients, we get \begin{equation}\label{eq:F1proof}
E_{\beta,1}(t) = \frac{3t^2}{10} F_\beta'(t) - \frac{2}{5} F_\beta''(t) \qquad(\beta=1,4),
\end{equation}
which finishes the proof of \eqref{eq:F1}. Using \eqref{eq:Ebeta1Fbeta1} once more, we infer the formula \eqref{eq:BaikJenkinsBeta4} for $F_{\beta,4}$.

\subsubsection{The higher order terms}\label{sect:higherorder}

Lacking a general method that would allow us to establish, not only for the first order term, but also for the higher order  terms (with explicit expressions)
\[
\frac{E_{\pm,j}}{F_\pm} \in \Q[t][q,q'],
\]
we have to resort to some reasonable assumptions. In view of the structure exposed for the higher order terms of the $\beta=2$ case in \eqref{eq:E2j_Structure}, and for the first-order term of the $\pm$-cases in~\eqref{eq:EpmOne}, as well as in view of the compelling numerical evidence which is detailed in Appendix~\ref{app:evidence}, we are led to the following hypothesis:

\subsubsection*{\bf Linear form hypothesis.} Not only $E_{\pm,1}$ but all the  $E_{\pm,j}$ can be represented in the form
\begin{equation}\label{eq:hypoL}
E_{\pm,j} = \sum_{k=1}^{2j} p_{\pm,jk} F_\pm^{(k)},\quad p_{\pm,jk} \in \Q[t].
\end{equation}

Based on that hypothesis, the relation \eqref{eq:EpmRec} between the $E_{2,j}$ and the $E_{\pm,j}$ can be recast as 
\begin{equation}\label{eq:linear1}
\sum_{\sigma=\pm}\sum_{\mu=1}^{2m} p_{\sigma,m\mu} \frac{F_\sigma^{(\mu)}}{F_\sigma} = \sum_{\mu=1}^{2m} p_{2,m\mu} \frac{F_2^{(\mu)}}{F_2}
-  \sum_{\substack{k,l\geq 1\\*[0.25mm] k+l=m}} \sum_{\kappa=1}^{2k} \sum_{\lambda=1}^{2l} p_{+,k\kappa} \, p_{-,l\lambda} \frac{F_+^{(\kappa)}}{F_+} \frac{F_-^{(\lambda)}}{F_-},
\end{equation}
which is a linear equation of the form \eqref{eq:Tequation}. Now, the method of Sect.~\ref{sect:TracyWidom} suggests the following algorithm for inductively calculating the $p_{\pm,jk} \in \Q[t]$:

Suppose we have already found the $p_{\pm,jk}$ for all $j< m$, we get by comparing the $\Q[t]$-coefficients of the $q$-monomials in \eqref{eq:linear1} that the polynomials $p_{\pm,m\mu}$ satisfy an overdetermined linear system of size $(16m^2-8m+5)\times 4m$. For instance, $m=1$ gives the $13\times 4$ system
{\begin{equation}\label{eq:exampleSystem}
\begin{pmatrix}
 0 & 0 & -\frac{1}{2} & \frac{1}{2} \\*[1mm]
 \frac{1}{2} & \frac{1}{2} & 0 & 0 \\*[1mm]
 0 & 0 & \frac{1}{4} & \frac{1}{4} \\*[1mm]
 -\frac{1}{2} & \frac{1}{2} & 0 & 0 \\*[1mm]
 0 & 0 & -\frac{1}{2} & \frac{1}{2} \\*[1mm]
 -\frac{t}{2} & -\frac{t}{2} & -\frac{1}{4} & -\frac{1}{4} \\*[1mm]
 0 & 0 & -\frac{t}{2} & -\frac{t}{2} \\*[1mm]
 0 & 0 & \frac{t}{2} & -\frac{t}{2} \\*[1mm]
 -\frac{1}{2} & -\frac{1}{2} & \frac{t^2}{4} & \frac{t^2}{4} \\*[1mm]
 0 & 0 & -\frac{1}{2} & -\frac{1}{2} \\*[1mm]
 0 & 0 & \frac{1}{2} & -\frac{1}{2} \\*[1mm]
 0 & 0 & \frac{t}{2} & \frac{t}{2} \\*[1mm]
 0 & 0 & \frac{1}{4} & \frac{1}{4} 
\end{pmatrix}
\begin{pmatrix}
p_{+,11}\\*[1mm]
p_{-,11}\\*[1mm]
p_{+,12}\\*[1mm]
p_{-,12}
\end{pmatrix} =
\begin{pmatrix}
 0 \\
 p_{2,11} \\*[1mm]
 p_{2,12} \\*[1mm]
 0 \\
 0 \\
 -t p_{2,11}-p_{2,12} \\*[1mm]
 -2 t p_{2,12} \\*[1mm]
 0 \\
 t^2 p_{2,12}-p_{2,11} \\*[1mm]
 -2 p_{2,12} \\*[1mm]
 0 \\*[1mm]
 2 t p_{2,12} \\*[1mm]
 p_{2,12} \\*[1mm]
\end{pmatrix},
\end{equation}}%
which is uniquely solved by (reclaiming the already known relation between \eqref{eq:E21} and \eqref{eq:EpmOne})
\begin{subequations}\label{eq:hyposol}
\begin{equation}
p_{\pm,11} = p_{2,11},\qquad p_{\pm,12} = 2p_{2,12}.
\end{equation}
Continuing in the same fashion, the $53\times 8$ system for $m=2$ is uniquely solved by
\begin{equation}
\begin{aligned}
p_{\pm,21} &= p_{2,21}+2 p_{2,24},&\qquad p_{\pm,22} &= 2 p_{2,22}-\tfrac{1}{2} p_{2,11}^2,\\*[1mm]
p_{\pm,23} &=4 p_{2,23}-2 p_{2,11}p_{2,12},&\qquad p_{\pm,24} &= 8 p_{2,24}-2 p_{2,12}^2,
\end{aligned}
\end{equation}
 and the $125\times 12$ system for $m=3$ uniquely by
\begin{equation}
\begin{aligned}
p_{\pm,31} &= 8 t p_{2,36}+p_{2,12} p_{2,11}^2-2 p_{2,23} p_{2,11}+p_{2,31}+2 p_{2,34},\\*[1mm] 
p_{\pm,32} &= 8 p_{2,11} p_{2,12}^2-8 p_{2,23} p_{2,12}-p_{2,11} p_{2,21}-18 p_{2,11} p_{2,24}+2 p_{2,32}+20 p_{2,35},\\*[1mm]
p_{\pm,33} &=\tfrac{1}{2} p_{2,11}^3-2 p_{2,22} p_{2,11}+16 p_{2,12}^3-2 p_{2,12} p_{2,21}-68 p_{2,12} p_{2,24}+4 p_{2,33}+120 p_{2,36},\\*[1mm]
p_{\pm,34} &=3 p_{2,12} p_{2,11}^2-4 p_{2,23} p_{2,11}-4 p_{2,12} p_{2,22}+8 p_{2,34},\\*[1mm]
p_{\pm,35} &=6 p_{2,11} p_{2,12}^2-8 p_{2,23} p_{2,12}-8 p_{2,11} p_{2,24}+16 p_{2,35},\\*[1mm]
p_{\pm,36} &= 4 p_{2,12}^3-16 p_{2,24} p_{2,12}+32 p_{2,36}.
\end{aligned}
\end{equation}
\end{subequations}
With the help of a CAS\footnote{\label{fn:suppl}A supplementary Mathematica notebook that covers the results up to $m=10$ comes with the source files at \url{https://arxiv.org/abs/2306.03798}. There, also the formulae stated in \eqref{eq:FbetaP}, \eqref{eq:FbetaD}, \eqref{eq:Fstar_boxslash}, \eqref{eq:Fstar_boxbslash}, \eqref{eq:mu_fixedpointfree}, \eqref{eq:nu_fixedpointfree} and the numerical data shown in Tables~\ref{tab:mu}--\ref{tab:nu} are extended up to $m=10$.} these calculations were successfully (with unique solutions in each step) continued up to the $1525\times 40$ system for $m=10$---where we decided to stop. Just by inspection\footnote{Interestingly, and even more puzzling, such an inspection reveals that the polynomials $p_{\pm,jk}$ and $p_{2,jk}$ share a common sparsity pattern.} we obtain that, in all these instances, the $E_{\pm,j}$ share the same coeffient polynomials
\begin{equation}\label{eq:plusMinusSame}
p_{+,jk} = p_{-,jk} \quad (k=1,\ldots,2j, \, j=1,\ldots,m);
\end{equation}
we conjecture this to be generally true and the systems always to be uniquely solvable. 

Now, from the linearity of \eqref{eq:FfromG}, and from the representations \eqref{eq:hypoL} sharing the same coefficient polynomials for both choices of the sign $\pm$,  we get that  the  $E_{\beta,j}$ satisfy
\[
E_{\beta,j} = \sum_{k=1}^{2j} p_{+,jk} F_\beta^{(k)}.
\]
Inserting 
 the concrete expressions for $p_{2,jk}$ (see \cite[Eqs.~(3.4a--c)]{arxiv:2301.02022}) into \eqref{eq:hyposol} gives the explicit formulae for the $p_{+,jk}$ ($\beta=1,4$) as displayed in (\ref{eq:F_beta_j}b/c); which finishes the proof of Thm.~\ref{thm:hard2soft}.

\section{Expansion of the (Generalized) Poissonized Distributions}\label{sect:poissondist}

Expanding the (generalized) Poissonizations of the length distributions requires the hard-to-soft edge transition of Thm.~\ref{thm:hard2soft} to be applied to the probability distributions 
\[
E_\beta^\text{hard}(4r_\varoast;\nu_\beta)\Big|_{\beta=\beta(\varoast)}
\]
for integer $\nu=l^\varoast$, but we begin with considering a general $\nu$ first. Using the scaling $t_\nu(r)$ as introduced in \eqref{eq:tnu_vanilla}, 
which satisfies the differential equation
\begin{equation}\label{eq:tnu}
t_\nu'(r) = -r^{-2/3} - \frac{r^{-1}}{6} t_\nu(r),
\end{equation}
we observe that, for large intensities $r$, the modes of those probability distributions are located in the range of parameters $\nu$ for which the scaled variable $t_{\nu}(r_\varoast)$ stays bounded. 

In fact, with literally the same proof as for the $\beta=2$ case \cite[Thm.~4.1]{arxiv:2301.02022}, we infer from Thm.~\ref{thm:hard2soft} the following theorem. In particular, we recall from that proof the relation
\begin{equation}\label{eq:Ebeta1Fbeta1}
F_{\beta,1}(t) = \frac{1}{2} E_{\beta,1}(t) - \frac{t^2}{6} F_\beta'(t)
\end{equation}
which we already used in Sect.~\ref{sect:firstorder} to infer the $\beta=1$ case of \eqref{eq:F1} from  \eqref{eq:BaikJenkinsBeta1}.

\begin{theorem}\label{thm:FbetaP} Let be $\beta=1$ or $\beta = 4$. Then there holds the expansion
\begin{equation}\label{eq:FbetaPExpan}
E_\beta^{\text{\rm hard}}(4r;\nu_\beta) = F_\beta(t) + \sum_{j=1}^m F_{\beta,j}(t) \cdot r^{-j/3} +r^{-(m+1)/3}\cdot O(e^{-t/2})\bigg|_{t=t_{\nu}(r)},
\end{equation}
which is uniformly valid when $r,\nu \to \infty$ subject to $t_0 \leq t_{\nu}(r) \leq r^{1/3}$,
with $m$ being any fixed non-negative integer and $t_0$ any fixed real number. Preserving uniformity, the expansion can be repeatedly differentiated w.r.t. the variable $r$. Here the $F_{\beta,j}(t)$ are certain smooth functions{\rm;} that have simple expressions in terms of the functions $F_{\beta}$, $E_{\beta,j}$ in \eqref{eq:hard2soft}; the first three  are\footnote{To validate the formulae displayed in (\ref{eq:FbetaP}), 
Fig.~\ref{fig:FbetaP} plots $F_{\beta,3}(t)$ next to the approximation
\begin{equation}\label{eq:FbetaP3}
F_{\beta,3}(t) \approx r \cdot \Big( E_\beta^{\text{hard}}(4r;\nu_\beta) - F_\beta(t) - F_{\beta,1}(t)\cdot r^{-1/3} - F_{\beta,2}(t)\cdot r^{-2/3}\Big)\,\Big|_{t=t_{\nu}(r)}
\end{equation}
for $r=160$ and $r=1200$, varying $\nu$ in such a way that $t_{\nu}(r)$ covers the range of $t$ on display.}
\begin{subequations}\label{eq:FbetaP}
\begin{align}
F_{\beta,1}(t) &= -\frac{t^2}{60} F_\beta'(t) - \frac{1}{5} F_\beta''(t)\label{eq:FbetaP1}\\
\intertext{and, subject to the linear form hypothesis,}
F_{\beta,2}(t) &= \Big(\frac{9}{700} + \frac{2t^3}{1575}\Big) F_\beta'(t) + \Big(\frac{11t}{525} + \frac{t^4}{7200}\Big) F_\beta''(t) + \frac{t^2}{300} F_\beta'''(t) + \frac{1}{50} F_\beta^{(4)}(t),\label{eq:FbetaP2}\\*[2mm]
F_{\beta,3}(t) &=  -\Big(\frac{34 t}{7875} + \frac{41 t^4}{283500}\Big)  F_\beta'(t)
-\Big( \frac{13t^2}{3600}  + \frac{t^5}{47250} \Big) F_\beta''(t)  \\*[1mm]
&\quad -\Big(\frac{289}{31500} +\frac{19 t^3}{31500} + \frac{t^6}{1296000}\Big) F_\beta'''(t)
 -\Big(\frac{11t}{2625} +\frac{t^4}{36000} \Big)F_\beta^{(4)}(t)\notag\\*[1mm]
&\quad -\frac{t^2 }{3000}F_\beta^{(5)}(t)
-\frac{1}{750} F_\beta^{(6)}(t).\notag
\end{align}
\end{subequations}
\end{theorem}

\begin{figure}[tbp]
\includegraphics[width=0.325\textwidth]{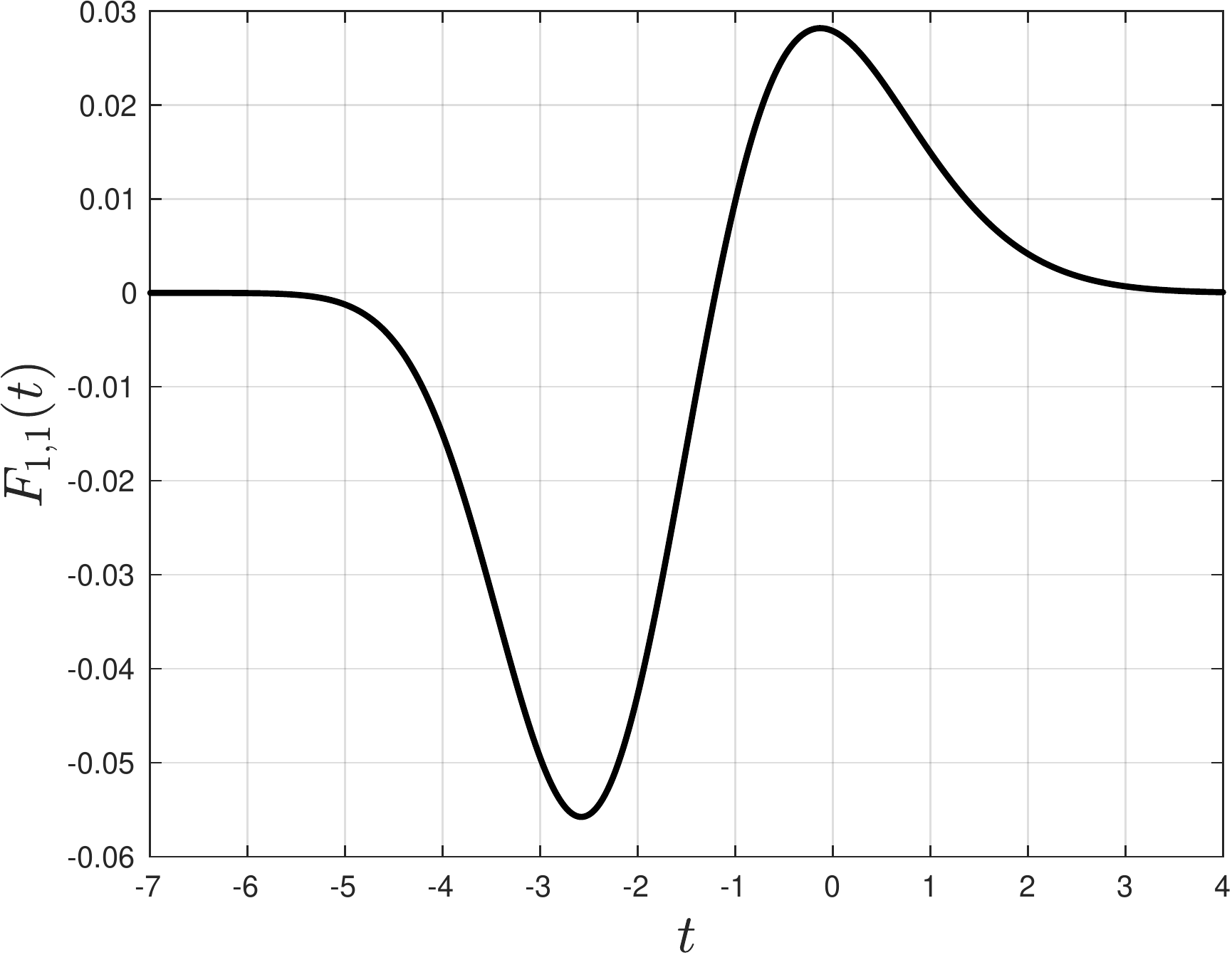}\hfil\,
\includegraphics[width=0.325\textwidth]{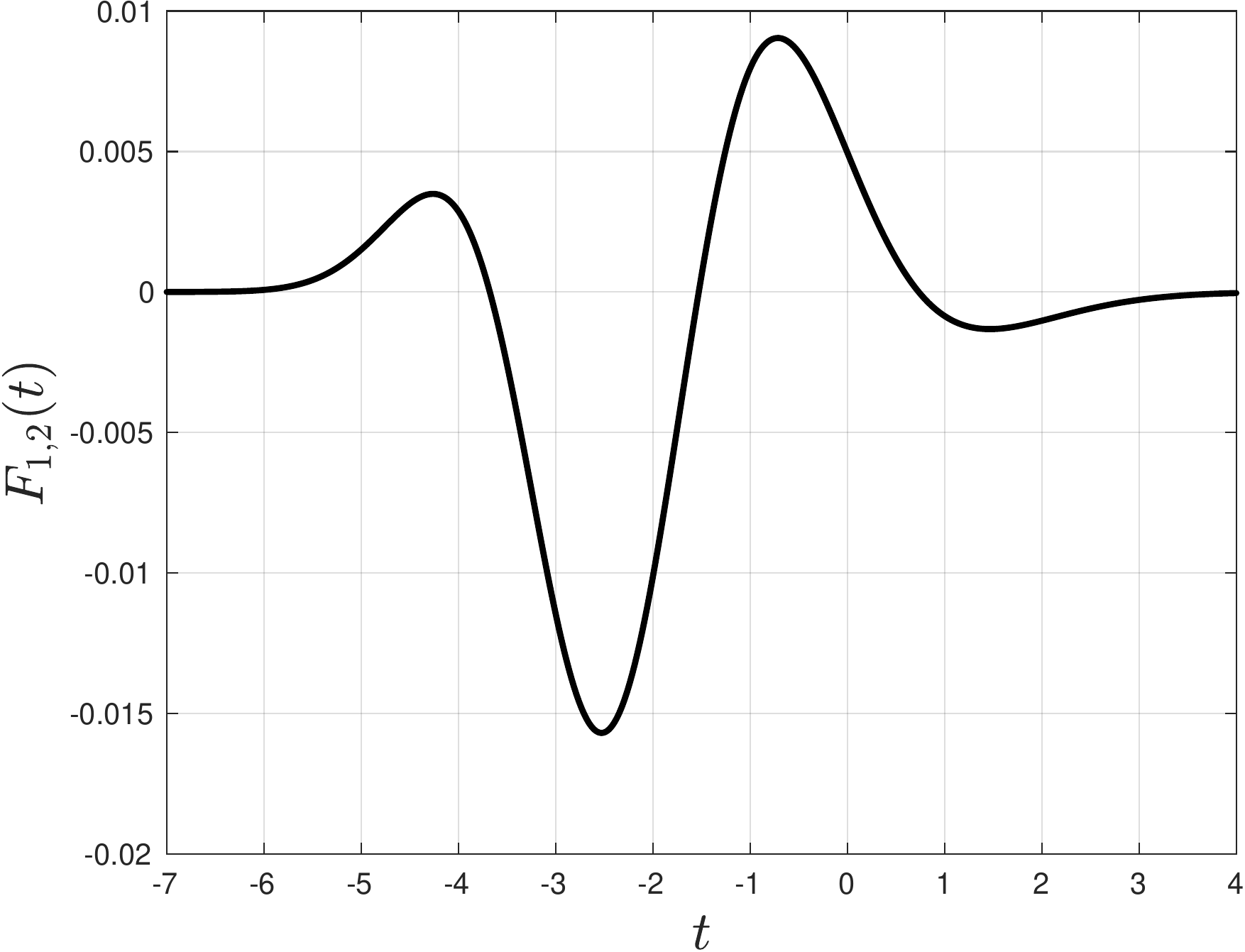}\hfil\,
\includegraphics[width=0.3125\textwidth]{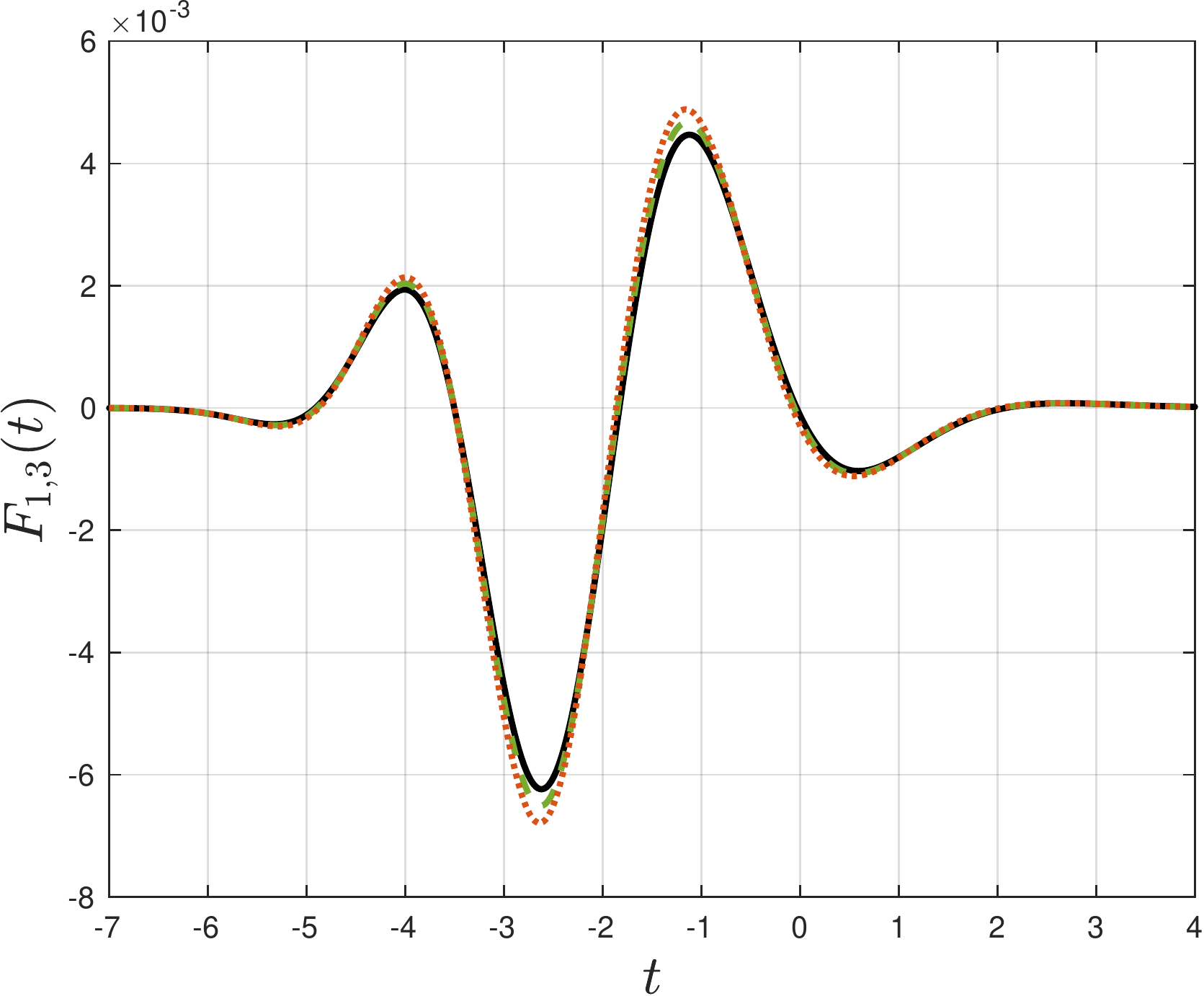}\\
\includegraphics[width=0.325\textwidth]{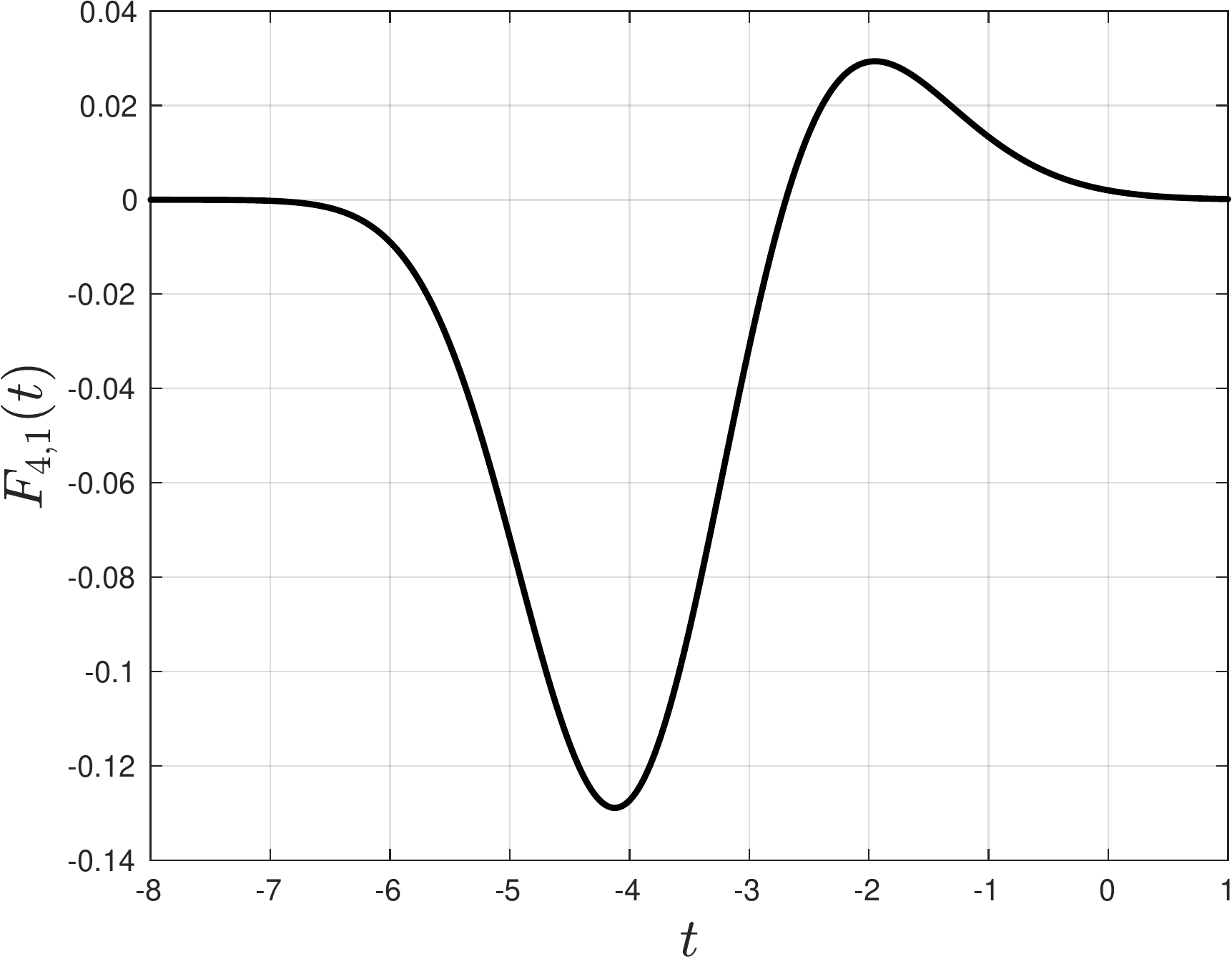}\hfil\,
\includegraphics[width=0.325\textwidth]{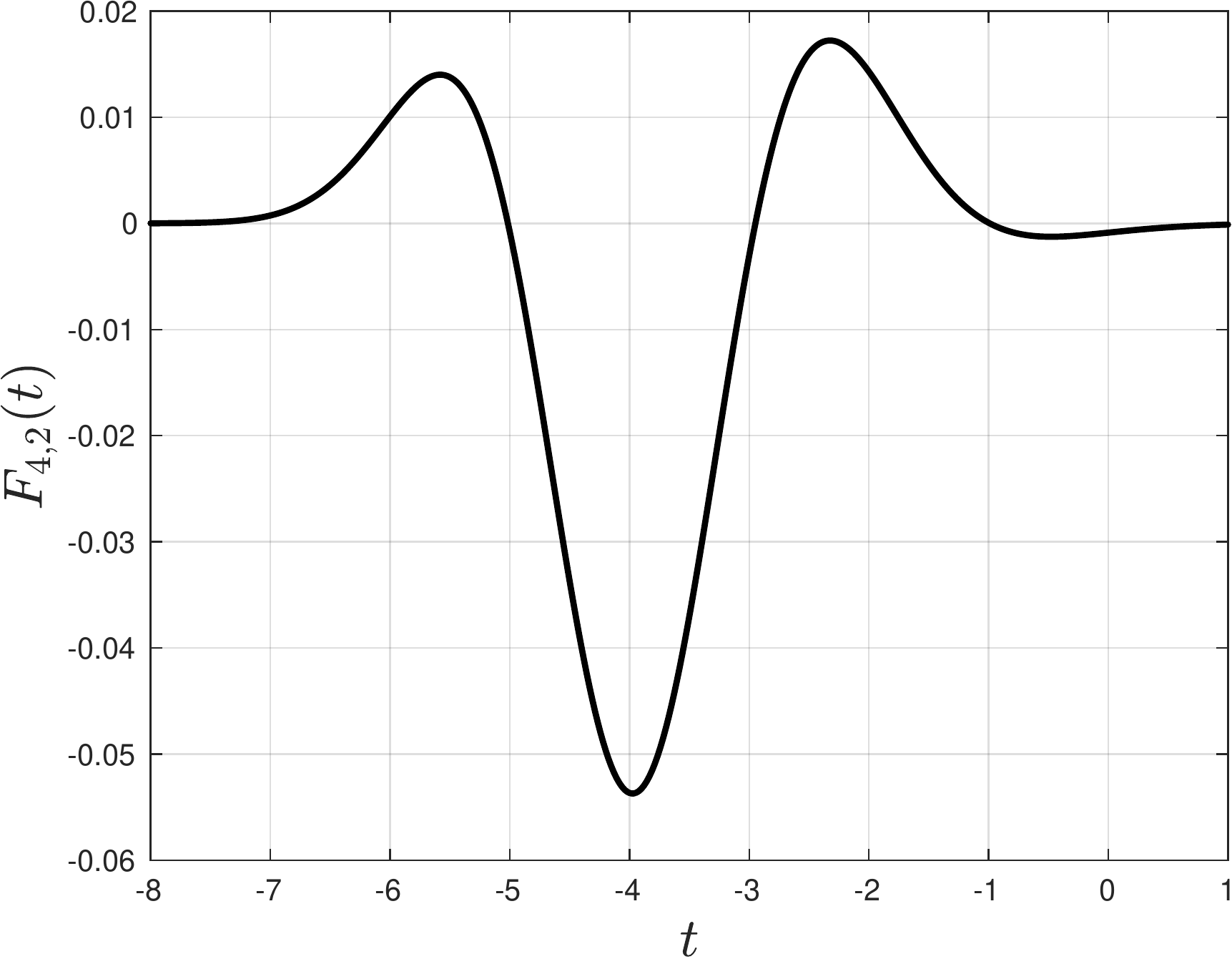}\hfil
\includegraphics[width=0.325\textwidth]{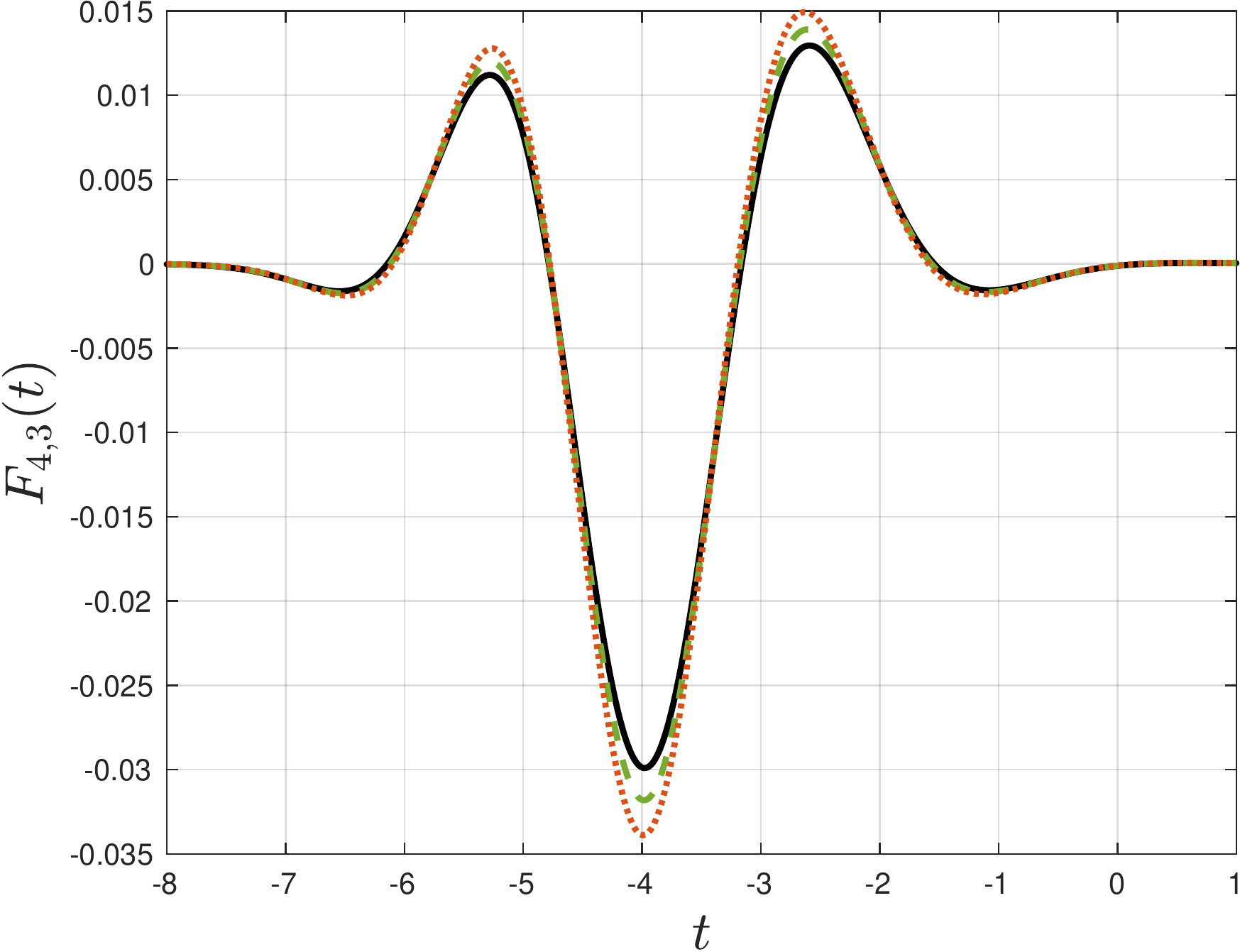}
\caption{{\footnotesize Top row $\beta=1$; bottom row $\beta=4$. Plots of $F_{\beta,1}(t)$ (left panels) and $F_{\beta,2}(t)$ (middle panels) as in \eqref{eq:FbetaP}.
The right panels show $F_{\beta,3}(t)$ as in \eqref{eq:FbetaP} (black solid line) with the approximation \eqref{eq:FbetaP3} for $r=160$ (red dotted line) and $r=1200$ (green dashed line); the parameter $\nu$ has been varied such that $t_{\nu}(2r)$ covers the range of $t$ on display. Note that the functions $F_{\beta,j}(t)$ ($j=1,2,3$) are about two orders of magnitude smaller in scale than their counterparts in Fig.~\ref{fig:hard2soft}.}}
\label{fig:FbetaP}
\end{figure}


\begin{remark} Combined with the $\beta=2$ case \cite[Thm.~4.1]{arxiv:2301.02022} the leading order term of \eqref{eq:FbetaPExpan} and its optimal rate of convergence can be written, for $\beta=1,2,4$, in the unified form 
\[
E_\beta^\text{hard}(4r^2;\nu_\beta)\Big|_{\nu=2r+t r^{1/3}} = F_\beta(t) + O(\nu^{-2/3}),
\]
uniformly valid as $\nu\to\infty$ while $t$ stays bounded; this corrects a suggestion in \cite[Eq.(3.27)]{arxiv.2205.05257}.

\end{remark}

If we use the abbreviations \eqref{eq:lr_varoast}
we infer from the case of integer $\nu=l^\varoast$ in Thm.~\ref{thm:FbetaP} that the (generalized) Poisson generating functions 
in \eqref{eq:Poisson} expand in a unified fashion as follows:\footnote{\label{fn:FM23}Up to a slightly different scaling and the use of Gauss brackets, the case $m=1$ was previously put forward for $\varoast=\boxbslash, \boxslash$ as \cite[Eqs.~(3.11)/(3.21)]{arxiv.2205.05257} with $F_{\beta,1}(t)$ expressed in 
operator theoretic terms and in terms of  Painlevé III$'$ transcendents. We note, however, that the error estimate as stated in \cite[Eqs.~(3.11)/(3.21)]{arxiv.2205.05257} neglects the effects of the Gauss bracket (see \cite[Rem.~4.2]{arxiv:2301.02022}). When comparing the left and middle columns of Fig.~\ref{fig:FbetaP} with \cite[Figs.~4/5]{arxiv.2205.05257} a similar remark applies as in \cite[Fn.~27]{arxiv:2301.02022}.}%
\begin{corollary}\label{cor:gen_Poisson_expan} For $\varoast\in\{\boxslash,\boxbslash,\boxdot\}$ the (generalized) Poissonizations have the expansion
\begin{equation}\label{eq:PoissonVaroastExpan}
P_\varoast(r;l) = F_\beta(t) + \sum_{j=1}^m F_{\beta,j}(t) \cdot r_\varoast^{-j/3} +r_\varoast^{-(m+1)/3}\cdot O(e^{-t/2})\bigg|_{t=t_{l^\varoast}(r_\varoast),\; \beta = \beta(\varoast)},
\end{equation}
which is uniformly valid when $r,l \to \infty$ subject to $t_0 \leq t_{l^\varoast}(r_\varoast) \leq r_\varoast^{1/3}$, with $m$ being any fixed non-negative integer and $t_0$ any fixed real number. Preserving uniformity, the expansion can be repeatedly differentiated w.r.t. the variable $r$. 
\end{corollary}

\section{Generalized De-Poissonization}\label{sect:gen_de_poisson}

Consider an entire function 
\[
f(z) = \sum_{n=0}^\infty a_n z^n
\]
with positive Maclaurin coefficients $a_n > 0$, such that in particular $f(r) > 0$ for $r\geq 0$. Thus there is a random variable $N_r \in \N_0$ with distribution
\begin{equation}\label{eq:gen_poisson_dist}
\prob(N_r = n) = \frac{a_n r^n}{f(r)},
\end{equation}
which we call the {\em generalized} Poisson distribution\footnote{In the context of 
sampling from random combinatorial structures, these distributions are often called {\em Boltzmann probabilities}, cf.~\cite{MR2095975}.} of intensity $r$ induced by $f$; the standard Poisson case corresponds to the choice $f(z)=e^z$. 

The mean and variance of $N_r$ are easily seen to be the auxiliary functions associated with the entire function $f$ (cf. the positivity part of Def.~\ref{def:hayman}), 
\begin{equation}\label{eq:gen_Poisson_mean_var}
\E(N_r) = r \frac{f'(r)}{f(r)} =: a(r),\quad \Var(N_r) = r a'(r) =: b(r).
\end{equation}
By Hadamard's convexity theorem $a(r)$ is monotonely increasing and $b(r)$ is positive.

Given a sequence $(p_n)$ of probabilities, the associated {\em generalized} Poisson generating function, or {\em generalized} Poissonization, is defined as
\begin{equation}\label{eq:genPr}
P(r) := \sum_{n=0}^\infty p_n \frac{a_n r^n}{f(r)},
\end{equation}
which clearly continues to a meromorphic (entire if $f$ is zero free) function $P(z)$ in the complex plane. If the probability distribution is sufficiently concentrated around its mean $a(r)$, we would expect intuitively that there are some intensities $r_n$ such that simultaneously
\begin{equation}\label{eq:intuition}
a(r_n) \approx n, \quad p_n \approx P(r_n).
\end{equation}
When $r_n \to \infty$ as $n\to\infty$, such an approximate reconstruction is particularly useful if the generalized Poissonization $P(r)$ enjoys a comparatively simple asymptotics as $r\to \infty$. We call such an asymptotic reconstruction of $p_n$ a {\em generalized} de-Poissonization. Because of the necessity of suitable additional assumptions on $P(r)$ (or on the sequence $p_n$ to begin with), such a de-Poissonization is {\em Tauberian} in nature.  We discuss
the generalization of two Tauberian conditions that have been put forward to address the standard Poisson case.

\subsection{Generalized De-Poissonization by monotonicity}\label{sect:johansson}

The following result generalizes Johansson's de-Poissonization lemma \cite{MR1618351} (we modify the exposition given in \cite[§2.2]{MR3468920} accordingly), the Tauberian condition being here the monotonicity of the sequence of probabilities under consideration.

\begin{lemma}\label{lem:depoisson_mono} Let $f(z)$ be an $H$-admissible entire function with positive Mclaurin coefficients and auxiliary functions $a(r), b(r)$ as in \eqref{eq:gen_Poisson_mean_var}. Let $P(z)$ be the associated generalized Poisson generating function of a decreasing sequence of probabilities
\[
1\geq p_0 \geq p_1 \geq p_2 \geq \cdots \geq 0.
\] 
Then:
\begin{itemize}\itemsep=5pt
\item[I.]
If $r_n^- < r_n < r_n^+$, where $r_n$ is the unique solution of $a(r_n)=n$, there holds 
\begin{equation}\label{eq:enclosure}
P(r_n^+) - \Delta(r_n^+) \leq p_n \leq P(r_n^-) + \Delta(r_n^-), \qquad \Delta_n(r) := (r/r_n)^n\frac{f(r_n)}{f(r)}.
\end{equation}
Note that the error terms are independent of the sequence $(p_n)$. 

\item[II.] If,  for any fixed $\alpha > 0$, the $r_n^\pm$ are chosen as solutions of
\begin{equation}\label{eq:rnpm}
a(r_n^\pm) = n  \pm \sqrt{2\alpha\cdot b(r_n^\pm) \log n}
\end{equation}
such that $b(r_n^\pm) \sim b(r_n)$ as $n\to \infty$, the error term in \eqref{eq:enclosure} satisfies
\begin{equation}\label{eq:DeltaEst}
\Delta_n(r_n^\pm) = n^{-\alpha} + o(1).
\end{equation}
\end{itemize}
\end{lemma}

\begin{remark} The estimate \eqref{eq:DeltaEst} can often be sharpened to take the form (see Example~\ref{ex:depoisson_mono})
\[
\Delta_n(r_n^\pm) = n^{-\alpha}(1 + o(1)),
\]
which is the reason why we prefer the form stated in \eqref{eq:DeltaEst} to writing $\Delta_n(r_n^\pm) = o(1)$.
\end{remark}

\begin{proof} I. Take a random variable $N_r$ distributed according to \eqref{eq:gen_poisson_dist}.
The monotonicity of the probabilities $p_n$ then implies, for all $r\geq 0$,
\begin{multline*}
p_n - \prob(N_r \geq n) \leq (1-\prob(N_r>n)) \cdot p_n = \prob(N_r\leq n)\cdot p_n = \left(\sum_{m\leq n} \frac{a_m r^m}{f(r)}\right)\cdot p_n\\*[1mm]
\leq P(r) = \sum_{m\leq n}  \frac{a_m r^m}{f(r)} p_m + \sum_{m> n}  \frac{a_m r^m}{f(r)} p_m
\leq \sum_{m\leq n}  \frac{a_m r^m}{f(r)} + p_n = p_n + \prob(N_r \leq n),
\end{multline*}
that is, after rearranging, the enclosure
\[
P(r) - \prob(N_r \leq n) \leq p_n \leq P(r) + \prob(N_r \geq  n).
\]
We now establish bounds on the tails of the generalized Poisson distribution.
First, Markov's inequality  allows us to separate the tails by introducing a free parameter $t$ in the form
\[
\frac{\E \big(t^{N_r} \big)}{t^n} \geq \prob\big(t^{N_r} \geq t^{n} \big) =
\begin{cases}
\prob\big(N_r \geq n \big) &\quad \text{$t>1$,}\\*[2mm]
\prob\big(N_r \leq n \big) &\quad \text{$0<t<1$.}
\end{cases}
\]
Next, since by \eqref{eq:gen_poisson_dist} 
\[
\E \big(t^{N_r} \big) = \sum_{k=0}^\infty t^k \cdot \prob(N_r = k) = \frac{f(tr)}{f(r)}
\]
and thus
\[
\frac{d}{dt} \frac{\E \big(t^{N_r} \big)}{t^n} = \frac{d}{dt} \frac{f(t r)}{t^{n} f(r)} = \frac{f(t r)}{t^{n+1} f(r)} \left(a(t r) - n\right),
\]
an optimization of the Markov bounds over the free parameter $t$ yields the Chernoff bounds
\[
\prob(N_{r_n^-} \geq n) \leq \Delta_n(r_n^-),\qquad \prob(N_{r_n^+} \leq n) \leq \Delta_n(r_n^+).
\]

II. By normal approximation \eqref{eq:CLT} we get, when $r,n \to  \infty$, the asymptotics
\[
\Delta_n(r) = \sqrt{\frac{b(r_n)}{b(r)}}\left( \exp\left(-\frac{(n-a(r))^2}{2b(r)}\right)+ o(1)\right),
\]
so that the particular choices \eqref{eq:rnpm} for $r_n^\pm$ yield \eqref{eq:DeltaEst} if $b(r_n^\pm) \sim b(r_n)$ as $n\to \infty$.
\end{proof}

\begin{example}\label{ex:depoisson_mono} Looking at the two concrete cases that are relevant in the present paper, we demonstrate  that the estimate \eqref{eq:DeltaEst} can be sharpened to 
\[
\Delta_n(r_n^\pm) = n^{-\alpha}(1 + o(1))
\]
even if we construct the $r_n^\pm$ by solving \eqref{eq:rnpm} through an expansion as $n\to\infty$, stopping right at the first order where the choice of the sign becomes relevant. 

\medskip

\begin{subequations}
\begin{itemize}
\item The standard Poisson case,  $f(z)=e^z$. Here we have $a(r)=b(r)=r$, $r_n=n$ and the equation \eqref{eq:rnpm} is solved to leading orders by
\begin{equation}\label{eq:rnpm_std_poisson}
r_n^{\pm} := n  \pm \sqrt{2\alpha\, n \log n}.
\end{equation}
\item The exponential generating function \eqref{eq:finv} of the number of involutions, $f(z)=e^{z+z^2/2}$. Here we have $a(r)=r+r^2$, $b(r)=r+2r^2$,
\[
r_n = \sqrt{n+\frac14}-\frac12 = n^{1/2} - \frac{1}{2} + O(n^{-1/2})
\]
and the equation \eqref{eq:rnpm} is solved to leading orders by
\begin{equation}\label{eq:rnpm_inv}
r_n^\pm := n^{1/2} - \frac{1}{2} \pm \sqrt{\alpha \log n}.
\end{equation}
\end{itemize}

\noindent
In both cases we have $r_n^- < r_n < r_n^+$ for sufficiently large\footnote{E.g., $n\geq 2$ if $\alpha\geq 1/94$ (only relevant in the second example, though).} $n$ and, if plugged into the definition of $\Delta_n(r)$, a routine calculation shows
\begin{equation}\label{eq:Delta_dePoisson}
\Delta_n(r_n^{\pm}) = n^{-\alpha} \left(1 + O\big(n^{-1/2} \log^{3/2} n\big)\right) \qquad (n\to\infty).
\end{equation}
\end{subequations}
\end{example} 

\subsection{Generalized Jasz expansion}\label{sect:jasz} 

In the standard Poisson case, a much finer tool for an asymptotic reconstruction of $p_n$ from the Poisson generating function $P(r)$ is {\em analytic} de-Poissonization, as studied in the 1998 memoir of Jacquet and Szpankowski \cite{MR1625392}.
This technique allows for precise asymptotic expansions,\footnote{Dubbed {\em Jasz expansions} in \cite{MR2483235}.} the Tauberian condition being a growth condition of $P(z)$ in the complex domain at its essential singularity $z=\infty$.

This technique was used in our work \cite{arxiv:2301.02022} on asymptotic expansions in the general permutation case; the required amount of uniformity was obtained from the theory of $H$-admissibility and a certain hypothesis (tameness hypothesis) regarding the $l$-dependent families of the finitely many zeros of $P(z;l)$ in certain sectors of the complex plane. For details and more references see \cite[§5 and Appendix A]{arxiv:2301.02022}.

Though it is rather clear that the fundamental theorem of Jacquet and Szpankowski \cite[Thm.~A.2]{arxiv:2301.02022} can be extended to cover generalized Poisson generating functions \eqref{eq:gen_Poisson_mean_var} as studied here, we leave the detailed analytic estimates to future work and content ourselves, for the time being, with establishing generalized Jasz expansions in just a formal fashion.

We consider the case of a zero free $f(z)$ such that the induced generalized Poisson generating function $P(z)$ is entire. We write the power series expansion of $P(z)$, centered at $z=r$, in the operator form
\[
P(z) = e^{(z-r)D} P(r),
\]
where $D$ denotes differentiation w.r.t. the variable $r$. By Cauchy's formula (taking a contour encircling $z=0$ counter-clockwise with index one) we get
\begin{subequations}\label{eq:jasz_gen}
\begin{multline}
p_n = \frac{1}{a_n\,2\pi i} \oint P(z) f(z) \frac{dz}{z^{n+1}} =  \left(\frac{e^{-rD}}{a_n\,2\pi i} \oint e^{zD} f(z) \frac{dz}{z^{n+1}} \right) P(r)\\*[2mm]
= \left(\frac{e^{-rD} }{a_n} [z^n] e^{zD} f(z) \right) P(r) = \left(e^{-rD} \sum_{k=0}^n \frac{a_{n-k}}{a_n k!} D^k \right) P(r) = \sum_{j=0}^\infty c_j(n;r) P^{(j)}(r),
\end{multline}
where the symbol $[z^n]$ means extraction of the coefficient of $z^n$ in a power series and the coefficients $c_j(n;r)$ are polynomials in $r$, obtained from evaluating the Cauchy product
\[
e^{-rx} \sum_{k=0}^n \frac{a_{n-k}}{a_n k!} x^k = \sum_{j=0}^\infty c_j(n;r) x^j.
\]
Putting $a_k = 0$ for $k<0$, we thus get the explicit expression
\begin{equation}
c_j(n;r) = \frac{1}{j!} \sum_{k=0}^{j} \binom{j}{k} \frac{a_{n-k}}{a_n} (-r)^{j-k}.
\end{equation}
\end{subequations}
The first few of these polynomials are
\[
c_0(n;r) = 1,\quad c_1(n;r) = \frac{a_{n-1}}{a_n} - r,\quad c_2(n;r) = \frac{1}{2}\left( \frac{a_{n-2}}{a_n}-\frac{2a_{n-1}}{a_n}r+r^2\right).
\]
As $n\to\infty$, the formal series \eqref{eq:jasz_gen} can be turned into an asymptotic expansion by first choosing any reasonably simple approximation $r_n^* \approx r_n$ (where $a(r_n)=n$), followed by expanding the derivatives 
 $P^{(j)}(r_n^*)$ and the ratios $a_{n-k}/a_n$.

\subsubsection{The standard Poisson case} Here we have $f(z)=e^z$, $a_n = 1/n!$ and $r_n=n$. As seen in \cite[Eq.~(A.4)]{arxiv:2301.02022} the coefficients $b_j(n):=c_j(n,n)$ are polynomials in $n$ of a degree $\leq j/2$. In Sect.~\ref{sect:CDFexpan} we apply the Jasz expansion \eqref{eq:jasz_gen} to a (family of) Poisson generating functions with
\begin{subequations}\label{eq:jasz_std_4}
\begin{equation}
P^{(j)}(n) = O(n^{-2j/3}).
\end{equation}
For such a $P(r)$, if we truncate the expansion at $O(n^{-4/3})$ and keep only those terms which do not get absorbed in the error term, we thus get the particular Jasz expansion
\begin{equation}
p_n = P(n) - \frac{n}{2} P''(n) + \frac{n}{3} P'''(n) + \frac{n^2}{8} P^{(4)}(n) - \frac{n^3}{48}P^{(6)}(n) + O(n^{-4/3}).
\end{equation}
\end{subequations}
We note that if the (family of) $P(z)$ satisfies the assumptions of \cite[Thm.~A.2]{arxiv:2301.02022}, this expansion can be established rigorously (see~\cite[Example~A.3]{arxiv:2301.02022} for details).

\subsubsection{The case of the exponential generating function of the number of involutions}\label{sect:jasz_gen} Here we have $f(z)=e^{z+z^2/2}$ with $a_n = I_n/n!$. Particularly simple intermediate results are obtained if we approximate \eqref{eq:rinv} by
\[
r_n^* = n^{1/2} - \frac{1}{2}.
\]
We then get, by using the asymptotic expansion \eqref{eq:iotaexpansion} for $a_n$, that $c_0(n;r_n^*) = 1$ and
\begin{align*}
c_1(n;r_n^*) &= \frac{3 n^{-1/2}}{8}-\frac{n^{-1}}{8}-\frac{n^{-3/2}}{128} + O(n^{-2}),\\*[1mm]
c_2(n;r_n^*) &=-\frac{1}{4}+\frac{n^{-1/2}}{8}+\frac{n^{-1}}{128} + O(n^{-3/2}),\\*[1mm]
c_3(n;r_n^*) &=-\frac{5 n^{-1/2}}{96}+\frac{5n^{-1}}{64} -\frac{167n^{-3/2}}{3072} + O(n^{-2}),\\*[1mm]
c_4(n;r_n^*) &= \frac{1}{32}-\frac{n^{-1/2}}{32}+\frac{11n^{-1}}{512} + O(n^{-3/2}),\\*[1mm]
c_5(n;r_n^*) &= \frac{n^{-1/2}}{768} -\frac{n^{-1}}{96} + \frac{159n^{-3/2}}{10240} + O(n^{-2}),\\*[1mm]
c_6(n;r_n^*) &=-\frac{1}{384}+\frac{n^{-1/2}}{256} - \frac{175n^{-1}}{36864} + O(n^{-3/2}). 
\end{align*}
Generally we have $c_{2j}(n;r_n^*) = O(1)$ and $c_{2j+1}(n;r_n^*) = O(n^{-1/2})$.
In Sect.~\ref{sect:CDFexpan} we apply the generalized Jasz expansion \eqref{eq:jasz_gen} to a (family of) Poisson generating functions satisfying
\begin{subequations}\label{eq:jasz_gen_4}
\begin{equation}
P^{(j)}(n) = O(n^{-j/6}).
\end{equation}
For such a $P(r)$, if we truncate the expansion at $O(n^{-4/3})$ and keep only those terms which do not get absorbed in the error term, we thus get the particular generalized Jasz expansion
\begin{equation}
\begin{aligned}
p_n &= P(r_n^*) + \left(\frac{3n^{-1/2}}{8} - \frac{n^{-1}}{8}\right) P'(r_n^*) + \left(-\frac{1}{4}+\frac{n^{-1/2}}{8}\right) P''(r_n^*) -\frac{5 n^{-1/2}}{96} P'''(r_n^*) \\*[2mm]
 &\quad +\left(\frac{1}{32}-\frac{n^{-1/2}}{32}\right) P^{(4)}(r_n^*) -\frac{1}{384} P^{(6)}(r_n^*) + O(n^{-4/3}).
\end{aligned}
\end{equation}
\end{subequations}

\section{Expansions of the Length Distributions}\label{sect:depoisson}

\subsection{The limit law} If we apply the monotonicity based (generalized) de-Poissonization of Lemma~\ref{lem:depoisson_mono} to the expansion \eqref{eq:PoissonVaroastExpan}, we obtain the following theorem. To the extent that an error estimate is given here (suboptimal, though), we sharpen the limit laws of Baik and Rains \cite[Thm.~3.1/3.4]{MR1845180}.\footnote{The general involution case was studied in \cite{MR1845180} by means of a multi-variate standard Poisson generating function that accounts for the undetermined number of fixed-points in its second intensity.
 In contrast, we deal with that case in a more direct fashion, using the notion of generalized Poisson generating functions.} Similar suboptimal  $O\big(n^{-1/6}\sqrt{\log n}\,\big)$ error estimates were established in \cite[Cor.~1.2]{MR3161478} for the joint probability distribution of maximal crossing and nesting in random matchings, and in \cite[Eq.~(1.12)]{arxiv:2301.02022} for the general permutation case.

\begin{theorem}\label{thm:leading_de_poisson} For $\varoast\in\{\boxslash,\boxbslash,\boxdot\}$, writing $\gamma(\boxdot) := 1$ and $\gamma(\varoast) := 2$ otherwise, the discrete probability distributions satisfy the limit law
\begin{equation}\label{eq:leading_de_poisson}
p_\varoast(n;l) = F_{\beta(\varoast)}\big(t_{l^\varoast}(\gamma(\varoast) \cdot n)\big) + O\big(n^{-1/6}\sqrt{\log n}\,\big),
\end{equation}
which is uniformly valid when $n, l \to \infty$ subject to $t_0\leq t_{l^\varoast}(\gamma(\varoast) \cdot n) \leq t_1$, with $t_0<t_1$ being any fixed ordered pair of real numbers.
\end{theorem}
\begin{proof}  Using $f(z)=e^z$ in the fixed-point free cases $\varoast=\boxslash,\boxbslash$ and $f(z)=e^{z+z^2/2}$ in the general involution case $\varoast=\boxdot$, the generalized Poisson generating function of  $(p_\varoast(n;l))_n$ is given by $P_{\varoast}(r;l)$. Since the combinatorial monotonicity properties outlined in \cite[Lemma~7.5]{MR1845180} allow us to apply Lemma~\ref{lem:depoisson_mono}, we get by choosing $\alpha=1$ in the estimates of Example~\ref{ex:depoisson_mono} the enclosure
(with error terms that are independent of $l$)
\[
P_{\varoast}(r_{n;\varoast}^+;l) + O(n^{-1}) \leq p_\varoast(n;l) \leq P_{\varoast}(r_{n;\varoast}^-;l) + O(n^{-1}),
\]
where 
\[
r_{n;\varoast}^\pm :=n\pm \sqrt{2n\log n} \quad (\varoast =\boxslash,\boxbslash), \qquad
r_{n;\boxdot}^\pm := n^{1/2}-\frac{1}{2}\pm\sqrt{\log n}.
\]
We observe that, in all cases,
\[
(r_{n;\varoast}^\pm)_\varoast = \gamma(\varoast) n + O(\sqrt{n\log n})
\]
so that a routine expansion gives, uniformly valid under the stated assumptions on $n$, $l$, 
\[
t_{l^\varoast}((r_{n;\varoast}^\pm)_\varoast) = t_{l^\varoast}(\gamma(\varoast) n) + O\big(n^{1/6}\sqrt{\log n}\,\big).
\]
The case $m=0$ of the expansion \eqref{eq:PoissonVaroastExpan}, followed by a Taylor expansion, gives then
\[
P_{\varoast}(r_{n;\varoast}^\pm;l) = F_{\beta(\varoast)}\big(t_{l^\varoast}((r_{n;\varoast}^\pm)_\varoast)\big) + O(n^{-1/3})
= F_{\beta(\varoast)}\big(t_{l^\varoast}(\gamma(\varoast) n)\big) + O\big(n^{1/6}\sqrt{\log n}\,\big),
\]
which finishes the proof.
\end{proof}

The weakest link in the estimates of the proof, which causes the suboptimal $O\big(n^{1/6}\sqrt{\log n}\,\big)$ error term (cf. Rem.~\ref{rem:onesixth} and Cor.~\ref{cor:limit_law_finite_size} for the optimal error terms), is that the monotonicity based sandwiching of Lemma~\ref{lem:depoisson_mono} forces us to keep,  by a factor
\[
1 \pm c \sqrt{n^{-1} \log n},
\]
a safe distance to the ``proper'' choice $r_n$  of the intensity (which satisfies $a(r_n)=n$).
Analytic generalized de-Poissonization allows us to eliminate that need of keeping a distance.

\subsection{Expansions of distributions: the fixed-point free cases}\label{sect:CDFexpan}

For $\varoast \in \{\boxslash,\boxbslash\}$ we are dealing with standard Pois\-son\-ization and follow the arguments given in \cite[§5.1]{arxiv:2301.02022} for the structurally similar general permutation case. 

Let us consider any fixed compact interval $[t_0,t_1]$ and a sequence of integers $l_n\to \infty$ with
\begin{equation}\label{eq:tnstar} 
t_0 \leq t_n^*:= t_{l_n^\varoast}(2n) \leq t_1.
\end{equation}
When $n-n^{3/5} \leq r \leq n + n^{3/5}$ and $n\geq n_0$ with $n_0$ large enough (depending only on $t_0, t_1$) we thus get the uniform bounds
\[
2\sqrt{2r} + (t_0-1) (2r)^{1/6} \leq l_n^\varoast \leq 2\sqrt{2r} + (t_1+1) (2r)^{1/6}.
\]
Suppressing the dependence on $\varoast$, we briefly write  $\beta=\beta(\varoast)$ and write the induced Poisson generating function, and exponential generating function of the length distribution, as
\[
P_k(z):= P_\varoast(z;l_k),\quad f_k(z) := e^z P_k(z) = f_{l_k}^\varoast(\sqrt{2z}).
\]
Now, Cor.~\ref{cor:gen_Poisson_expan} gives the expansion
\begin{equation}\label{eq:Pn_expan}
P_n(r) = F_\beta(t) + \sum_{j=1}^m F_{\beta,j}(t) (2r)^{-j/3} + O(r^{-(m+1)/3}) \bigg|_{t=t_{l_n^\varoast}(2r)},
\end{equation}
which is uniformly valid when $n-n^{3/5} \leq r \leq n + n^{3/5}$ as $n\to \infty$, $m$ being any fixed non-negative integer. Here, the implied constant in the error term depends only on $t_0,t_1$, but not on the specific sequence $l_n$. Preserving uniformity, the expansion can be repeatedly differentiated w.r.t. the variable $r$. In particular, using the differential equation \eqref{eq:tnu} we get that $P_n^{(j)}(n)$ expands in powers of $n^{-1/3}$, starting with a leading order term of the form
\begin{subequations}\label{eq:Pn_deri_expan}
\begin{equation}\label{eq:Pn_deri_est}
P_n^{(j)}(n) = (-1)^j 2^j F_\beta^{(j)}(t_n^*)\cdot (2n)^{-2j/3} + O(n^{-(2j+1)/3}) \qquad (n\to\infty);
\end{equation}
the first cases, insofar as they are needed for the Jasz expansion \eqref{eq:jasz_std_4}, are (cf. also~\cite[Eq.~(5.4)]{arxiv:2301.02022})
\begin{align}
P_n(n) &= F_\beta(t) + F_{\beta,1}(t) \cdot (2n)^{-1/3} + F_{\beta,2}(t) \cdot (2n)^{-2/3} \\*[1mm] 
&\qquad + F_{\beta,3}(t) \cdot (2n)^{-1} + O(n^{-4/3})\bigg|_{t=t_n^*}, \notag \\
P''_n(n) &= 4 F_\beta''(t) \cdot (2n)^{-4/3} + \left(\frac{10}{3} F_\beta'(t) +\frac{4t}{3} F_\beta''(t) + 4 F_{\beta,1}''(t)\right) \cdot (2n)^{-5/3} \\
&\hspace*{-1.2cm} + \left(\frac{7t}{9} F_\beta'(t) + 6 F_{\beta,1}'(t) + \frac{t^2}{9} F_{\beta}''(t) + \frac{4t}{3} F_{\beta,1}''(t) + 4 F_{\beta,2}''(t)\right) \cdot(2n)^{-2} + O(n^{-7/3})\bigg|_{t=t_n^*}, \notag \\
P'''_n(n) &= -8 F_\beta'''(t) \cdot (2n)^{-2} + O(n^{-7/3})\Big|_{t=t_n^*}, \\
P^{(4)}_n(n) &=  16F_\beta^{(4)}(t)  \cdot(2n)^{-8/3}\\
& \qquad + \left(80 F_\beta'''(t) +\frac{32t}{3} F_\beta^{(4)}(t) + 16 F_{\beta,1}^{(4)}(t)\right) \cdot (2n)^{-3}+ O(n^{-10/3})\bigg|_{t=t_n^*},\notag \\
P^{(6)}_n(n) &= 64 F_\beta^{(6)}(t) \cdot (2n)^{-4} + O(n^{-13/3})\Big|_{t=t_n^*},
\end{align}
\end{subequations}
where the implied constants in the error terms depend only on $t_0$ and $t_1$.

By Thm.~\ref{thm:gen_fun_zeros} we know that the $H$-admissible exponential generating functions $f_n(z)$ have only finitely many zeros in some sector $|\arg z\,| \leq \pi/2 + \epsilon$, $\epsilon > 0$. If we denote the auxiliary functions of $f_n(z)$ by $a_n(r)$ and $b_n(r)$, the expansion \eqref{eq:Pn_expan} and its derivatives give
\[
a_n(r) = r + O(r^{1/3}),\qquad b_n(r) = r + O(r^{2/3}),
\]
uniformly valid when $n-n^{3/5} \leq r \leq n + n^{3/5}$ as $n\to \infty$; the implied constants in the error terms depend only on $t_0,t_1$.

Therefore, as in the general permutation case \cite[§5.1]{arxiv:2301.02022}, we have now all the uniformity properties in place which are required for the application of \cite[Cor.~A.7]{arxiv:2301.02022}---with the sole exception of the uniform tameness 
(see \cite[Def.~A.2]{arxiv:2301.02022}) of the finitely many zeros of $f_n(z)$ in the sector $|\arg z\,| \leq \pi/2 + \epsilon$: these zeros should neither come too close to the positive real axis nor should they be getting too large. Since numerical experiments strongly indicate this to be true,\footnote{By Sects.~\ref{sect:ChazyI}--\ref{sect:exact}, the zeros of the generating functions correspond to the pole field of a Chazy I equation which is equivalent of a certain Painlevé III equation. Fornberg and Weideman \cite{MR2804960} developed a numerical method, the {\em pole field solver}, specifically for the task of numerically studying the pole fields of equations of the Painlevé class. They documented results for Painlevé I \cite{MR2804960}, Painlevé II \cite{MR3260257}, its imaginary variant \cite{MR3390079} and, together with Fasondini, for (multivalued) variants of the Painlevé III, V, and VI equations \cite{MR3656730, MR3724928}.} 
we will assume the validity of the following hypothesis.

\subsection*{Tameness hypothesis} For $\varoast = \boxslash,\boxbslash$, any real $t_0<t_1$ and any sequence of integers $l_n\to\infty$ satis\-fy\-ing~\eqref{eq:tnstar}, the zeros of the family $f_n(z)= f_{l_n}^\varoast(\sqrt{2z})$ of exponential generating functions are uniformly tame, with parameters and implied constants only depending on $t_0$ and $t_1$. 

\medskip

Subject to this hypothesis we prove the following theorem.

\begin{theorem}\label{thm:p_slash_cases_expan}
Let be $\varoast\in\{\boxslash,\boxbslash\}$, $t_0<t_1$ any ordered pair of real numbers and assume the tameness hypothesis. Then
there holds the expansion
\begin{equation}\label{eq:p_slash_cases_expan}
p_\varoast(n;l) = F_{\beta(\varoast)}(t) + \sum_{j=1}^m F_{\varoast,j}(t) \cdot (2n)^{-j/3} + O\big(n^{-(m+1)/3}\big)\bigg|_{t=t_{l^\varoast}(2n)},
\end{equation}
which is uniformly valid when $n,l \to \infty$ subject to $t_0 \leq t_{l^\varoast}(2n) \leq t_1$, with $m$ being any fixed non-negative integer. Here the $F_{\varoast,j}$ are certain smooth functions that have simple expressions in terms 
 of the functions $F_{\beta}$, $F_{\beta,j}$ as defined in \eqref{eq:PoissonVaroastExpan}; the first three are, with $\beta=\beta(\varoast)$,\footnote{To validate the formulae in (\ref{eq:FbetaD}), 
Fig.~\ref{fig:FbetaD} plots, for $\varoast=\boxslash,\boxbslash$, the function $F_{\varoast,3}(t)$ next to the approximation
\begin{equation}\label{eq:FbetaD3}
F_{\varoast,3}(t) \approx 2n\cdot  \Big( p_\varoast(n;l) - F_{\beta(\varoast)}(t) - F_{\varoast,1}(t)\cdot (2n)^{-1/3} - F_{\varoast,2}(t)\cdot (2n)^{-2/3}\Big)\,\Big|_{t=t_{l^\varoast}(2n)}
\end{equation}
for $n=250$, $n=500$, $n=1000$, varying the integer $l$ such that $t_{l^\varoast}(2n)$ covers the range of $t$ on display.}
\begin{subequations}\label{eq:FbetaD}
\begin{align}
F_{\varoast,1}(t) &= -\frac{t^2}{60} F_\beta'(t) - \frac{6}{5} F_\beta''(t)\label{eq:FbetaD1}\\
\intertext{and, subject to the linear form hypothesis,}
F_{\varoast,2}(t) &= \Big(-\frac{551}{700} + \frac{2t^3}{1575}\Big) F_\beta'(t) + \Big(-\frac{43t}{175} + \frac{t^4}{7200}\Big) F_\beta''(t) + \frac{t^2}{50} F_\beta'''(t) + \frac{18}{25} F_\beta^{(4)}(t),\label{eq:FbetaD2}\\*[2mm]
F_{\varoast,3}(t) &=  -\Big(\frac{1144 t}{7875} + \frac{41 t^4}{283500}\Big)  F_\beta'(t)
+\Big( \frac{11t^2}{1680}  - \frac{t^5}{47250} \Big) F_\beta''(t)  \\*[1mm]
&\quad +\Big(\frac{20413}{15750} +\frac{9 t^3}{3500} - \frac{t^6}{1296000}\Big) F_\beta'''(t)
 +\Big(\frac{258t}{875} -\frac{t^4}{6000} \Big)F_\beta^{(4)}(t)\notag\\*[1mm]
&\quad -\frac{3t^2 }{250}F_\beta^{(5)}(t)
-\frac{36}{125} F_\beta^{(6)}(t).\notag
\end{align}
\end{subequations}
\end{theorem}

\begin{figure}[tbp]
\includegraphics[width=0.325\textwidth]{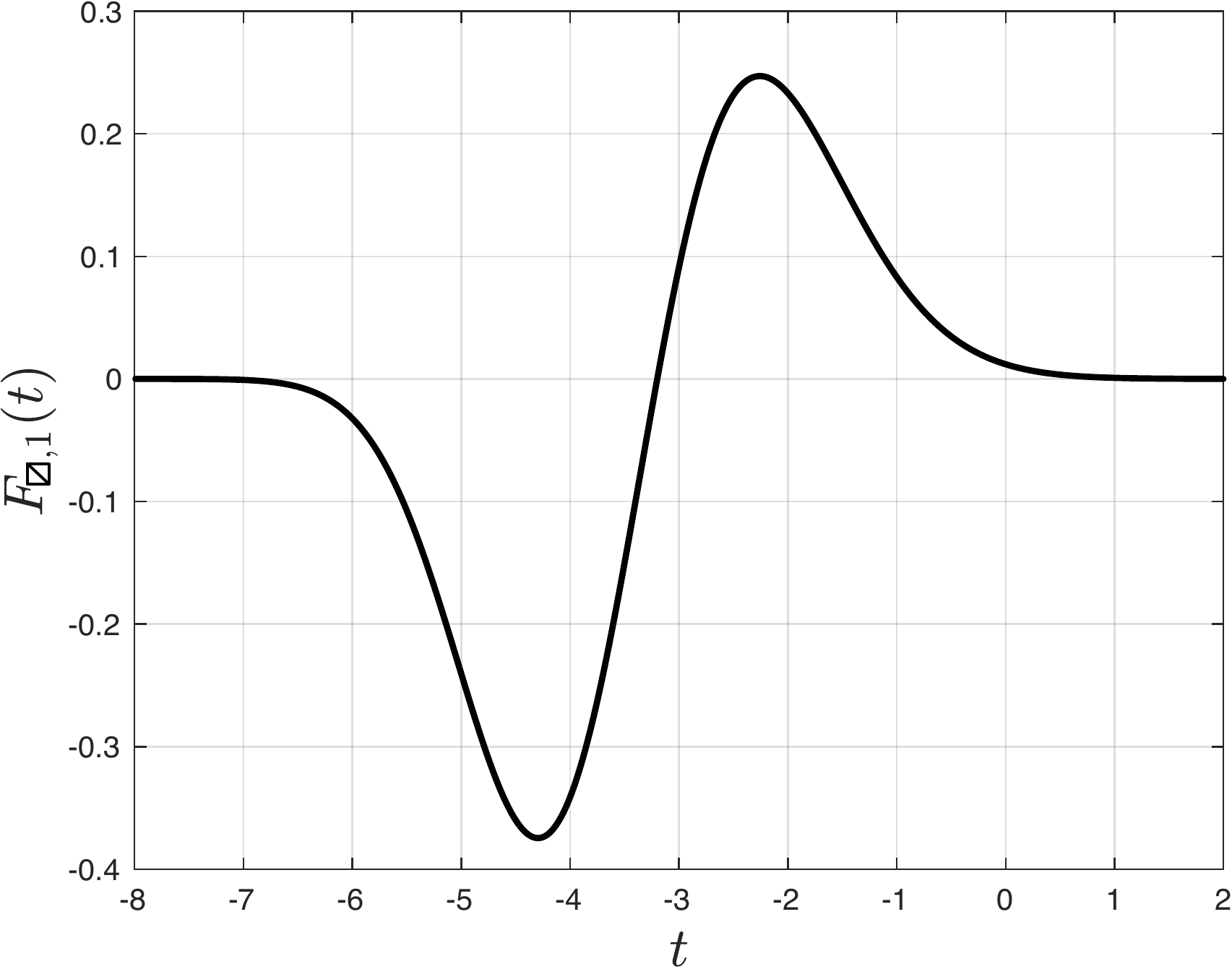}\hfil\,
\includegraphics[width=0.325\textwidth]{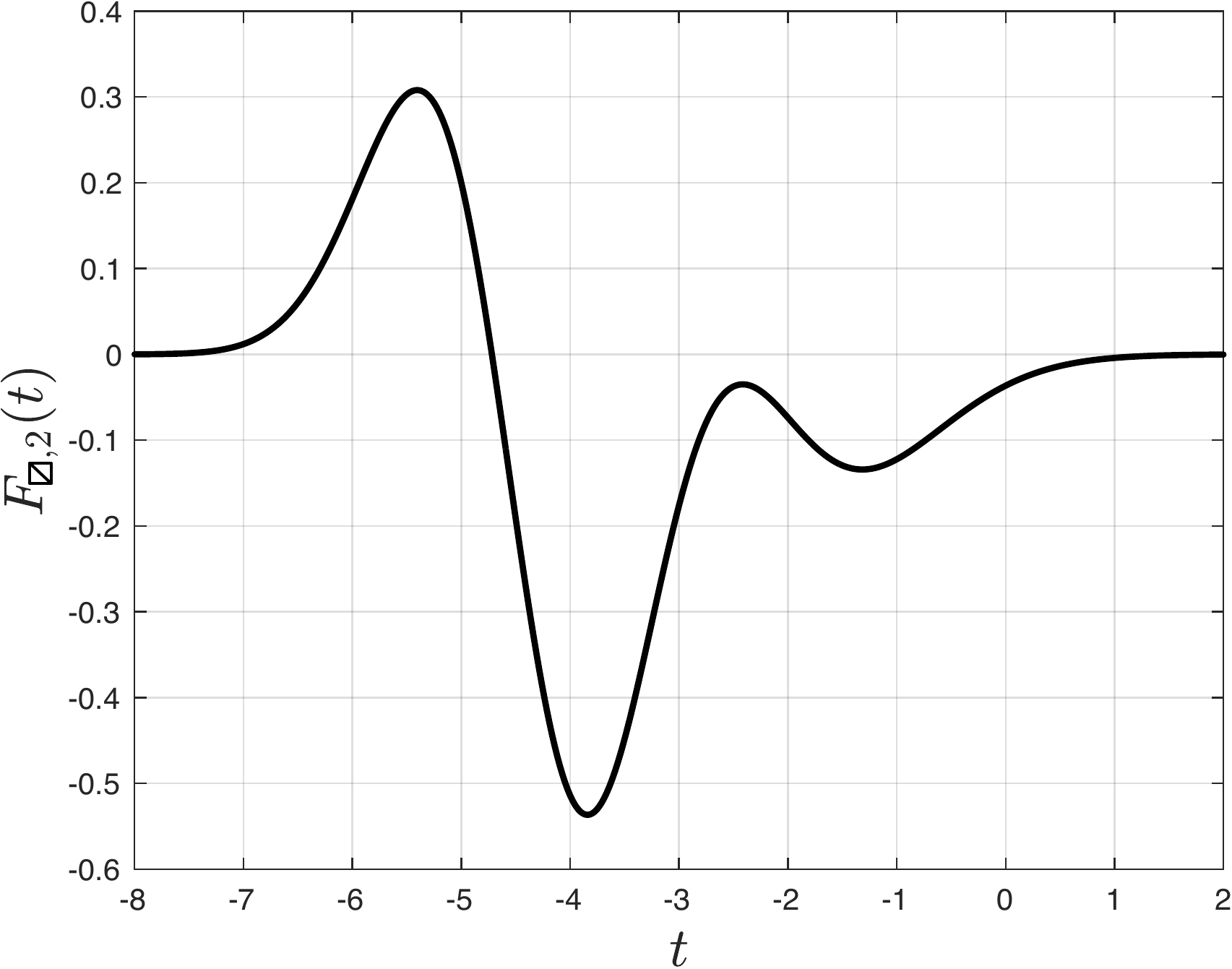}\hfil
\includegraphics[width=0.325\textwidth]{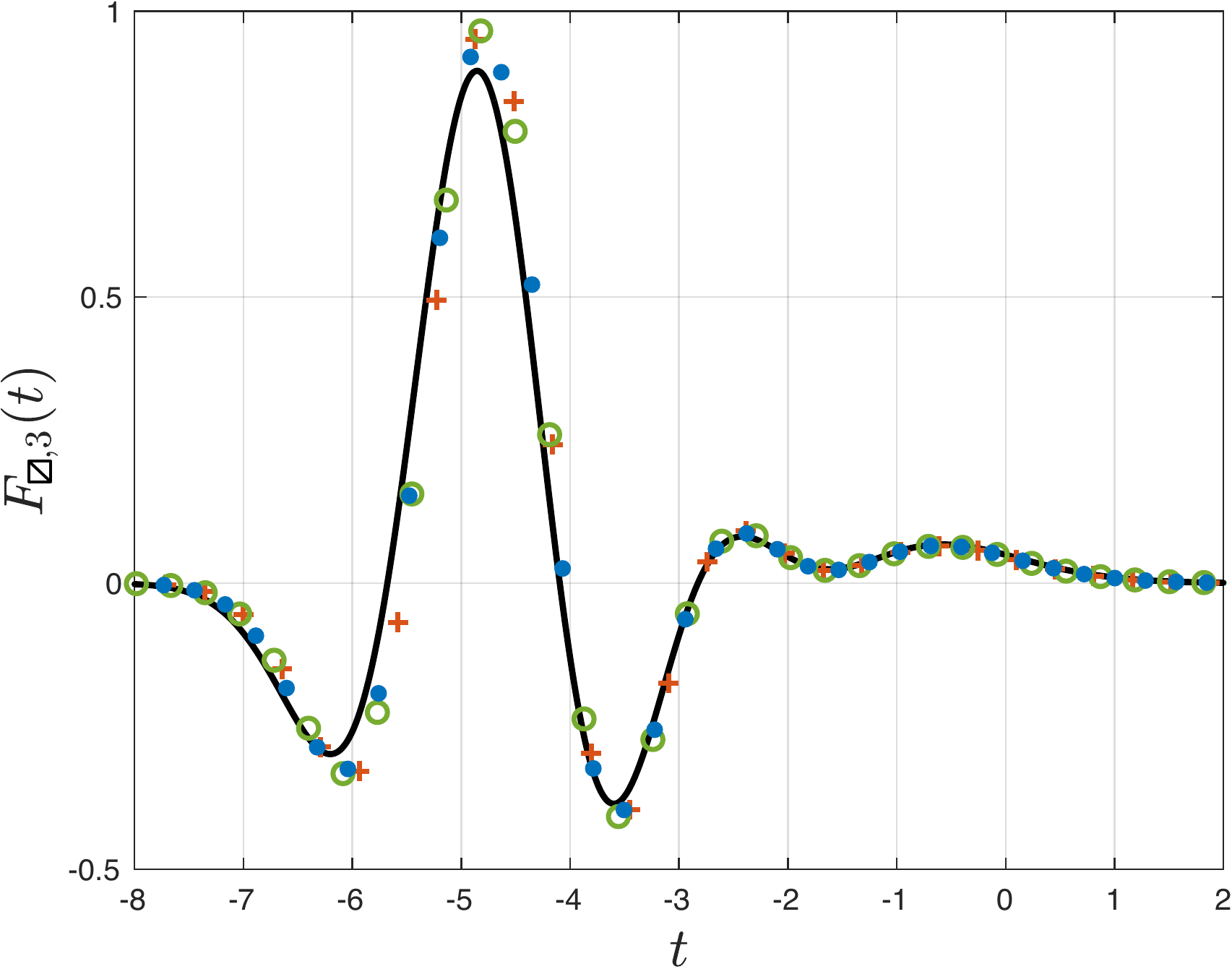}\\
\includegraphics[width=0.325\textwidth]{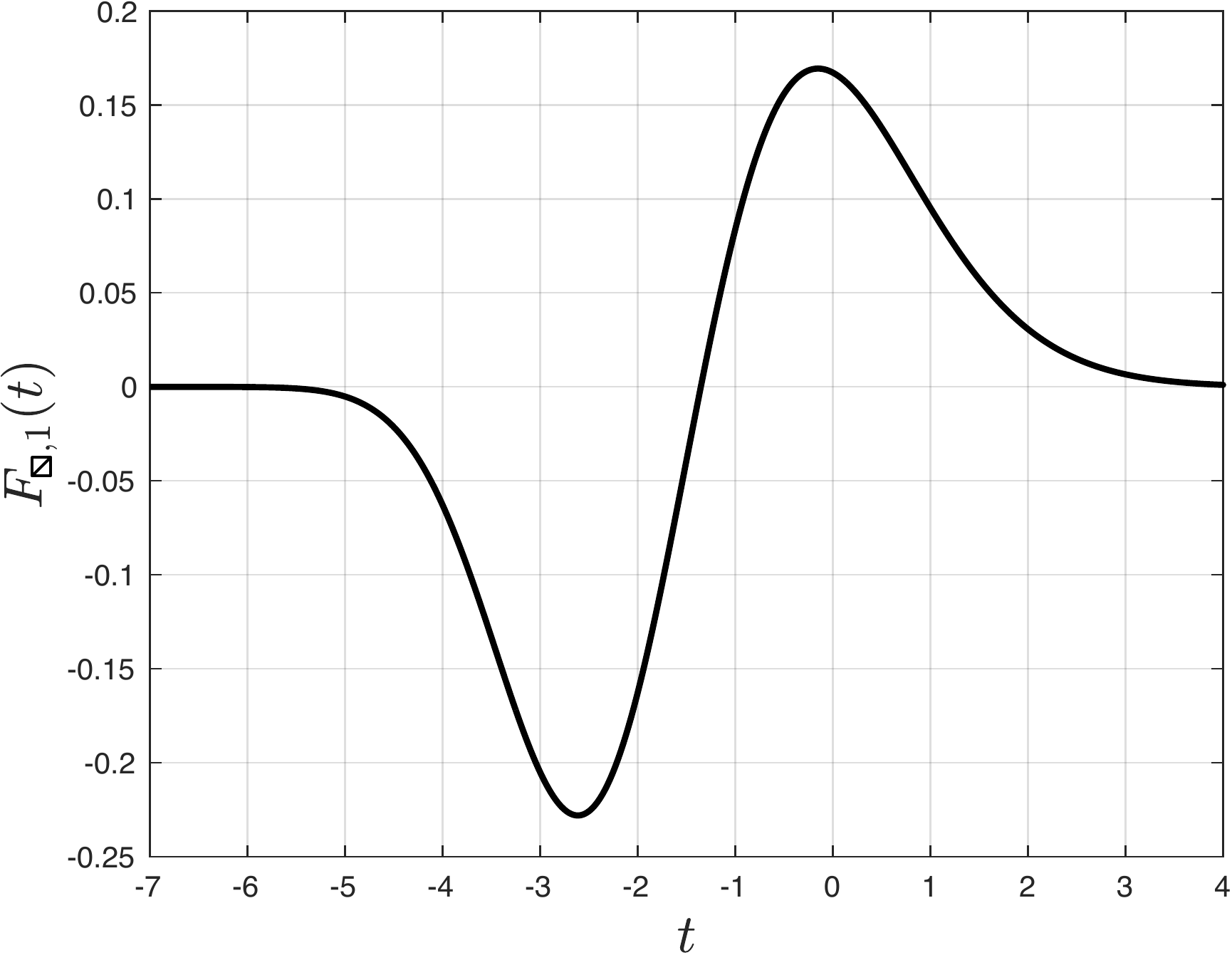}\hfil\,
\includegraphics[width=0.325\textwidth]{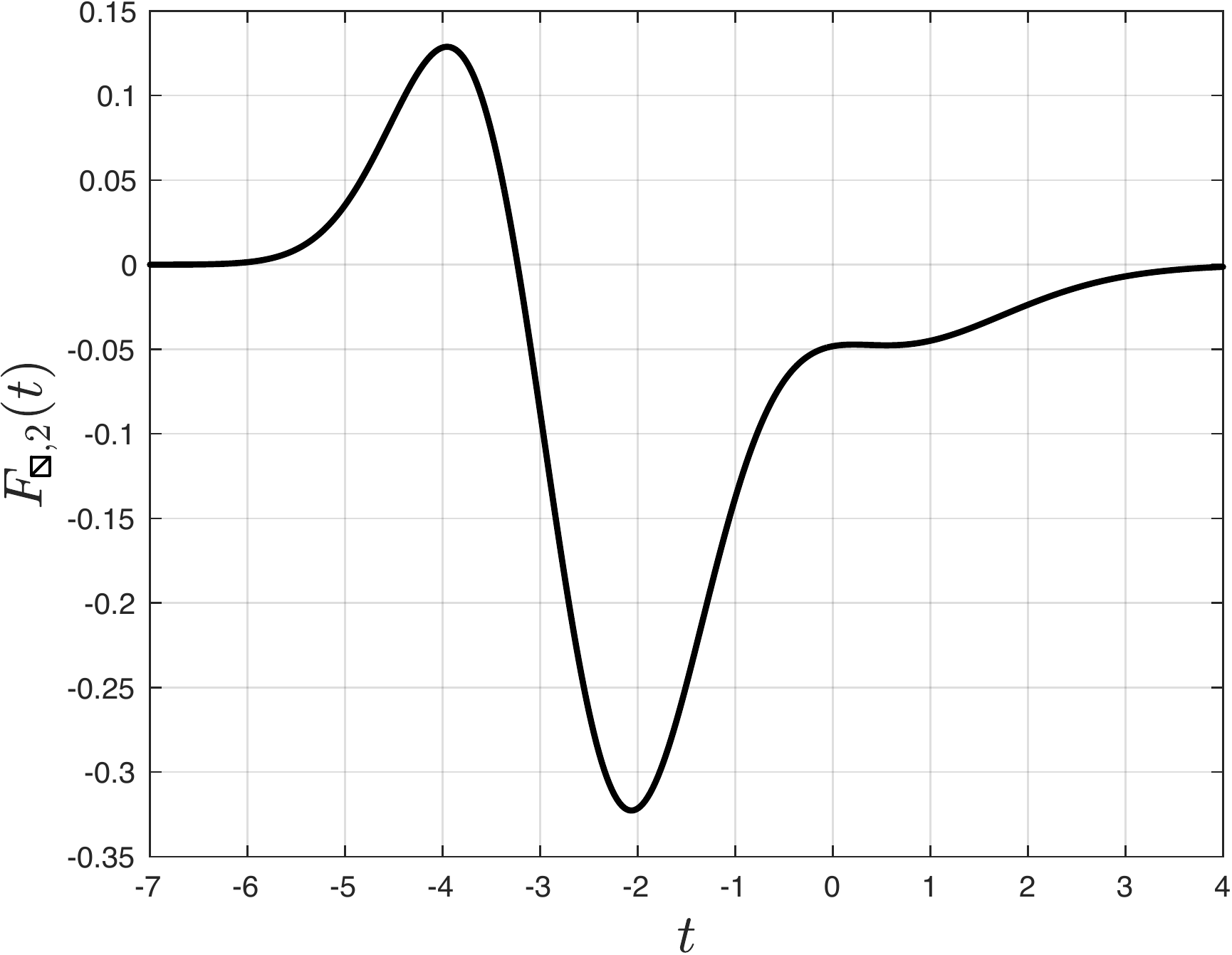}\hfil\,
\includegraphics[width=0.325\textwidth]{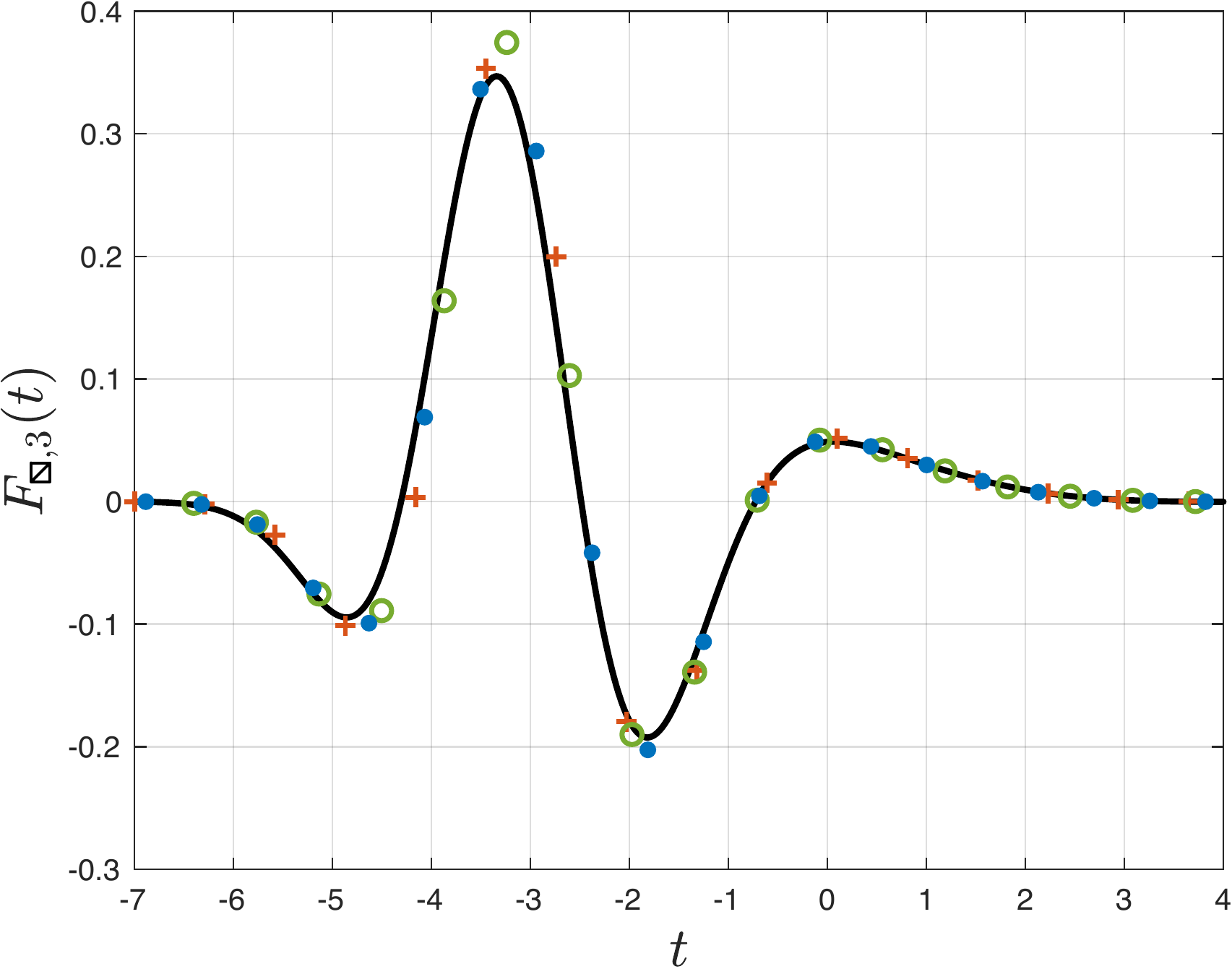}
\caption{{\footnotesize Top row $\varoast=\boxslash$; bottom row $\varoast=\boxbslash$. Plots of $F_{\varoast,1}(t)$ (left panels; both agree with the simulation-based approximation of their graphical form in \cite[Fig.~12]{arxiv.2205.05257}) and $F_{\varoast,2}(t)$ (middle panels) as in \eqref{eq:FbetaD}.
The right panels show $F_{\varoast,3}(t)$ as displayed in \eqref{eq:FbetaD} (black solid line) next to the approximations \eqref{eq:FbetaD3} for $n=250$ (red $+$), $n=500$ (green $\circ$) and $n=1000$ (blue $\bullet$); the integer $l$ has been varied such that $t_{l^\varoast}(2n)$ covers the range of $t$ on display. The evaluation of \eqref{eq:FbetaD3} uses a table of exact values of $p_\varoast(n;l)$ up to $n=1000$ (see~Sect.~\ref{sect:exact}).}}
\label{fig:FbetaD}
\end{figure}

\begin{proof} Given the preparations preceding the formulation of the theorem, the proof of \eqref{eq:p_slash_cases_expan} follows literally the one of \cite[Cor.~A.7]{arxiv:2301.02022} for the general permutation case. In particular, the Jasz expansion \eqref{eq:jasz_std_4} is now rigorously established (it applies because of \eqref{eq:Pn_deri_est}) and gives 
\[
p_\varoast(n;l_n) = P_n(n) - \frac{n}{2} P_n''(n) + \frac{n}{3} P_n'''(n) + \frac{n^2}{8} P_n^{(4)}(n) - \frac{n^3}{48}P_n^{(6)}(n) + O(n^{-4/3}).
\]
Inserting the expansions of the derivatives displayed in \eqref{eq:Pn_deri_expan} yields (cf. also \cite[Eq.~(5.10)]{arxiv:2301.02022})
\begin{align*}
F_{\varoast,1}(t) &= F_{\beta,1}(t) - F_\beta''(t), \\*[1mm]
F_{\varoast,2}(t) &= F_{\beta,2}(t) - \frac{5}{6} F_\beta'(t) - \frac{t}{3}F_{\beta}''(t) - F_{\beta,1}''(t)  + \frac{1}{2} F_\beta^{(4)}(t), \\*[1mm]
F_{\varoast,3}(t) &= F_{\beta,3}(t) -\frac{7t}{36} F_\beta'(t) - \frac{3}{2}F_{\beta,1}'(t) -\frac{t^2}{36}F_{\beta}''(t) - \frac{t}{3}F_{\beta,1}''(t) - F_{\beta,2}''(t) \\*[1mm]
&\qquad + \frac{7}{6} F_\beta'''(t) +
\frac{t}{3} F_\beta^{(4)}(t) + \frac{1}{2} F_{\beta,1}^{(4)}(t) - \frac{1}{6} F_\beta^{(6)}(t).
\end{align*}
Finally, using the expressions of the functions $F_{\beta,j}$ as displayed in \eqref{eq:FbetaP} gives \eqref{eq:FbetaD}.
\end{proof}

\subsection{Expansions of distributions: the general involution case} For the time being, in contrast to the case of the standard de-Poissonization, we lack in the general case rigorous analytic results such as the theorem of Jacquet and Szpankowski \cite[Thm.~A.2]{arxiv:2301.02022}. Therefore, to discuss $\varoast=\boxdot$, we content ourselves with a formal asymptotic expansion.

As in Sect.~\ref{sect:jasz_gen}, let us choose $r_n^* = n^{1/2} - \frac12$. For any fixed compact interval $[t_0,t_1]$ we then consider a sequence $l_n\to \infty$ such that
\[
t_0 \leq t_n^* = t_{l_n}((r_n^*)^2) \leq t_1,
\]
and Cor.~\ref{cor:gen_Poisson_expan} gives the expansion, for $r$ sufficiently close to $r_n^*$,
\begin{equation}
P_n(r) := F_1(t) + \sum_{j=1}^m F_{1,j}(t) \cdot r^{-2j/3} + O(r^{-2(m+1)/3}) \bigg|_{t = t_{l_n}(r^2)}.
\end{equation}
This expansion can be repeatedly differentiated w.r.t. the variable $r$. In particular, using the differential equation \eqref{eq:tnu} we get that $P_n^{(j)}(n)$ expands in powers of $n^{-1/6}$, starting with a leading order term of the form
\[
P_n^{(j)}(r_n^*) = (-1)^j 2^j F_1^{(j)}(t_n^*) \cdot n^{-j/6} + O(n^{-(j+1)/6})\qquad (n\to\infty);
\]
the first specific cases, insofar as they are needed for the  Jasz expansion \eqref{eq:jasz_gen_4}, are 
\begin{align*}
P_n(r_n^*) &= F_1(t) + F_{1,1}(t) \cdot n^{-1/3} + F_{1,2}(t) \cdot n^{-2/3} + \frac{1}{3}F_{1,1}(t) \cdot n^{-5/6}
+ F_{1,3}(t) \cdot n^{-1}\\
&\qquad + \frac{2}{3}F_{1,2}(t) \cdot n^{-7/6} + O(n^{-4/3})\bigg|_{t=t_n^*}, \\
P'_n(r_n^*) &= -2 F_1'(t) \cdot n^{-1/6} - \Big(\frac{t}{3} F_1'(t) + 2 F_{1,1}'(t) \Big)  \cdot n^{-1/2} - \frac{1}{3} F_1'(t) \cdot n^{-2/3} + O(n^{-5/6})\bigg|_{t=t_n^*},\\
\intertext{}
P''_n(r_n^*) &= 4 F_1''(t) \cdot n^{-1/3} + \Big(\frac{4}{3} F_1'(t) +\frac{4t}{3} F_1''(t) + 4 F_{1,1}''(t) \Big)  \cdot n^{-2/3} 
+\frac{4}{3} F_1''(t) \cdot n^{-5/6}\\
&\qquad + \Big(\frac{4t}{9} F_1'(t) + 4 F_{1,1}'(t) + \frac{t^2}{9} F_{1}''(t) + \frac{4t}{3} F_{1,1}''(t) + 4 F_{1,2}''(t) \Big) \cdot n^{-1}\\
&\qquad\qquad + \Big(\frac{8}{9} F_1'(t) + \frac{8t}{9}F_1''(t) + \frac{8}{3}F_{1,1}''(t) \Big) \cdot n^{-7/6} + O(n^{-4/3})\bigg|_{t=t_n^*}, \\
P'''_n(r_n^*) &= -8 F_1'''(t) \cdot n^{-1/2} + O(n^{-5/6})\bigg|_{t=t_n^*}\\
P^{(4)}_n(r_n^*) &=  16F_1^{(4)}(t) \cdot n^{-2/3}+ \Big(32 F_1'''(t) +\frac{32t}{3} F_1^{(4)}(t) + 16 F_{1,1}^{(4)}(t) \Big)  \cdot n^{-1}\\ 
& \qquad + \frac{32}{3}F_1^{(4)}(t) \cdot n^{-7/6} + O(n^{-4/3})\bigg|_{t=t_n^*}, \\
P^{(6)}_n(r_n^*) &= 64 F_1^{(6)}(t) \cdot n^{-1} + O(n^{-4/3}).
\end{align*}
If we plug these expansions into the generalized Jasz expansion \eqref{eq:jasz_gen_4}, we get
\begin{multline*}
p_\boxdot(n;l_n) = F_1(t) + (F_{1,1}(t) - F_1''(t))n^{-1/3} \\
+ \Big( F_{1,2}(t) - \frac{13}{12} F_1'(t) - \frac{t}{3}F_1''(t) - F_{1,1}''(t) + \frac{1}{2}F_1^{(4)}(t) \Big)n^{-2/3} + \Big(\frac{1}{3}F_{1,1}(t) + \frac{1}{6}F_1''(t) \Big) n^{-5/6}\\
+ \Big( F_{1,3}(t) - \frac{17t}{72} F_1'(t) - \frac{7}{4}F_{1,1}'(t) - \frac{t^2}{36}F_1''(t) - \frac{t}{3}F_{1,1}''(t) - F_{1,2}''(t) \\
 + \frac{17}{12}F_1'''(t)+ \frac{t}{3}F_1^{(4)}(t) +\frac{1}{2}F_{1,1}^{(4)}(t)-\frac{1}{6}F_1^{(6)}(t) \Big)n^{-1}
\\
+ \Big(\frac{2}{3}F_{1,2}(t) + \frac{5}{72} F_1'(t) - \frac{t}{18}F_1''(t) -  \frac{1}{6}F_{1,1}''(t)-\frac{1}{4}F_1^{(4)} \Big)n^{-7/6} + O(n^{-4/3}) \bigg|_{t=t_n^*}.
\end{multline*}
Finally, after inserting the expressions of the functions $F_{1,j}$ as displayed in \eqref{eq:FbetaP} and using
\[
t_l((r_n^*)^2) = t_{l+1}(n) \cdot \big(1-n^{-1/2}/2\big)^{-1/3} = t_{l+1}(n) \cdot \Big(1+ \frac{1}{6}n^{-1/2} + \frac{1}{18}n^{-1} + O(n^{-4/3})\Big)
\]
we are led to the following very specific conjecture.

\begin{figure}[tbp]
\includegraphics[width=0.375\textwidth]{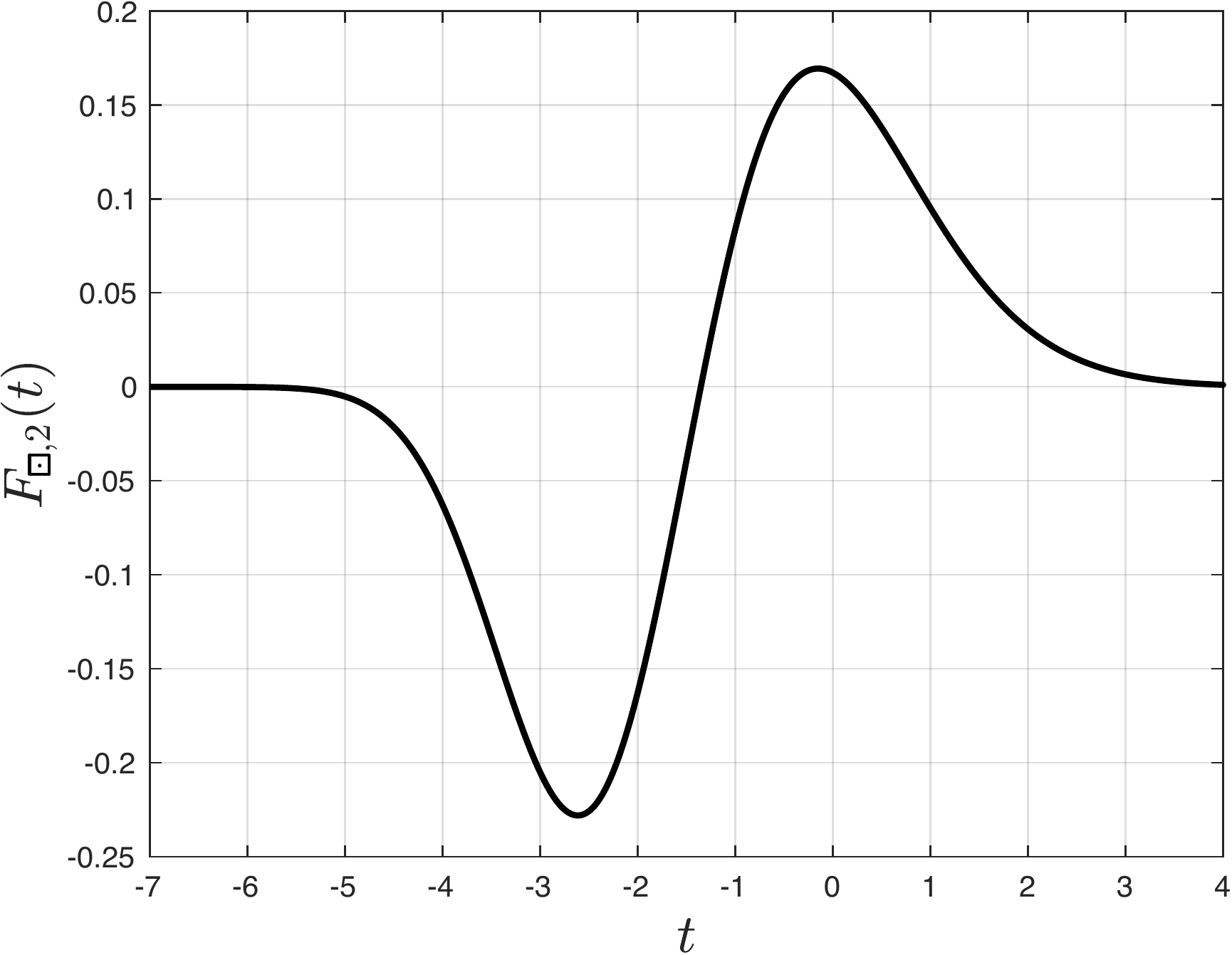}\hfil\,
\includegraphics[width=0.375\textwidth]{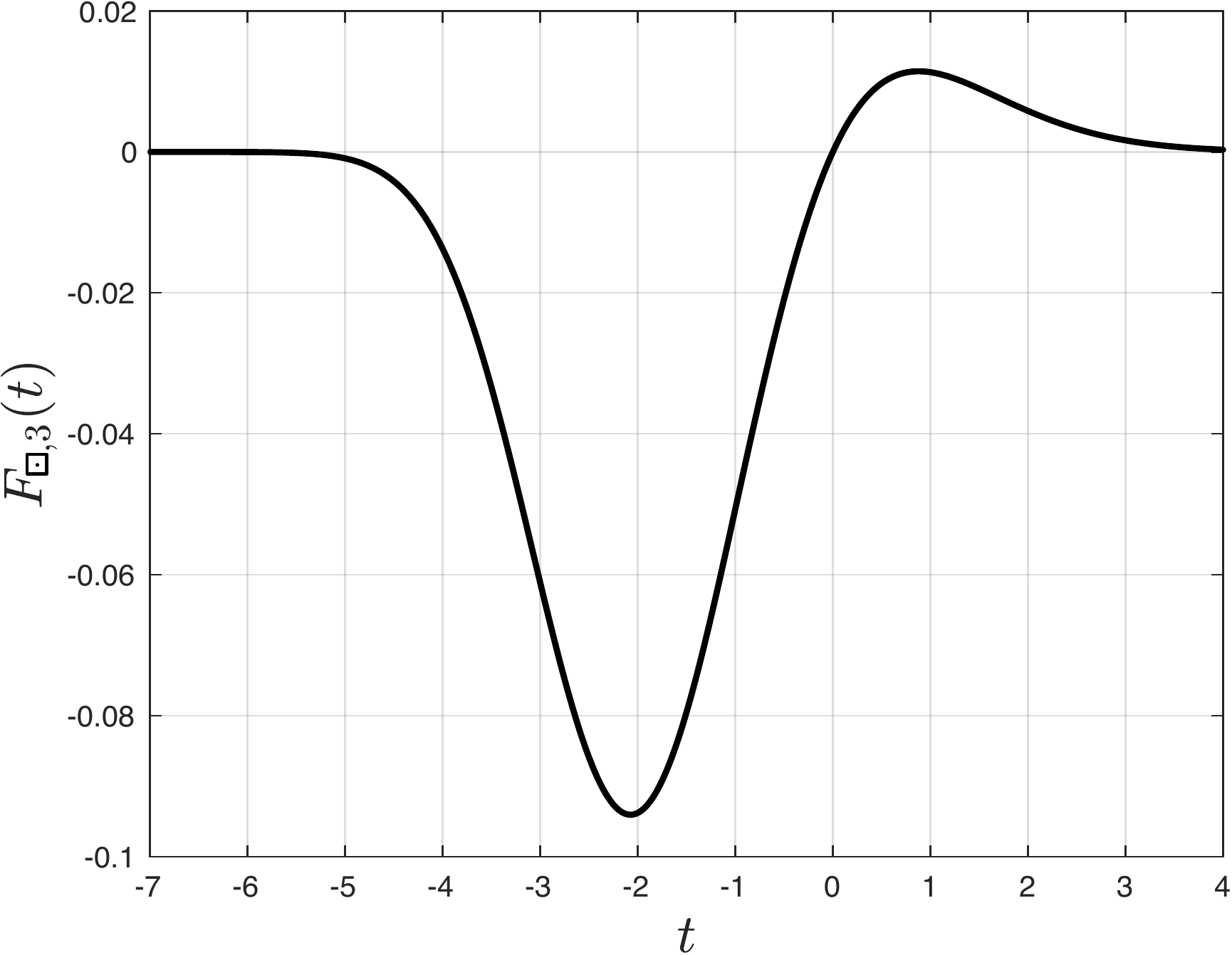}\\
\includegraphics[width=0.375\textwidth]{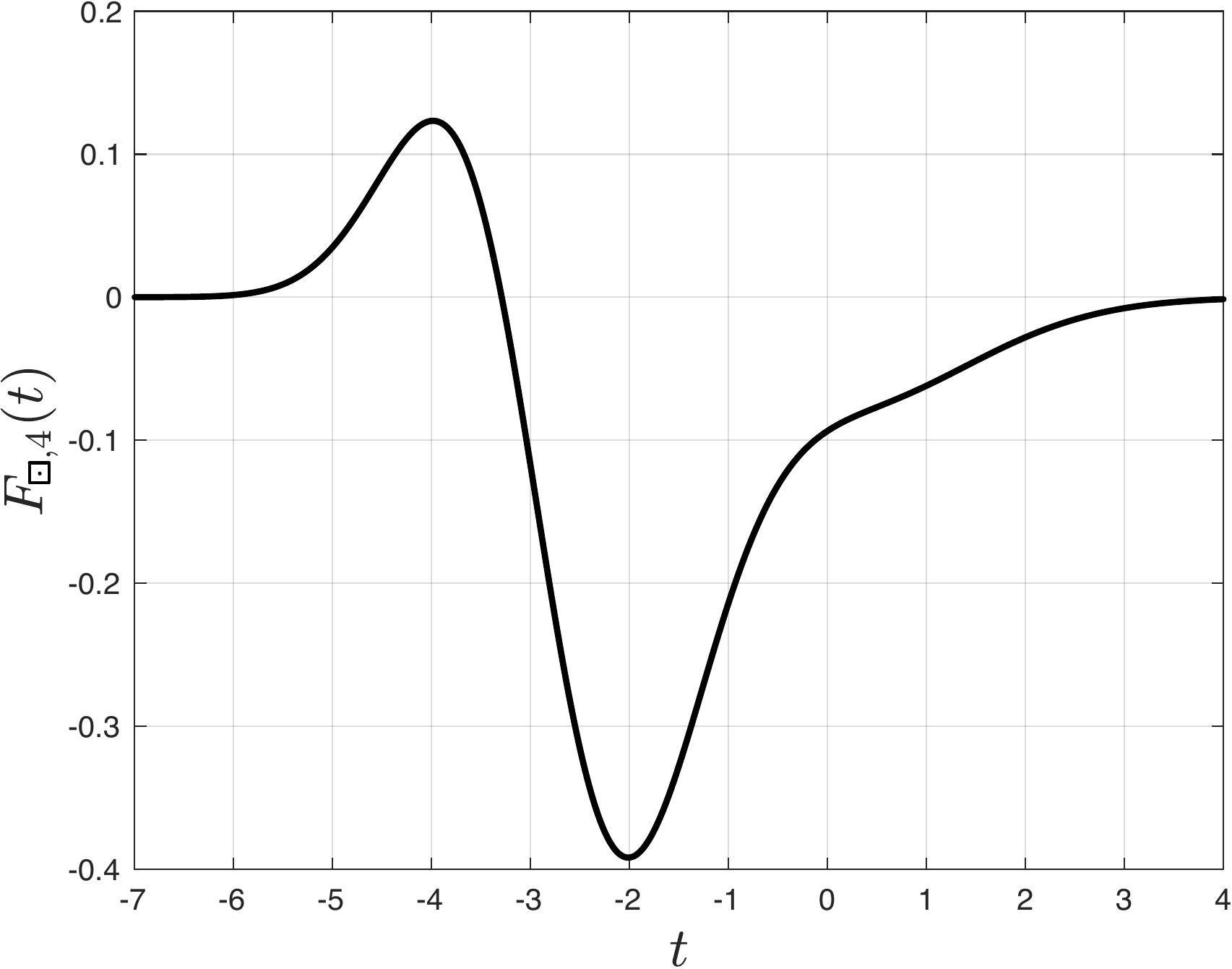}\hfil\,
\includegraphics[width=0.38\textwidth]{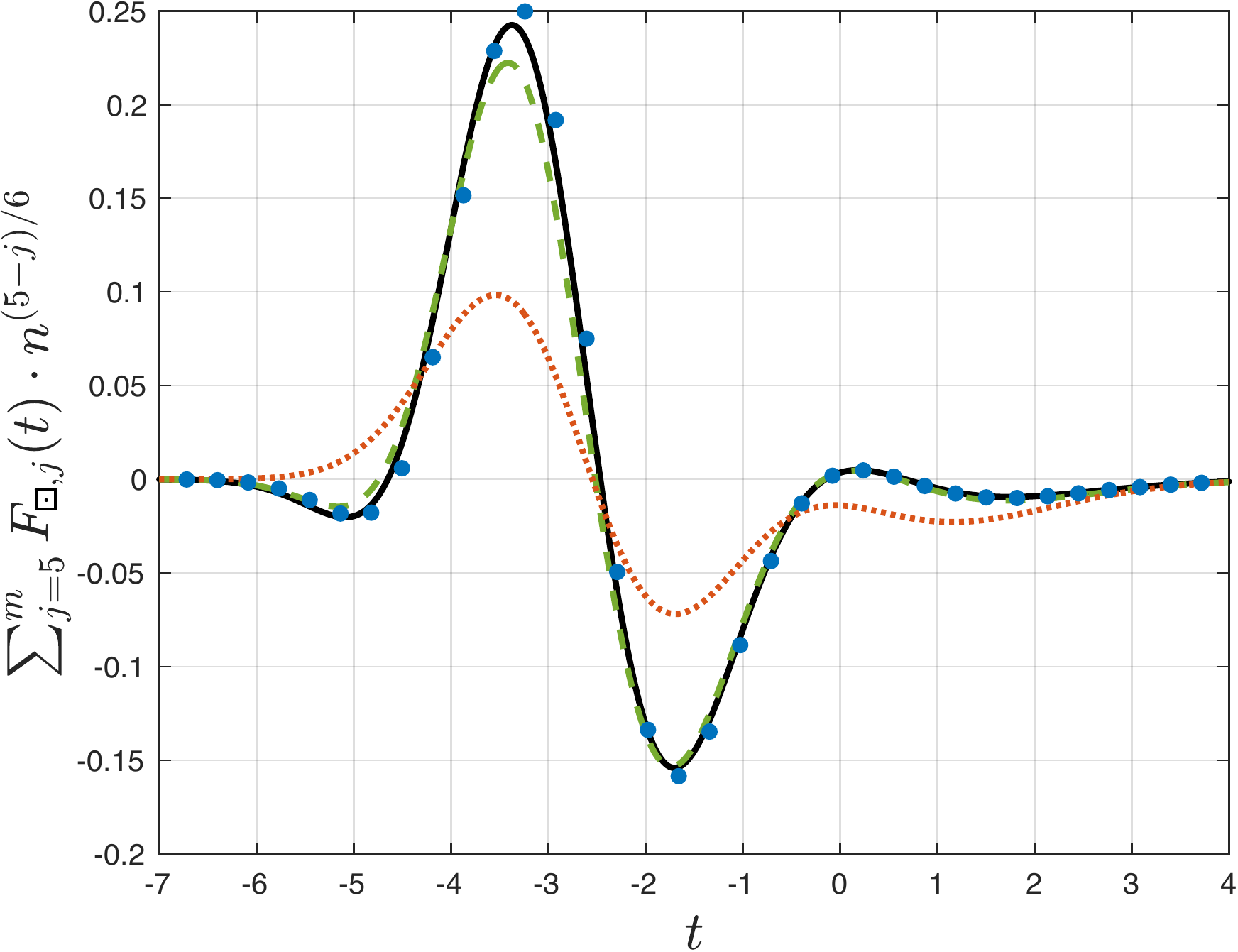}
\caption{\footnotesize First three panels: plots of $F_{\boxdot,j}(t)$, $j=2,3,4$. Last panel: plot of $\sum_{j=5}^m F_{\boxdot,j}(t) \cdot n^{(5-j)/6}$ for $m=5$ (dotted red line), $m=6$ (dashed green line), $m=7$ (solid black line) vs. the right rand side of \eqref{eq:Fdot5}
 for $n=1000$ (blue $\bullet$); the integer $l$ has been varied such that $t_{l+1}(n)$ covers the range of $t$ on display. The evaluation of \eqref{eq:Fdot5} uses a table of exact values of the probabilities $p_\boxdot(n;l)$ for $n=1000$ (see~Sect.~\ref{sect:exact}). The choice of $m=7$ in~\eqref{eq:Fdot5} uses all the expressions displayed in \eqref{eq:Fdot} and exhibits an excellent agreement (whereas $m=5$ is insufficient, reflecting that $n^{-1/6}\approx 0.316$ is a comparatively large quantity here).}
\label{fig:Fdot}
\end{figure}%

\begin{conjecture}\label{conj:p_dot_expan}
Let be $t_0<t_1$ any ordered pair of real numbers. Then there holds
\begin{equation}\label{eq:p_dot_expan}
p_\boxdot(n;l) = F_1(t) + \sum_{j=1}^m F_{\boxdot,j}(t) \cdot n^{-j/6} + O\big(n^{-(m+1)/6}\big)\bigg|_{t=t_{l+1}(n)},
\end{equation}
which is uniformly valid when $n,l \to \infty$ subject to $t_0 \leq t_{l+1}(n) \leq t_1$, with $m$ being any fixed non-negative integer. Here the $F_{\boxdot,j}$ are certain smooth functions that have simple expressions in terms 
 of the functions $F_1$, $F_{1,j}$ in \eqref{eq:PoissonVaroastExpan}. The first seven of them are $F_{\boxdot,1}(t) = 0$ and
\begin{subequations}\label{eq:Fdot}
\begin{align}
F_{\boxdot,2}(t) &= -\frac{t^2}{60} F_1'(t) - \frac{6}{5} F_1''(t),\\*[1mm]
F_{\boxdot,3}(t) &= \frac{t}{6} F_1'(t),\\*[1mm]
F_{\boxdot,4}(t) &= \Big(-\frac{363}{350} + \frac{2t^3}{1575}\Big) F_1'(t) + \Big(-\frac{43t}{175} + \frac{t^4}{7200}\Big) F_1''(t) + \frac{t^2}{50} F_1'''(t) + \frac{18}{25} F_1^{(4)}(t),\\*[1mm]
F_{\boxdot,5}(t) &= -\frac{t^2}{90}F_1'(t) + \Big(\frac{1}{10}- \frac{t^3}{360} \Big) F_1''(t) - \frac{t}{5}F_1'''(t),\\*[1mm]
\intertext{}
F_{\boxdot,6}(t) &=  -\Big(\frac{323 t}{2625} + \frac{41 t^4}{283500}\Big)  F_1'(t)
+\Big( \frac{31t^2}{1260}  - \frac{t^5}{47250} \Big) F_1''(t)  \\*[1mm]
&\quad +\Big(\frac{12569}{7875} +\frac{9 t^3}{3500} - \frac{t^6}{1296000}\Big) F_1'''(t)
 +\Big(\frac{258t}{875} -\frac{t^4}{6000} \Big)F_1^{(4)}(t)\notag\\*[1mm]
&\quad -\frac{3t^2 }{250}F_1^{(5)}(t)
-\frac{36}{125} F_1^{(6)}(t),\notag\\*[1mm]
F_{\boxdot,7}(t) &= \Big(\frac{117}{1400} + \frac{t^3}{675}\Big) F_1'(t) + \Big(-\frac{171t}{700}+ \frac{t^4}{2520}\Big)F_1''(t) \\*[2mm]
&\quad +\Big(-\frac{41t^2}{1400}+\frac{t^5}{43200}\Big)F_1'''(t) + \Big(- \frac{3}{25} + \frac{t^3}{300}\Big) F_1^{(4)}(t) + \frac{3t}{25} F_1^{(5)}(t).\notag
\end{align}
\end{subequations}
\end{conjecture}

To validate the intricate formulae of Conjecture~\ref{conj:p_dot_expan} as displayed in (\ref{eq:Fdot}), we have to exercise some care: though Sect.~\ref{sect:exact} provides us with a table of the exact values of $p_\boxdot(n;l)$ for $n=1000$, the different powers of the comparatively large quantity $n^{-1/6}\approx 0.316$ can barely differentiate the corresponding terms of the expansion \eqref{eq:p_dot_expan}. We therefore
plot in Fig.~\ref{fig:Fdot} both sides of the approximative relation
\begin{equation}\label{eq:Fdot5}
\sum_{j=5}^{m} F_{\boxdot,j}(t) \cdot n^{(5-j)/6} \approx n^{5/6} \Big(p_\varoast(n;l) - F_1(t) - \sum_{j=1}^4 F_{\boxdot,j} n^{-j/6}\Big)\Bigg|_{t=t_{l+1}(n)}
\end{equation}
for $m=5,6,7$ and $n=1000$, varying the integer $l$ such that $t_{l+1}(n)$ covers the range of $t$ on display.
The choice $m=7$ in \eqref{eq:Fdot5}, which uses all the expressions displayed in \eqref{eq:Fdot}, exhibits an excellent agreement. This provides a strong testament for the correctness of the conjecture.

\begin{remark}\label{rem:onesixth} Note that the expansion \eqref{eq:p_dot_expan} is evaluated at $t=t_{l+1}(n)$, which lets the $n^{-1/6}$ term vanish and the expansion to start with
\[
p_\boxdot(n;l) = F_1(t)  - \Big(\frac{t^2}{60} F_1'(t) + \frac{6}{5} F_1''(t) \Big)n^{-1/3} + O(n^{-1/2})\bigg|_{t=t_{l+1}(n)}.
\]
In contrast, if we evaluated at $t=t_{l}(n) = t_{l+1}(n) - n^{-1/6}$ as in Thm.~\ref{thm:leading_de_poisson}, the expansion would start with
\[
p_\boxdot(n;l) = F_1(t)  + F_1'(t) n^{-1/6} - \Big(\frac{t^2}{60} F_1'(t) + \frac{7}{10} F_1''(t) \Big)n^{-1/3} + O(n^{-1/2})\bigg|_{t=t_{l}(n)},
\]
which shows that in the case $\varoast=\boxdot$ the error estimate in \eqref{eq:leading_de_poisson} is suboptimal only up to the sublogarithmic factor $\sqrt{\log n}$.
\end{remark}

Upon observing $F_{\boxdot,1}(t)=0$ and $F_{\boxdot,2}(t) = F_{\boxbslash,1}(t)$ we can combine the first order terms in Thm.~\ref{thm:p_slash_cases_expan}
and Conjecture~\ref{conj:p_dot_expan} into the following strengthening of Thm.~\ref{thm:leading_de_poisson}. To this end we modify \eqref{eq:lr_varoast}, writing
\begin{equation}\label{eq:l_varoast_plus}
l^\varoast_* := 
\begin{cases}
l-1,  & \varoast = \boxslash,\\*[0.5mm]
2l+1, & \varoast = \boxbslash,\\*[0.5mm]
l+1,  & \varoast = \boxdot.
\end{cases}
\end{equation}

\begin{corollary}\label{cor:limit_law_finite_size} Subject to the tameness hypothesis for $\varoast=\boxslash,\boxbslash$ and to Conjecture~{\em\ref{conj:p_dot_expan}} for $\varoast=\boxdot$, the discrete probability distributions satisfy the Baik–Rains limit laws with a first order finite-size correction and optimal convergence rate of the form
\[
p_\varoast(n;l) = F_{\beta}(t) - \Big(\frac{t^2}{60} F_\beta'(t) + \frac{6}{5} F_\beta''(t) \Big)(\gamma n)^{-1/3} + O\big(n^{-(1+\gamma/2)/3}\big)\bigg|_{t = t_{l_*^\varoast}(\gamma n), \beta=\beta(\varoast), \gamma=\gamma(\varoast)},
\]
which is uniformly valid when $n, l \to \infty$ while $t$ stays bounded.
\end{corollary}

\subsection{Expansions of the discrete probability densities} Thm.~\ref{thm:p_slash_cases_expan} and Conjecture~\ref{conj:p_dot_expan} give expansions of 
the discrete probabilities \eqref{eq:prob_dens}, if written as differences which we cast in the joint form (with $\gamma=\gamma(\varoast)$)
\begin{equation}\label{eq:difference}
\begin{aligned}
p^*_\varoast(n;l) & = p_\varoast(n;l) - p_\varoast(n;l-1)\\*[1mm]
& =  F_{\beta(\varoast)}(t) + \sum_{j=1}^m F_{\varoast,j}(t) \cdot (\gamma n)^{-\gamma j/6} + O\big(n^{-\gamma(m+1)/6}\big)\bigg|_{t=t_{(l-1)_*^\varoast}(\gamma n)}^{t=t_{l_*^\varoast}(\gamma n)}.
\end{aligned}
\end{equation}
The differences 
can be further expanded by applying the central differencing formula (which is, basically, just a Taylor expansion for smooth $F$ centered at the midpoint)
\begin{equation}\label{eq:central_diff}
F(t+h) - F(t) = h F'(t+h/2) + \frac{h^3}{24}F'''(t+h/2) +  \frac{h^5}{1920}F^{(5)}(t+h/2) + \cdots.
\end{equation}
The increments  in \eqref{eq:difference} are $h = h_\varoast:= t_{l_*^\varoast}(\gamma(\varoast) n)-t_{(l-1)_*^\varoast} (\gamma(\varoast) n)$,
that is,
\begin{equation}
h_\boxslash = (2n)^{-1/6},\qquad h_\boxbslash = 2(2n)^{-1/6},\qquad h_\boxdot = n^{-1/6}.
\end{equation} 
We thus get the following result.

\begin{figure}[tbp]
\includegraphics[width=0.325\textwidth]{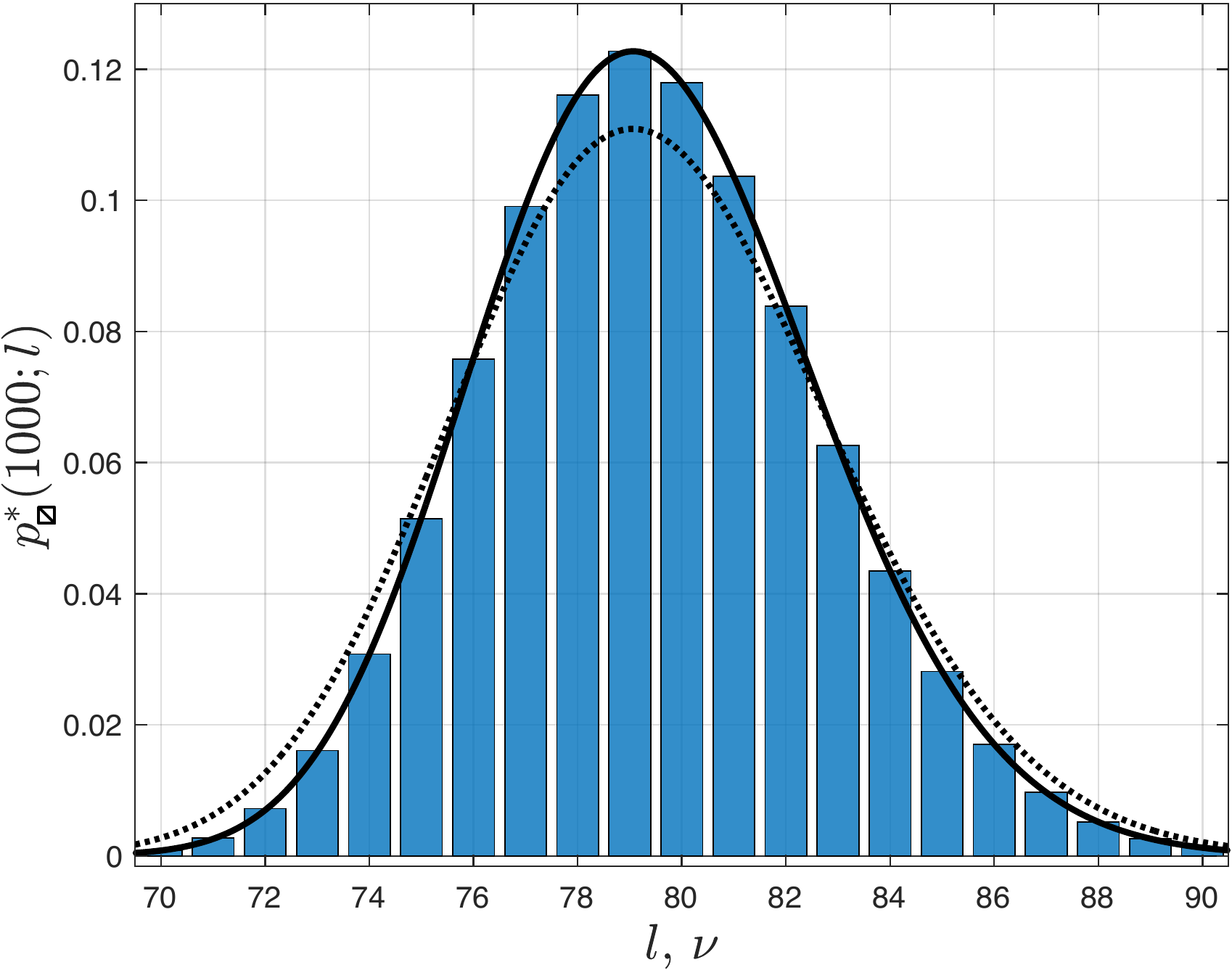}\hfil\,
\includegraphics[width=0.325\textwidth]{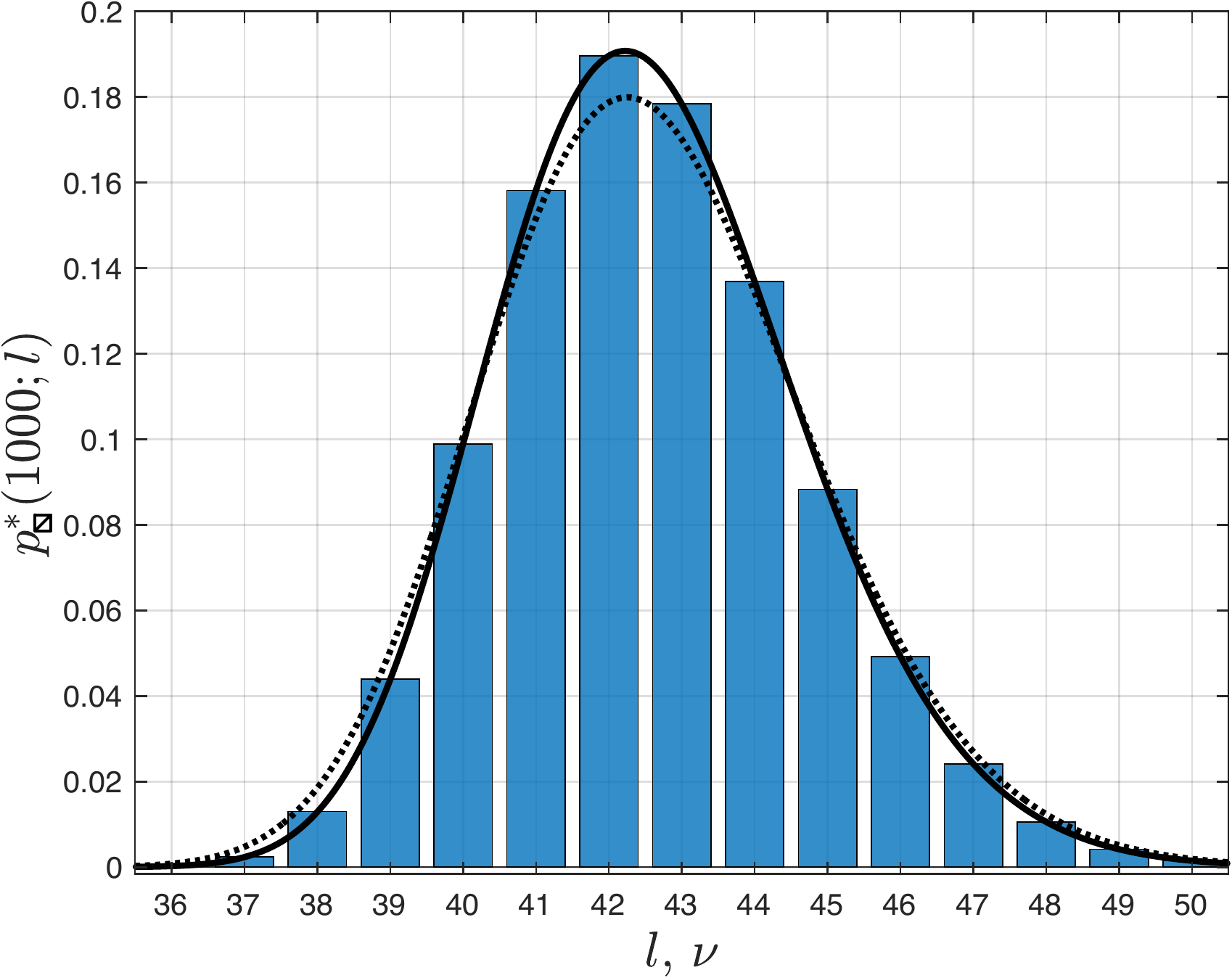}\hfil\,
\includegraphics[width=0.325\textwidth]{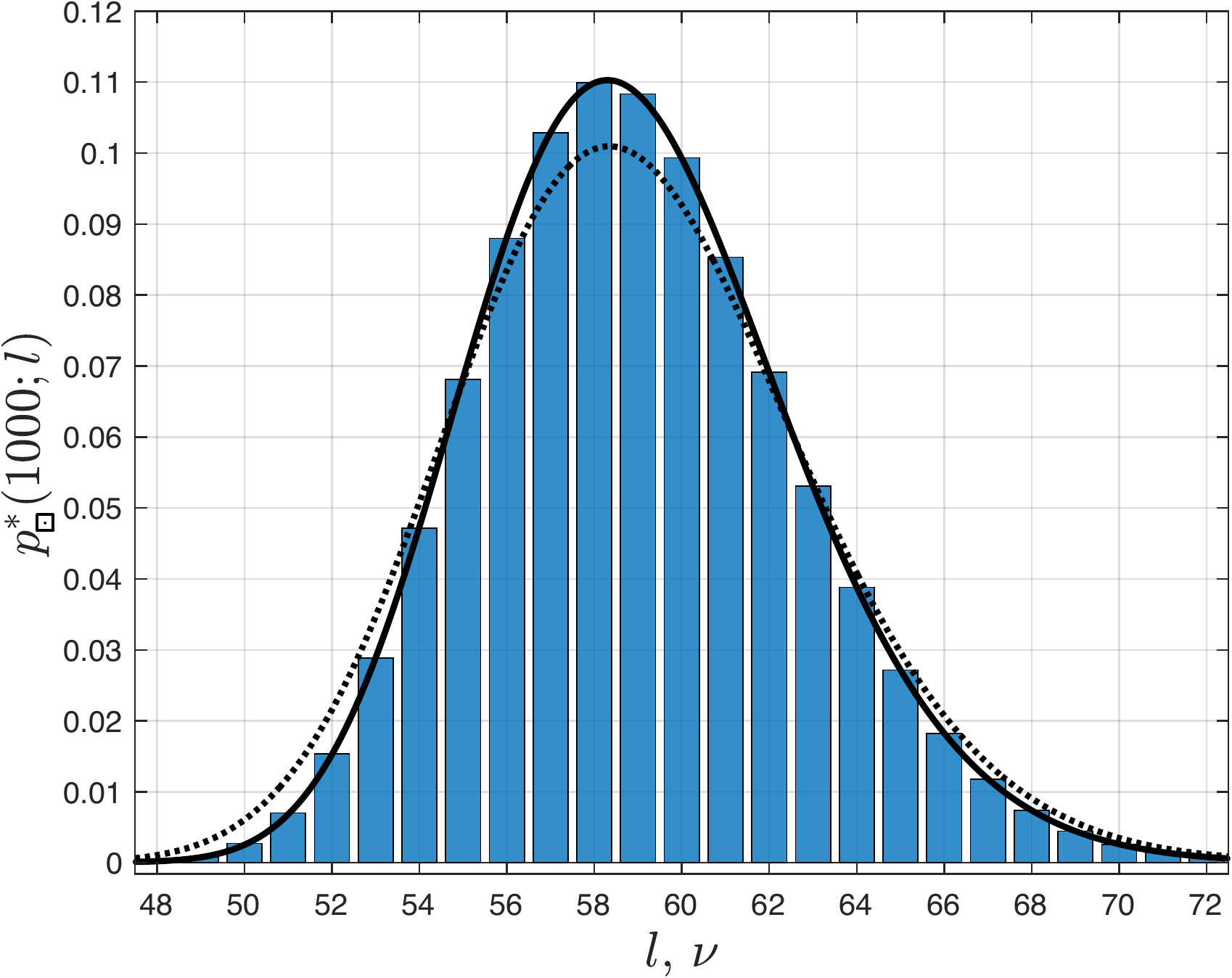}
\caption{{\footnotesize The exact discrete length probabilities for $n=1000$ (blue bars centered at the integers $l$) vs. their asymptotic expansions \eqref{eq:PDF_expan} with $m=0$ (the Baik–Rains limit laws; dotted lines) and with $m$ chosen such that the error improves by a factor of $O(n^{-2/3})$ (solid line). The expansions are displayed as functions of a continuous variable $\nu$, evaluating the right-hand-side of~\eqref{eq:PDF_expan} with $\nu$ replacing the integer $l$. Left panel: $p^*_\boxslash(n;l)$ vs. the choice $m=2$ (solid line). Middle panel:  $p^*_\boxbslash(n;l)$ vs. the choice $m=2$ (solid line). Right panel: $p^*_\boxdot(n;l)$ vs. the choice $m=5$ (solid line). The exact values are from the tables compiled in Sect.~\ref{sect:exact}. Note that a graphically accurate continuous approximation of the discrete distribution must intersect the bars right in the middle of their top sides: this is the case for the choices of $m$ made for the solid lines.  In contrast, the uncorrected limit laws (dotted lines) are noticeable inaccurate.}}
\label{fig:PDF}
\end{figure}

\begin{corollary} For $\varoast \in \{\boxslash,\boxbslash,\boxdot\}$, subject to the tameness hypothesis and Conjecture~\ref{conj:p_dot_expan}, there hold the expansions
\begin{equation}\label{eq:PDF_expan}
h_\varoast^{-1} \cdot p^*_\varoast(n;l) = F_{\beta(\varoast)}'(t) + \sum_{j=1}^m F_{\varoast,j}^*(t) \cdot (\gamma n)^{-\gamma j/6} + O(n^{-\gamma(m+1)/6}) \bigg|_{t=t_{(l-\frac12)^\varoast_*}(\gamma n),\;\gamma=\gamma(\varoast)},
\end{equation}
which are uniformly valid when $n,l\to\infty$ while $t_0\leq t\leq t_1$, with $m$ being any fixed non-negative integer and
$t_0 < t_1$ any fixed ordered pair of reals. The $F_{\varoast,j}^*$ are certain smooth functions that have simple expressions in terms of the functions $F_{\beta}$, $F_{\beta,j}$ in \eqref{eq:PoissonVaroastExpan}. 
\end{corollary}

By applying the difference formula \eqref{eq:central_diff} to the differences \eqref{eq:difference}
after inserting the concrete expressions \eqref{eq:FbetaD}/\eqref{eq:Fdot} for $F_{\varoast,j}$, a routine calculation shows that the first few instances of the functions $F_{\varoast,j}^*$ are in the case $\varoast=\boxslash$ given by
\begin{subequations}\label{eq:Fstar_boxslash}
\begin{align}
F_{\boxslash,1}^*(t) &= -\frac{t}{30} F_4'(t) - \frac{t^2}{60}F_4''(t) - \frac{139}{120} F_4'''(t)\\
\intertext{and, subject to the linear form hypothesis, by}
F_{\boxslash,2}^*(t) &= \frac{2t^2}{525}F_4'(t) + \Big(-\frac{8711}{8400}+\frac{23t^3}{12600}\Big)F_4''(t) + \Big(-\frac{1763t}{8400}+\frac{t^4}{7200}\Big)F_4'''(t) \\*[1mm]
&\qquad\quad + \frac{139t^2}{7200}F_4^{(4)}(t) + \frac{6437}{9600}F_4^{(5)}(t), \notag \\*[1mm]
F_{\boxslash,3}^*(t) &=
-\Big(\frac{761}{5250} + \frac{41 t^3}{70875} \Big) F_4'(t)
-\Big(\frac{1573 t}{12000} + \frac{71 t^4 }{283500}\Big) F_4''(t)\\*[1mm]
&\qquad\quad+\Big(\frac{837 t^2}{56000}-\frac{13 t^5}{504000}\Big) F_4'''(t) 
+\Big(\frac{514831 }{336000}+\frac{613 t^3}{302400} -\frac{t^6}{1296000} \Big) F_4^{(4)}(t)\notag\\*[1mm]
&\qquad\quad+\Big(\frac{535313 t}{2016000} - \frac{139 t^4}{864000}\Big) F_4^{(5)}(t)
-\frac{6437 t^2}{576000} F_4^{(6)}(t)
-\frac{2085527}{8064000}F_4^{(7)}(t)\notag,
\end{align}
\end{subequations}
in the case $\varoast=\boxbslash$ given by 
\begin{subequations}\label{eq:Fstar_boxbslash}
\begin{align}
F_{\boxbslash,1}^*(t) &= -\frac{t}{30} F_1'(t) - \frac{t^2}{60}F_1''(t) - \frac{31}{30} F_1'''(t)\\
\intertext{and, subject to the linear form hypothesis, by}
F_{\boxbslash,2}^*(t) &= \frac{2t^2}{525}F_1'(t) + \Big(-\frac{551}{525}+\frac{23t^3}{12600}\Big)F_1''(t) + \Big(-\frac{467t}{2100}+\frac{t^4}{7200}\Big)F_1'''(t)\hspace*{0.75cm} \\*[1mm]
&\qquad\quad + \frac{31t^2}{1800}F_1^{(4)}(t) + \frac{317}{600}F_1^{(5)}(t), \notag\\*[1mm]
\intertext{}
F_{\boxbslash,3}^*(t) &=
-\Big(\frac{18}{125} + \frac{41 t^3}{70875} \Big) F_1'(t)
-\Big(\frac{671 t}{5250} + \frac{71 t^4 }{283500}\Big) F_1''(t)\\*[1mm]
&\qquad\quad+\Big(\frac{17 t^2}{1000}-\frac{13 t^5}{504000}\Big) F_1'''(t) 
+\Big(\frac{7109 }{5250}+\frac{181 t^3}{75600} -\frac{t^6}{1296000} \Big) F_1^{(4)}(t)\notag\\*[1mm]
&\qquad\quad+\Big(\frac{31313 t}{126000} - \frac{31 t^4}{216000}\Big) F_1^{(5)}(t)
-\frac{317 t^2}{36000} F_1^{(6)}(t)
-\frac{22403}{126000}F_1^{(7)}(t)\notag,
\end{align}
\end{subequations}
and finally, in the case $\varoast=\boxdot$, subject to Conjecture~\ref{conj:p_dot_expan}, given by $F_{\boxdot,1}^*(t)=0$,
\begin{subequations}\label{eq:Fstar_boxdot}
\begin{align}
F_{\boxdot,2}^*(t) &= -\frac{t}{30} F_1'(t) - \frac{t^2}{60}F_1''(t) - \frac{139}{120} F_1'''(t)\\*[1mm]
F_{\boxdot,3}^*(t) &= \frac{1}{6} F_1'(t) + \frac{t}{6} F_1''(t)  \\
\intertext{and}
F_{\boxdot,4}^*(t) &= \frac{2t^2}{525}F_1'(t) + \Big(-\frac{10811}{8400}+\frac{23t^3}{12600}\Big)F_1''(t) + \Big(-\frac{1763t}{8400}+\frac{t^4}{7200}\Big)F_1'''(t) \\*[1mm]
&\hspace*{-0.5cm} + \frac{139t^2}{7200}F_1^{(4)}(t) + \frac{6437}{9600}F_1^{(5)}(t), \notag\\*[1mm]
F_{\boxdot,5}^*(t) &= -\frac{t}{45}F_1'(t) - \frac{7t^2}{360}F_1''(t) - \Big(\frac{19}{240}+\frac{t^3}{360}\Big)F_1'''(t) - \frac{139t}{720}F_1^{(4)}(t), \\*[1mm]
F_{\boxdot,6}^*(t) &=
-\Big(\frac{1933}{15750} + \frac{41 t^3}{70875} \Big) F_1'(t)
-\Big(\frac{291 t}{4000} + \frac{71 t^4 }{283500}\Big) F_1''(t)\\*[1mm]
&\hspace*{-0.5cm}+\Big(\frac{16633 t^2}{504000}-\frac{13 t^5}{504000}\Big) F_1'''(t) 
+\Big(\frac{612131 }{336000}+\frac{613 t^3}{302400} -\frac{t^6}{1296000} \Big) F_1^{(4)}(t)\notag\\*[1mm]
&\hspace*{-0.5cm}+\Big(\frac{535313 t}{2016000} - \frac{139 t^4}{864000}\Big) F_1^{(5)}(t)
-\frac{6437 t^2}{576000} F_1^{(6)}(t)
-\frac{2085527}{8064000}F_1^{(7)}(t)\notag,\\*[1mm]
F_{\boxdot,7}^*(t) &=
\frac{t^2}{225} F_1'(t) +\Big(-\frac{331}{2016} + \frac{29 t^3 }{9450}\Big) F_1''(t)
+\Big(-\frac{15509 t}{50400}+\frac{31 t^4}{60480}\Big) F_1'''(t)\\*[1mm]
&\hspace*{-0.5cm} +\Big(-\frac{6287 t^2}{302400}+\frac{t^5}{43200} \Big) F_1^{(4)}(t)
+\Big(-\frac{47}{2304} + \frac{139 t^3}{43200}\Big) F_1^{(5)}(t)
+\frac{6437 t}{57600} F_1^{(6)}(t)\notag.
\end{align}
\end{subequations}

\smallskip

Fig.~\ref{fig:PDF} displays plots of the length probability densities for $n=1000$ together with their Baik--Rains limit laws (that is, the choice $m=0$ in \eqref{eq:PDF_expan}; note that because of $F_{\boxdot,1}^*=0$ the error term is the same in all three cases) and expansions for which $m$ has been chosen to improve the error by a factor of $O(n^{-2/3})$. 

\section{Expansions of Expected Value and Variance}\label{sect:mean}

\subsection{Expected value}
The expected values of the random variables $L_n^\varoast$, that is,
\[
\E(L^\varoast_n) = 
\begin{cases}
\displaystyle\sum_{l=1}^n l\cdot \prob(L_n^\varoast=l),  &\quad \varoast = \boxslash,\boxdot,\\*[5mm]
\displaystyle\sum_{l=1}^n 2l\cdot \prob(L_n^\boxbslash=2l),  &\quad \varoast = \boxbslash,
\end{cases}
\]
can be written briefly in the joint form 
\[
\E(L^\varoast_n) = (\gamma(\varoast)n)^{1/6}  h_\varoast \sum_{l=1}^n l \cdot p_\varoast^*(n;l).
\]
Following \cite[§7]{arxiv:2301.02022}, the latter expression can be recast, by shift and rescale, in the form
\begin{subequations}
\begin{equation}
\E(L^\varoast_n) = 2\sqrt{\gamma n} + \delta(\varoast) +   h_\varoast \sum_{l=1}^n t_{(l-\frac12)^\varoast_*}(\gamma n) \cdot (\gamma n)^{1/6} h_\varoast^{-1} p_\varoast^*(n;l) \bigg|_{\gamma=\gamma(\varoast)},
\end{equation}
where we use the abbreviations
\begin{equation}
\delta(\boxslash) := 3/2,\qquad \delta(\boxbslash):=0,\qquad \delta(\boxdot) := -1/2.
 \end{equation}
\end{subequations}
If we assume a sufficiently uniform decay of the tails, inserting the expansions \eqref{eq:PDF_expan} gives
\begin{equation}
\E(L^\varoast_n) = 2\sqrt{\gamma n} + \delta(\varoast)  + \sum_{j=0}^m \mu_{\varoast,j}^{(n)} \cdot (\gamma n)^{(1-\gamma j)/6} + O(n^{(1-\gamma(m+1))/6}) \bigg|_{\gamma=\gamma(\varoast)}
\end{equation}
with coefficients (still depending on $n$, though), writing $F_{\varoast,0}^* := F_{\beta(\varoast)}'$,
\[
\mu_{\varoast,j}^{(n)} := h_\varoast \sum_{l=1}^n t_{(l-\frac12)^\varoast_*}(\gamma n)\cdot F_{\varoast,j}^*\big(t_{(l-\frac12)^\varoast_*}(\gamma n)\big)  \bigg|_{\gamma=\gamma(\varoast)}.
\]
Arguing as in \cite[§4.3]{arxiv.2206.09411}, \cite[§7]{arxiv:2301.02022}, if we assume (a) that  the decay $F_{\varoast,j}^*(t) \to 0$ (and likewise of all their derivatives) is exponentially fast as $t\to\pm\infty$ and (b) that  the
$F_{\varoast,j}^*$ can be extended analytically to a strip containing the real axis, we obtain
\begin{equation}\label{eq:trap}
\mu_{\varoast,j}^{(n)} \doteq h_\varoast \sum_{l=-\infty}^\infty t_{(l-\frac12)^\varoast_*}(\gamma n) \cdot F_{\varoast,j}^*\big(t_{(l-\frac12)^\varoast_*}(\gamma n)\big)  \bigg|_{\gamma=\gamma(\varoast)} \doteq \int_{-\infty}^\infty t\, F_{\varoast,j}^*(t)\,dt =:\mu_{\varoast,j},
\end{equation}
where “$\doteq$” denotes equality up to terms that are exponentially small for large $n$. Here, in the first step
the series was obtained by adding, under assumption (a), the exponentially small tails, and in the second step  we have identified the series as the trapezoidal rule with step-size $h =h_\varoast$---a quadrature rule known to converge, under assumption~(b), exponentially fast to the integral (see, e.g., \cite[Eq.~(3.4.14)]{MR760629}). In summary we are led to the following conjecture.

\begin{conjecture}\label{conj:E_expan} For $\varoast\in\{\boxslash,\boxbslash,\boxdot\}$ and $m$ any fixed non-negative integer, as $n\to\infty$,
\begin{equation}\label{eq:E_expan}
\E(L^\varoast_n) = 2\sqrt{\gamma n} + \delta(\varoast)  + \sum_{j=0}^m \mu_{\varoast,j} \cdot (\gamma n)^{(1-\gamma j)/6} + O(n^{(1-\gamma(m+1))/6}) \bigg|_{\gamma=\gamma(\varoast)},
\end{equation}
where the constants $\mu_{\varoast,j}$ are given by
\[
\mu_{\varoast,0}  = \int_{-\infty}^\infty t F_{\beta(\varoast)}'(t)\,dt, \qquad \mu_{\varoast,j}  = \int_{-\infty}^\infty t F_{\varoast,j}^*(t)\,dt \quad (j=1,2,\ldots).
\] 
\end{conjecture} 

\begin{table}[tbp]
\caption{Highly accurate values of the moments $M_{\beta,1},\ldots,M_{\beta,5}$ of the Tracy–Widom distributions $F_\beta$, $\beta=1,4$, computed as in \cite[Table~3]{arxiv.2206.09411} (cf. also \cite[Table~9]{MR2895091}).}
\label{tab:moments}
{\footnotesize
\begin{tabular}{rrr}
\hline\noalign{\smallskip}
$j$ & $M_{1,j}$\hspace*{1.75cm} & $M_{4,j}$ \hspace*{1.5cm}\\
\noalign{\smallskip}\hline\noalign{\smallskip}
$1$ & $ -1.20653\,35745\,82024\,81442\cdots$& $  -3.26242\,79028\,55175\,75465\cdots$ \\
$2$ & $  3.06350\,43011\,75039\,55546\cdots$& $  11.67888\,32628\,73371\,70764\cdots$ \\
$3$ & $ -6.97763\,61359\,60327\,30644\cdots$& $ -44.68327\,25223\,28257\,83341\cdots$ \\
$4$ & $ 21.45673\,87602\,71069\,02360\cdots$& $ 180.40053\,05488\,70404\,65211\cdots$ \\
$5$ & $-61.49120\,46024\,95471\,52526\cdots$& $-762.06682\,37306\,65236\,06580\cdots$ \\
\noalign{\smallskip}\hline
\end{tabular}}
\end{table}

The functional form of $F_{\varoast,j}^*$ displayed in (\ref{eq:Fstar_boxslash}–\ref{eq:Fstar_boxdot}), namely being a linear combination of higher order derivatives of $F_{\beta(\varoast)}$ with polynomial coefficients, allows us to express the coefficients $\mu_{\varoast,j}$ in terms of the moments (for highly accurate numerical values see Table~\ref{tab:moments})
\[
M_{\beta,j}:= \int_{-\infty}^\infty t^j F_\beta'(t)\,dt
\]
of the Tracy--Widom distributions $F_\beta$ ($\beta=1,4$). In fact, repeated integration by parts yields the simplifying rule (where $k\geq 1$)
\[
\int_{-\infty}^\infty t^j\, F_\beta^{(k)}(t) \,dt = 
\begin{cases}
\displaystyle\frac{(-1)^{k-1} j!}{(j-k+1)!} M_{\beta,j-k+1} &\quad k\leq j+1,\\*[4mm]
0 &\quad\text{otherwise}.
\end{cases}
\]
Applying this rule yields the first few instances in the cases $\varoast=\boxslash,\boxbslash$, writing $\beta=\beta(\varoast)$, as
\begin{subequations}\label{eq:mu}
\begin{equation}\label{eq:mu_fixedpointfree}
\begin{gathered}
\mu_{\varoast,0} = M_{\beta,1},\quad \mu_{\varoast,1} = \frac{M_{\beta,2}}{60}, \quad \mu_{\varoast,2} = \frac{351}{700} - \frac{M_{\beta,3}}{1400},\\*[1mm]
\mu_{\varoast,3} = \frac{8753M_{\beta,1}}{63000} +\frac{281M_{\beta,4}}{4536000} ,
\end{gathered}
\end{equation}
and in the case $\varoast=\boxdot$ as
\begin{equation}
\begin{gathered}
\mu_{\boxdot,0} = M_{1,1},\quad \mu_{\boxdot,1} = 0, \quad \mu_{\boxdot,2} = \frac{M_{1,2}}{60}, \quad \mu_{\boxdot,3} =  - \frac{M_{1,1}}{6},\quad \mu_{\boxdot,4} = \frac{263}{350} - \frac{M_{1,3}}{1400}, \\*[1mm]
\mu_{\boxdot,5} = \frac{M_{1,2}}{360}, \quad \mu_{\boxdot,6} = \frac{2407M_{1,1}}{15750} +\frac{281 M_{1,4}}{4536000},\quad \mu_{\boxdot,7} = -\frac{349}{1400}-\frac{M_{1,3}}{2800}.
\end{gathered}
\end{equation}
\end{subequations}
Highly accurate numerical values of these coefficients are listed in Table~\ref{tab:mu}.

\begin{table}[tbp]
\caption{Highly accurate values of $\mu_{\varoast,j}$, computed from \eqref{eq:mu} and Table~\ref{tab:moments}. (For the values of $\mu_{\boxslash,4},\ldots,\mu_{\boxslash,7}, \mu_{\boxbslash,4},\ldots,\mu_{\boxbslash,7}$  see the supplementary material mentioned in Fn.~\ref{fn:suppl}.)}
\label{tab:mu}
{\footnotesize
\begin{tabular}{rrrr}
\hline\noalign{\smallskip}
$j$ & $\mu_{\boxslash,j}$\hspace*{1.75cm} & $\mu_{\boxbslash,j}$ \hspace*{1.5cm}& $\mu_{\boxdot,j}$ \hspace*{1.5cm} \\
\noalign{\smallskip}\hline\noalign{\smallskip}
$0$ & $ -3.26242\,79028\,55175\,75465\cdots$& $-1.20653\,35745\,82024\,81442\cdots$ & $ -1.20653\,35745\,82024\,81442\cdots$\\
$1$ & $  0.19464\,80543\,81222\,86179\cdots$& $ 0.05105\,84050\,19583\,99259\cdots$ & $  0.00000\,00000\,00000\,00000\cdots$\\
$2$ & $  0.53334\,51946\,58805\,89845\cdots$& $ 0.50641\,25972\,39971\,66236\cdots$ & $  0.05105\,84050\,19583\,99259\cdots$\\
$3$ & $ -0.44209\,47341\,58188\,90204\cdots$& $-0.16630\,23411\,92052\,28837\cdots$ & $  0.20108\,89290\,97004\,13573\cdots$\\
$4$ & $ -0.03369\,11358\,88346\,37655\cdots$& $-0.00981\,19952\,49787\,79659\cdots$ & $  0.75641\,25972\,39971\,66236\cdots$\\
$5$ & $  0.03852\,72197\,81191\,57824\cdots$& $ 0.05746\,29940\,91860\,81619\cdots$ & $  0.00850\,97341\,69930\,66543\cdots$\\
$6$ & $  0.05153\,76359\,56129\,10159\cdots$& $ 0.02332\,12891\,18646\,87049\cdots$ & $ -0.18305\,97519\,50135\,96634\cdots$\\
$7$ & $  0.00772\,23611\,77637\,88596\cdots$& $ 0.00472\,40800\,60295\,18940\cdots$ & $ -0.24679\,37013\,80014\,16881\cdots$\\
\noalign{\smallskip}\hline
\end{tabular}}
\end{table}

As a sanity check we have fitted (computed in extended precision) model expansions to exact data sets obtained from the tables up to $n=1000$ as compiled in Sect.~\ref{sect:exact}. Even though the accuracies of these fits vary strongly (significantly larger tables would be required in the cases $\varoast=\boxbslash,\boxdot$ to match the accuracy obtained for $\varoast=\boxslash$), all the digits that have been deemed correct (by comparing the fits for two different data sets) agree, up to one unit in the last place, with the numbers shown in Table~\ref{tab:mu}. The specific models and fits are:\footnote{The upper bounds of the index $j$ in the models have been chosen as to maximize the number of matching digits for the two different data sets that are used.}

\begin{itemize}
\item $\displaystyle\E(L_n^\boxslash) \approx 2\sqrt{2n} +\frac{3}{2} + \sum_{j=0}^{13} c_{\boxslash,j}\cdot (2n)^{(1-2j)/6}$
with data for $700 (800) \leq n \leq 1000$, 
\begin{gather*}
c_{\boxslash,0} \approx -3.26242\,79028\,55175,\quad c_{\boxslash,1} \approx 0.19464\,80543\,81, \\*[1mm]
c_{\boxslash,2} \approx 0.53334\,51946, \quad c_{\boxslash,3} \approx -0.44209\,4734, \quad c_{\boxslash,4} \approx -0.03369\,11,\\*[1mm]
 c_{\boxslash,5} \approx 0.03852\,7, \quad c_{\boxslash,6} \approx 0.05153, \quad c_{\boxslash,7}\approx  0.0077.
\end{gather*}
\item $\displaystyle\E(L_n^\boxbslash) \approx 2\sqrt{2n} + \sum_{j=0}^{5} c_{\boxbslash,j} \cdot(2n)^{(1-2j)/6}$
with data starting for $600(700)\leq n\leq 1000$, 
\[
c_{\boxbslash,0} \approx -1.20653\,35,\quad c_{\boxbslash,1} \approx 0.05105\,8, \quad c_{\boxbslash,2} \approx 0.506, \quad c_{\boxbslash,3} \approx -0.166 .
\]
\item $\displaystyle\E(L_n^\boxdot) \approx 2\sqrt{n} -\frac{1}{2} + c_{\boxdot,0} \cdot n^{1/6} + \sum_{j=1}^{8} c_{\boxdot,j+1} \cdot n^{-j/6}$,
with data for $700(800)\leq n\leq 1000$, 
\[
c_{\boxdot,0} \approx -1.20653\,2,\quad c_{\boxdot,2} \approx 0.050, \quad
c_{\boxdot,3} \approx 0.20, \quad c_{\boxdot,4} \approx 0.7 .
\]
\end{itemize}


\subsection{Variance}
The variances of the random variables $L_n^\varoast$, 
\[
\Var(L^\varoast_n) = 
\begin{cases}
\displaystyle\sum_{l=1}^n l^2\cdot \prob(L_n^\varoast=l) - \E(L_n^\varoast)^2,  &\quad \varoast = \boxslash,\boxdot,\\*[5mm]
\displaystyle\sum_{l=1}^n (2l)^2\cdot \prob(L_n^\boxbslash=2l)- \E(L_n^\boxbslash)^2,,  &\quad \varoast = \boxbslash,
\end{cases}
\]
can be recast, by a shift and rescale, in the form
\[
\Var(L^\varoast_n) =  h_\varoast \sum_{l=1}^n \left(t_{(l-\frac12)^\varoast_*}(\gamma n)\right)^2 \cdot (\gamma n)^{1/3} h_\varoast^{-1} p_\varoast^*(n;l) - \Big(\E(L_n^\varoast)-2\sqrt{\gamma n} -\delta(\varoast)\Big)^2\bigg|_{\gamma=\gamma(\varoast)}.
\]
By inserting the expansions \eqref{eq:PDF_expan} and \eqref{eq:E_expan} and arguing as for Conjecture \ref{conj:E_expan} we get:
\begin{conjecture} For $\varoast\in\{\boxslash,\boxbslash,\boxdot\}$ and $m$ any fixed non-negative integer, as $n\to\infty$,
\begin{equation}
\Var(L^\varoast_n) = \sum_{j=0}^m \nu_{\varoast,j} \cdot (\gamma n)^{(2-\gamma j)/6} + O(n^{(2-\gamma(m+1))/6}) \bigg|_{\gamma=\gamma(\varoast)},
\end{equation}
where the $\nu_{\varoast,j}$ can be expressed in terms of the $\mu_{\varoast,k}$ and the second moments of $F'_{\beta(\varoast)}$, $F_{\varoast,j}^*$.
\end{conjecture}

Using the formulae for $\mu_{\varoast,j}$ in \eqref{eq:mu} we obtain in the cases $\varoast=\boxslash,\boxbslash$, writing $\beta=\beta(\varoast)$,
\begin{subequations}\label{eq:nu}
\begin{equation}\label{eq:nu_fixedpointfree}
\begin{gathered}
\nu_{\varoast,0} = -M_{\beta,1}^2 + M_{\beta,2},\quad \nu_{\varoast,1} =  -c_\varoast - \frac{M_{\beta,1} M_{\beta,2} - M_{\beta,3}}{30} \quad \Big(c_\boxslash=\frac{139}{60}, c_\boxbslash=\frac{31}{15}\Big),\\*[1mm]
\nu_{\varoast,2} = -\frac{114M_{\beta,1}}{175}  + \frac{M_{\beta,1} M_{\beta,3}}{700}  - \frac{M_{\beta,2}^2 }{3600} - \frac{29M_{\beta,4}}{25200},\\*[1mm]
\nu_{\varoast,3} =-\frac{8753M_{\beta,1}^2}{31500} -\frac{281M_{\beta,1}M_{\beta,4}}{2268000}
+\frac{7289M_{\beta,2}}{31500} +\frac{M_{\beta,2}M_{\beta,3} }{42000}+ \frac{227M_{\beta,5}}{2268000},
\end{gathered}
\end{equation}
and in the case $\varoast=\boxdot$
\begin{equation}
\begin{gathered}
\nu_{\boxdot,0} = -M_{1,1}^2 + M_{1,2},\quad \nu_{\boxdot,1} = 0, \quad \nu_{\boxdot,2} =  -\frac{139}{60} - \frac{M_{1,1} M_{1,2} - M_{1,3}}{30},\\*[1mm]
 \nu_{\boxdot,3} =  -\frac{\nu_{\boxdot,0}}{3}, \quad
\nu_{\boxdot,4} = -\frac{114M_{1,1}}{175}  + \frac{M_{1,1} M_{1,3}}{700}  - \frac{M_{1,2}^2 }{3600} - \frac{29M_{1,4}}{25200},\quad \nu_{\boxdot,5}=1,\\*[3mm]
\nu_{\boxdot,6} = -\frac{1167M_{1,1}^2}{3500} -\frac{281M_{1,1}M_{1,4}}{2268000}
+\frac{3013M_{1,2}}{10500} +\frac{M_{1,2}M_{1,3} }{42000}+ \frac{227M_{1,5}}{2268000},\\*[1mm]
\nu_{\boxdot,7} = \frac{61M_{1,1}}{525} + \frac{M_{1,1}M_{1,3}}{2100} - \frac{M_{1,2}^2}{10800} - \frac{29M_{1,4}}{75600}.
\end{gathered}
\end{equation}
\end{subequations}
Highly accurate numerical values of these coefficients are listed in Table~\ref{tab:nu}.

\begin{table}[tbp]
\caption{Highly accurate values of $\nu_{\varoast,j}$, computed from \eqref{eq:nu} and Table~\ref{tab:moments}. (For the values of $\nu_{\boxslash,4},\ldots,\nu_{\boxslash,7}, \nu_{\boxbslash,4},\ldots,\nu_{\boxbslash,7}$  see the supplementary material mentioned in Fn.~\ref{fn:suppl}.)}
\label{tab:nu}
{\footnotesize
\begin{tabular}{rrrr}
\hline\noalign{\smallskip}
$j$ & $\nu_{\boxslash,j}$\hspace*{1.75cm} & $\nu_{\boxbslash,j}$ \hspace*{1.5cm}& $\nu_{\boxdot,j}$ \hspace*{1.5cm} \\
\noalign{\smallskip}\hline\noalign{\smallskip}
$0$ & $  1.03544\,74415\,45351\,61669\cdots$& $ 1.60778\,10345\,81361\,12010\cdots$ & $  1.60778\,10345\,81361\,12010\cdots$\\
$1$ & $ -2.53605\,85963\,78062\,38245\cdots$& $-2.17604\,71780\,23873\,29698\cdots$ & $  0.00000\,00000\,00000\,00000\cdots$\\
$2$ & $  2.08799\,84622\,98281\,00943\cdots$& $ 0.77069\,79923\,71449\,91015\cdots$ & $ -2.42604\,71780\,23873\,29698\cdots$\\
$3$ & $ -0.27084\,67828\,98392\,09134\cdots$& $ 0.30092\,29577\,39733\,86828\cdots$ & $ -0.53592\,70115\,27120\,37336\cdots$\\
$4$ & $ -0.52590\,53029\,58638\,02379\cdots$& $-0.49703\,33761\,70673\,86818\cdots$ & $  0.77069\,79923\,71449\,91015\cdots$\\
$5$ & $ -0.03597\,95025\,33132\,40447\cdots$& $-0.00633\,41102\,95048\,26528\cdots$ & $  1.00000\,00000\,00000\,00000\cdots$\\
$6$ & $  0.21324\,06135\,58107\,00045\cdots$& $ 0.02666\,33723\,56650\,13943\cdots$ & $  0.39024\,41263\,27587\,26384\cdots$\\
$7$ & $  0.27770\,22051\,50779\,68026\cdots$& $ 0.24693\,23583\,17086\,74356\cdots$ & $ -0.14527\,85274\,03524\,96809\cdots$\\
\noalign{\smallskip}\hline
\end{tabular}}
\end{table}

A similar sanity check as done for the expected value can be run for the variance: fitting (computed in extended precision) model expansions to exact data sets obtained from the tables up to $n=1000$ as compiled in Sect.~\ref{sect:exact}. Once more, even though the accuracies of these fits vary strongly, all the digits that have been deemed correct (by comparing the fits for two different data sets) agree, up to one unit in the last place, with the numbers shown in Table~\ref{tab:nu}. The specific models and fits are:\footnote{The upper bounds of the index $j$ in the models have been chosen as to maximize the number of matching digits for the two different data sets that are used.}
\begin{itemize}
\item $\displaystyle\Var(L_n^\boxslash) \approx (2n)^{1/3} \sum_{j=0}^{12} d_{\boxslash,j}\cdot (2n)^{-j/3}$
with data for $700(800)\leq n \leq 1000$,
\begin{gather*}
d_{\boxslash,0} \approx 1.03544\,74415\,4535,\quad d_{\boxslash,1} \approx -2.53605\,85963\,7, \\*[1mm]
d_{\boxslash,2} \approx 2.08799\,84623, \quad d_{\boxslash,3} \approx -0.27084\,678, \quad
d_{\boxslash,4} \approx -0.52590\,52,\\*[1mm]
d_{\boxslash,5} \approx -0.03598, \quad d_{\boxslash,6} \approx 0.2132, \quad d_{\boxslash,7} \approx 0.277.
\end{gather*}
\item $\displaystyle\Var(L_n^\boxbslash) \approx (2n)^{1/3} \sum_{j=0}^{4} d_{\boxbslash,j}\cdot (2n)^{-j/3}$
with data for $600(700)\leq n \leq 1000$, 
\begin{gather*}
d_{\boxbslash,0} \approx 1.60778\,1,\; d_{\boxbslash,1} \approx -2.17605, \;
d_{\boxbslash,2} \approx 0.770, \; d_{\boxbslash,3} \approx 0.299, \; d_{\boxbslash,4} \approx -0.49.
\end{gather*}
\item $\displaystyle\Var(L_n^\boxdot) \approx  d_{\boxdot,0} \cdot n^{1/3} + \sum_{j=0}^{8} d_{\boxdot,j+2} \cdot n^{-j/6}$
with data for $n=700(800)\leq n\leq 1000$,
\[
d_{\boxdot,0} \approx 1.60777\,9,\quad d_{\boxdot,2} \approx -2.425, \quad
d_{\boxdot,3} \approx -0.54, \quad d_{\boxdot,4} \approx 0.8 .
\]
\end{itemize}

\subsection*{Acknowledgements} We thank Peter Forrester for his encouragement and for suggesting the use of $\tau$-function representations to prove Thm.~\ref{thm:E_group}.

\appendix

\section{Expansions of Operator Determinants}

We briefly explain the modifications of \cite[§2]{arxiv:2301.02022} that are required for the purposes of this paper. In Sect.~\ref{sect:hard-to-soft} we consider, with $m$ being some non-negative integer and $t_0$ some real number, kernel expansions of the form
\[
K_{(h)}(x,y) = \sum_{j=0}^m h^j K_j(x,y)  + h^{m+1} R_{m+1,h}(x,y), \quad R_{m+1,h}(x,y) = O(e^{-(x+y)/2}),
\]
which are
\begin{itemize}\itemsep=3pt
\item uniformly valid for $t_0 \leq x,y < c h^{-1}$ as $h\to 0^+$, where $c>0$ is some constant;
\item repeatedly differentiable w.r.t. $x$, $y$ as uniform expansions under the same conditions.
\end{itemize}

\smallskip\noindent
Here, $K_{(h)}$ is a family of smooth kernels and the $K_j$ are of the functional form \eqref{eq:KjForm}, that is,
\begin{equation}\label{eq:KjFormGen}
p(x,y) \Ai\Big(\frac{x+y}{2}\Big)+q(x,y) \Ai'\Big(\frac{x+y}{2}\Big)
\end{equation}
with certain symmetric polynomials $p$ and $q$.

For a given continuous kernel $K(x,y)$ we denote the induced integral operator on $L^2(t, ch^{-1})$ by $\bar{\mathbf K}$
and the one on $L^2(t,\infty)$, if defined, by ${\mathbf K}$ (suppressing the dependence on $t$ in both cases). The space of trace class operators acting on $L^2(t,s)$ is written as ${\mathcal J}^1(t,s)$. As noted in \cite[Eq.~(2.3)]{arxiv:2301.02022} we have
\begin{equation}\label{eq:traceclassinclusion}
\|\bar {\mathbf K}\|_{{\mathcal J}^1(t,ch^{-1})} \leq \|{\mathbf K}\|_{{\mathcal J}^1(t,\infty)}.
\end{equation}
Using the factoring into Hilbert-Schmidt operators by differentiation as explained in \cite[§2.1]{arxiv:2301.02022}, we get the trace class bounds of the expansion kernels
\[
\|{\mathbf K_j}\|_{{\mathcal J}^1(t,\infty)} = O(e^{-t})
\]
and of the remainder 
\[
h^{m+1}\|\bar{\mathbf R}_{m+1,h}\|_{{\mathcal J}^1(t,ch^{-1})} = h^{m+1} O(e^{-t}) + e ^{-c h^{-1}/2} O(e^{-t/2}).
\]
If the symmetric operator ${\mathbf K_0}$ has a spectral radius $\rho({\mathbf K_0})$ that stays below $1$, uniformly when $t\geq t_0$, so that $\|{\mathbf K_0}\| = \rho({\mathbf K_0}) \leq c(t_0) < 1$,
we get by functional calculus the uniform operator norm bound
\[
\|({\mathbf I - \mathbf K_0})^{-1}\| = \frac{1}{1-\rho({\mathbf K_0})} \leq \frac{1}{1-c(t_0)}.
\]
In particular, this is the case for the specific choice $K_0(x,y) = z V_{\Ai}(x,y)$, with some fixed parameter $-1 \leq z \leq 1$, since numerical evidence shows that the minimal eigenvalue $\lambda_{\min}(\mathbf V_{\Ai})$ interpolates strictly monotonically between the limit cases
\[
\lim_{t\to -\infty} \lambda_{\min}(\mathbf V_{\Ai}) = -1,\quad \lim_{t\to \infty} \lambda_{\min}(\mathbf V_{\Ai}) = 0,
\]
and the maximal eigenvalue $\lambda_{\max}(\mathbf V_{\Ai})$ strictly monotonically between the limit cases
\[
\lim_{t\to -\infty} \lambda_{\max}(\mathbf V_{\Ai}) = 1,\quad \lim_{t\to \infty} \lambda_{\max}(\mathbf V_{\Ai}) = 0.
\]
Therefore, with exactly the same proof as for \cite[Thm.~2.1]{arxiv:2301.02022}, we obtain the following theorem.

\begin{theorem}\label{thm:detexpan} Let $K_{(h)}(x,y)$ be a continuous kernel, $K_0(x,y)= z V_{\Ai}(x,y)$ with $-1 \leq z \leq 1$, and $K_j$ ($j=1,2,\ldots$) of the functional form \eqref{eq:KjFormGen}. If, for some fixed non-negative integer $m$ and some real number $t_0$, there is a kernel expansions of the form 
\[
K_{(h)}(x,y) = K_0(x,y) + \sum_{j=1}^m h^j K_j(x,y) + h^{m+1} \cdot O\big(e^{-(x+y)/2}\big),
\]
which, for some constant $c>0$,  holds uniformly in $t_0\leq x,y < ch^{-1}$ as $h\to 0^+$and which can be repeatedly differentiated w.r.t. $x$ and $y$ as uniform expansions, then the Fredholm determinant of $K_{(h)}$ on $(t,ch^{-1})$ satisfies
\begin{equation}\label{eq:detexpan}
\det(I-K_{(h)})|_{L^2(t,ch^{-1})} = F(t)\cdot \sum_{j=0}^m d_j(t) h^j + h^{m+1} O(e^{-t}) +  e ^{-c h^{-1}/2} O(e^{-t/2}),
\end{equation}
uniformly for $t_0\leq t < ch^{-1}$ as $h\to 0^+$. Here $F(t) = \det(I-K_0)|_{L^2(t,\infty)}$
and the $d_j(t)$ are smooth functions depending on the kernels $K_0,\ldots,K_j$ which satisfy the right tail bounds $F(t)d_j(t) = O(e^{-t})$. 
If we write briefly 
\[
{\mathbf E}_j = ({\mathbf I}-{\mathbf K_0})^{-1}{\mathbf K_j}
\]
then the first cases of the expansion terms are explicitly given as $d_0(t)=1$ and
\begin{align*}
d_1(t) &= - \tr {\mathbf E}_1, \\*[0.5mm]
d_2(t) &= \frac{1}{2}(\tr{\mathbf E_1})^2 - \frac{1}{2}\tr{\mathbf E}_1^2 - \tr{\mathbf E_2},\\*[0.5mm]
d_3(t) &= -\frac{1}{6}(\tr{\mathbf E_1})^3 +\frac{1}{2}\tr {\mathbf E}_1\tr {\mathbf E}_1^2 -\frac{1}{2}\tr({\mathbf E}_1{\mathbf E}_2 + {\mathbf E}_2 {\mathbf E}_1) - \frac{1}{3}\tr {\mathbf E}_1^3 + \tr{\mathbf E}_1 \tr{\mathbf E}_2 - \tr{\mathbf E}_3,
\end{align*}
where the resolvents and traces are taken over $L^2(t,\infty)$. The determinantal expansion \eqref{eq:detexpan} can repeatedly be differentiated w.r.t. $t$, preserving uniformity.
\end{theorem}

\section{Yet another criterion for $H$-admissibility}\label{app:hayman}

The $H$-admissibility\footnote{For an exposition of Hayman's memoir \cite{Hayman56} on $H$-admissibility and Stirling-type formulae, and further pointers to the literature, see our previous work \cite[Appendix~A]{arxiv:2301.02022}.} of the generating function $f_l^\boxdot(z)$, as stated in Cor.~\ref{cor:f_boxdot_Hadmiss}, is based on a new
elegant criterion (Thm.~\ref{thm:criterion} below). Because we use the definition of $H$-admissibility to prove this result, we will review that definition and some fundamental results for convenience.
\begin{definition}[Hayman \protect{\cite[p.~68]{Hayman56}}]\label{def:hayman}
An entire function $f(z)$ is said to be {\em $H$-admissible} if the following four conditions are satisfied:
\begin{itemize}
\item[--] [{\em positivity}\/] for sufficiently large $r>0$, there holds $f(r)>0$; inducing there the real functions (which we call the auxiliary functions associated with $f$)
\[
a(r) = r \frac{f'(r)}{f(r)},\qquad b(r) = r a'(r);
\]
by Hadamard's convexity theorem $a(r)$ is monotonely increasing and $b(r)$ is positive.\\*[-3mm]
\item[--]  [{\em capture}\/] $b(r) \to \infty$ as $r\to\infty$;\\*[-3mm]
\item[--]  [{\em locality}\/] for some function $0<\delta(r)<\pi$ there holds\footnote{As is customary in asymyptotic analysis in the complex plane, we understand such asymptotics to hold {\em uniformly} in the stated angular segments for all $r\geq r_0$ with some sufficiently large $r_0>0$. }
\[
f(r e^{i\theta}) = f(r) e^{i\theta a(r) - \theta^2 b(r)/2}\, (1+ o(1)) \qquad (r\to\infty,\; |\theta|\leq \delta(r));
\]
\item[--]  [{\em decay}\/] for the angles in the complement there holds
\[
f(r e^{i\theta}) = \frac{o(f(r))}{\sqrt{b(r)}}\qquad (r\to\infty,\; \delta(r) \leq |\theta| \leq \pi).
\]
\end{itemize}
\end{definition}

The fundamental properties of $H$-admissible functions are as follows.

\begin{theorem}[Hayman \protect{\cite[Thm.~I/V, Cor.~I/II, Eq.~(4.3)]{Hayman56}}]\label{thm:hayman} Let $f$ 
be an entire $H$-admissible function with Maclaurin series
\[
f(z) = \sum_{n=0}^\infty a_nz^n\qquad (z\in \C).
\]
Then:
\begin{itemize}
\item[I.] {\em [maximum modulus]} For sufficiently large $r>0$, there holds
\begin{equation}\label{eq:maxmod}
|f(r e^{i\theta})| < f(r) \qquad (0 < |\theta| \leq \pi)
\end{equation}
and, as $r\to\infty$,
\begin{equation}\label{eq:bdelta}
a(r)\to\infty, \qquad b(r) = o(a(r))^2, \qquad b(r)\delta(r)^2 \to \infty.
\end{equation}
\item[II.]  {\em [normal approximation]} There holds, uniformly in $n \in \N_0$, that
\begin{equation}\label{eq:CLT}
\frac{a_n r^n}{f(r)} = \frac{1}{\sqrt{2\pi b(r)}}\left(\exp\left(-\frac{(n-a(r))^2}{2b(r)}\right)+ o(1)\right) \qquad (r\to\infty).
\end{equation}

\item[III.]  {\em [Stirling-type formula]} For $n$ sufficiently large, it follows from\/ {\rm I} that $a(r_n) = n$ has a unique solution $r_n$ such that $r_n\to\infty$ as $n\to\infty$ and therefore, by the normal approximation \eqref{eq:CLT}, there holds
\begin{equation}\label{eq:stirling}
a_n = \frac{f(r_n)}{r_n^n \sqrt{2\pi b(r_n)}} (1+ o(1)) \qquad (n\to\infty).
\end{equation}
\end{itemize}
\end{theorem}

For the probabilistic content of the normal approximation \eqref{eq:CLT} see, e.g., \cite{MR2095975} and \cite[Remark~2.1]{arxiv.2206.09411}. In practice, the normal approximation simplifies establishing asymptotic formulae for $a_n$, as compared to using the Stirling-type formula \eqref{eq:stirling}, since it allows us to solve the equation $a(r_n)=n$ just approximately, see, e.g., Appendix~\ref{app:inv}. The definition of $H$-admissible functions is often difficult to be verified and therefore rarely directly used. Instead, one uses general criteria and closure properties that guarantee $H$-admissibility, such as:

 
\begin{theorem}[Hayman \protect{\cite[Thm.~VI--X]{Hayman56}}]\label{thm:haymancriteria}
Let $f(z)$ and $g(z)$ be entire $H$-admissible functions and let $p(z)$ denote a polynomial with real coefficients. Then there hold the following closure properties:
\begin{itemize}
\item[I.] $f(z)g(z)$, $e^{f(z)}$ and $f(z) + p(z)$ are $H$-admissible.\\*[-3mm]
\item[II.] If the leading coefficient of $p$ is positive, $f(z)p(z)$ and $p(f(z))$ are $H$-admissible.\\*[-3mm]
\item[III.] If the Maclaurin coefficients of $e^{p(z)}$ are eventually positive, $e^{p(z)}$ is $H$-admissible.  
\end{itemize}
\end{theorem}

Hayman gives also a criterion \cite[Thm.~XI]{Hayman56} that is based on the concept of entire functions of genus zero, subject to conditions on the distribution of their zeros. As shown by the author in his study \cite{arxiv.2206.09411} of a Stirling-type formula related to the general permutation case, Hayman's result can conveniently repackaged in form of a singularity analysis at $z=\infty$, as provided by the following theorem. 

\begin{theorem}[Cf. the proof of \protect{\cite[Lemma~2.1]{arxiv.2206.09411}}]\label{thm:born} Let $f(z)$ be an entire function of exponential type with positive Maclaurin coefficients. If there are real constants $c, \tau, \nu$ with $c,\tau > 0$ and a positive integer $m \geq 2$ such that there holds, for the principal branch of the power function and for each $0 < \delta\leq \frac\pi2$,
the asymptotic expansion\begin{equation}\label{eq:mexpansion}
f(z^m) = cz^{\nu} e^{\tau z}(1 + O(z^{-1}))\qquad (z\to\infty,\; |\!\arg z|\leq \tfrac{\pi}{m}-\delta),
\end{equation}
then $f$ has genus zero, possesses at most finitely many zeros in the sector $|\arg z\,| \leq \pi - m \delta$ and is $H$-admissible. For $r\to\infty$, the associated auxiliary functions $a(r)$ and $b(r)$ satisfy
\begin{equation}\label{eq:auxexpansion}
a(r) = \frac{\tau}{m}r^{1/m} + \frac{\nu}{m} + O(r^{-1/m}), \qquad
b(r) = \frac{\tau}{m^2}r^{1/m}  + O(r^{-1/m}),
\end{equation}
and the solution $r_n$ of $a(r_n)=n$ satisfies
\begin{equation}\label{eq:rnexpansion}
r_n^{1/m} =\frac{m n - \nu}{\tau} + O(n^{-1}) \qquad (n\to\infty).
\end{equation}
\end{theorem}

Yet another criterion, which appears to be new, is stated in the following theorem.  Note that this criterion does {\em not} follow simply from closure under multiplication (Thm.\ref{thm:haymancriteria}.I): entire functions of the form $f(z^m)$, $m\geq 2$, are never $H$-admissible since only every $m$-th of their Maclaurin coefficients are non-zero, which is inconsistent with the Stirling-type formula~\eqref{eq:stirling} forcing all coefficients to be non-zero.

\begin{theorem}\label{thm:criterion} Let $f(z)$ be an entire $H$-admissible function that satisfies the capture condition, for some $\alpha, \beta >0$, in the form
\[
b(r) = \alpha r^\beta(1+o(1)) \qquad (r\to\infty).
\]
Then, for each positive integer $m$, the entire function $f_*(z):=e^z f(z^m)$ is $H$-admissible, too. 
\end{theorem}

\begin{proof} We will show that $f_*(z)$ inherits the constitutive conditions of positivity, capture, locality and decay from $f(z)$. The auxiliary functions  of $f(z)$ will be denoted by $a(r)$, $b(r)$ and by $\delta(r)$ we denote the angle that separates the segment of locality from that of  decay. Accordingly, the corresponding quantities of $f_*(z)$ will be denoted by $a_*(r)$, $b_*(r)$ and $\delta_*(r)$. We will use the following elementary inequality:
\begin{equation}\label{eq:cosineq}
|e^{r(e^{i\theta}-1)}| = e^{r(\cos \theta - 1)} \leq e^{-r \theta^2/5} \qquad (r\geq 0,\; |\theta| \leq \pi).
\end{equation}

\medskip

{\em Positivity and capture.}
Since $f_*(r) = e^r f(r^m)$ there holds, inherited from the positivity condition for the factor $f(r^m)$, that eventually $f_*(r)>0$ as $r\to\infty$. In particular, for $r>0$ large enough, we have
\[
a_*(r) := r \frac{f_*'(r)}{f_*(r)} = r + m a(r^m), \qquad b_*(r) := r a_*'(r) = r + m^2 b(r^m).
\]
By the assumption on $b(r)$ we thus get the capture condition for $f_*(z)$ specifically in the form
\[
b_*(r) = r + m^2\alpha r^{m\beta} (1+ o(1)) \to \infty \qquad (r\to\infty),
\]
which implies the estimate
\begin{equation}\label{eq:bstar}
r^{-\max(m\beta/2,1/2)} = O(b_*(r)^{-1/2})\qquad (r\to\infty)
\end{equation}
that we will use below when proving the decay estimates for $f_*(z)$.

\medskip

{\em Locality and decay.} We split, for $r$ sufficiently large, the angles $0\leq |\theta| \leq \pi$ into three segments $S_1, S_2, S_3$ such that $f_*(z)$ localizes in the first and decays in the other two:

\begin{itemize}
\item[$(S_1$)] $0\leq |\theta| \leq \min(\delta(r^m)/m,r^{-2/5}) =:\delta_*(r)$. 

\medskip
\noindent Here the $f(z^m)$ factor of $f_*(z)$ localizes and we get
\begin{align*}
f_*(re^{i\theta}) &= e^{i r(\sin\theta-\theta) + r( \cos\theta -1 + \theta^2/2)} \cdot e^re^{i \theta r - \theta^2 r/2}   f(r^m e^{im\theta})\\*[2mm] 
&= e^{O(r^{-1/5})} \cdot e^r f(r^m) e^{i \theta r - \theta^2 r/2} e^{i\theta m a(r^m) - \theta^2 m^2 b(r^m)/2}  (1+o(1))\\*[2mm] 
 &= f_*(r) e^{i\theta  a_*(r) - \theta^2 b_*(r)/2}  (1+o(1)),
\end{align*}
that is, $f_*(z)$ localizes in the segment $S_1$.\\*[-3mm]

\item[$(S_2$)] $\delta(r^m)/m \leq |\theta| \leq r^{-2/5}$, a segment that could possibly be void.

\medskip
\noindent By combining \eqref{eq:cosineq} and the third limit in \eqref{eq:bdelta}  with the assumption on $b(r)$ we obtain, for $r$ sufficiently large, 
\[
|e^{r (e^{i\theta}-1)}| \leq e^{-r\delta(r^m)^2/(5m^2)} \leq e^{-r^{1-m\beta}}.
\]
We thus get, inherited from the decay of $f(z^m)$, the decay of $f_*(z)$ in the form 
\[
f_*(re^{i\theta}) = e^{r (e^{i\theta}-1)} \cdot  \frac{o(e^r f(r^m))}{b(r^m)^{1/2}} = r^{-m\beta/2} e^{-r^{1-m\beta}}\cdot o(f_*(r)) =  \frac{o(f_*(r))}{b_*(r)^{1/2}},
\]
where we have used $r^{-m\beta/2} e^{-r^{1-m\beta}} = O\big(r^{-\max(m\beta/2,1/2)}\big)$ and \eqref{eq:bstar}. \\*[-3mm]

\item[$(S_3$)] $\max(\delta(r^m)/m, r^{-2/5}) \leq |\theta| \leq \pi$. 

\medskip
\noindent By \eqref{eq:maxmod} and \eqref{eq:cosineq} we get the decay of $f_*(z)$ in the form
\[
f_*(re^{i\theta}) = e^{r (e^{i\theta}-1)} \cdot O(e^r f(r^m)) = e^{-r^{1/5}/5} \cdot O(f_*(r)) = \frac{o(f_*(r))}{b_*(r)^{1/2}},
\]
where we have used $e^{-r^{1/5}/5} = o\big(r^{-\max(m\beta/2,1/2)}\big)$ and \eqref{eq:bstar}. 
\end{itemize}~\\*[-10mm]
\end{proof}

\section{Asymptotic Expansion of the Number of Involutions}\label{app:inv}

The exponential generating function of the number $I_n$ of involutions in the symmetric group of order $n!$ is given by
\begin{equation}\label{eq:finv}
f(z) = \sum_{n=0}^\infty \frac{I_n}{n!} z^n = e^{z + z^2/2},
\end{equation}
cf.~\cite[Example~II.13]{MR2483235}. Since $f(z)$ is $H$-admissible by Thm.~\ref{thm:haymancriteria}.III, with auxiliary functions
\[
a(r) = r f'(r)/f(r) = r+r^2,\qquad   b(r) = r a'(r) = r + 2r^2,
\]
we plug $r=n^{1/2}$ into the normal approximation \eqref{eq:CLT} and get immediately, as $n\to\infty$, 
\begin{equation}\label{eq:iotalead}
a_n=\frac{I_n}{n!} = \frac{e^{\sqrt{n}+n/2}}{n^{n/2}\sqrt{2\pi(\sqrt{n}+2n)}}\left(e^{-\frac{n}{2\sqrt{n}+4n}}
+ o(1)\right) = \frac{e^{\sqrt{n}-1/4}}{2\sqrt{\pi n}} \left(\frac{e}{n}\right)^{n/2} (1+ o(1)).
\end{equation}
Note that this straightforward calculation is considerably simpler than solving $a(r_n)=n$ for 
\begin{equation}\label{eq:rinv}
r_n = \sqrt{n+\frac14}-\frac12
\end{equation}
and using the Stirling-type formula \eqref{eq:stirling} as done, e.g., in \cite[Example~5.17]{MR2172781}. 

For the purposes of Sect.~\ref{sect:jasz} we need an asymptotic expansion of $a_n$ that goes beyond the leading order term. Following the approach of Wimp and Zeilberger \cite{MR808671}, which is based on the Birkhoff--Trjitzinsky theory, we first infer from the differential equation $f'(z) = (1+z)f(z)$ the three-term recursion
\[
(n+2) a_{n+2} = a_{n+1} + a_n \quad (n\geq 0), \quad a_0 = a_1 = 1,
\]
and then get, by inserting the corresponding asymptotic Birkhoff series\footnote{Note that this series is generally only determined up to a multiplicative constant $K>0$ which has to be computed by other methods. Here the value 
$K = 1/(2e^{1/4}\pi^{1/2})$ is obtained by matching \eqref{eq:Birkhoff} to \eqref{eq:iotalead}.}
\begin{equation}\label{eq:Birkhoff}
a_n \sim K \cdot \frac{e^{\sqrt{n}}}{\sqrt{n}} \left(\frac{e}{n}\right)^{n/2} \big(1 + c_1 n^{-1/2} + c_2 n^{-1} + c_3 n^{-3/2} + \cdots \big) \qquad (n\to\infty)
\end{equation}
into the recurrence and by solving recursively for the coefficients, that
\begin{equation}\label{eq:iotaexpansion}
a_n = \frac{e^{\sqrt{n}-1/4}}{2\sqrt{\pi n}} \left(\frac{e}{n}\right)^{n/2} \Big(
 1+\frac{7}{24}n^{-1/2} - \frac{215}{1152} n^{-1}
 -\frac{18013}{414720}n^{-3/2} + O\big(n^{-2}\big) \Big)
 \end{equation}
which is the expansion used in Sect.~\ref{sect:jasz}. To get an expansion of $I_n$ itself, we multiply \eqref{eq:iotaexpansion} with the well-known expansion of Stirling's formula (cf. \cite[p.~760]{MR2483235}),
\begin{equation}\label{eq:tauexpansion}
n! = \sqrt{2\pi n}\left(\frac{n}{e}\right)^n \Big( 1+\frac{n^{-1}}{12}+O\big(n^{-2}\big) \Big),
\end{equation}
which gives
\begin{equation}\label{eq:I_n_asympt}
I_n = \frac{e^{\sqrt{n}-1/4}}{\sqrt{2}} \left(\frac{n}{e}\right)^{n/2} \Big( 1+\frac{7}{24}n^{-1/2}-\frac{119}{1152}n^{-1}
-\frac{7933}{414720}n^{-3/2} + 
O\big(n^{-2}\big) \Big).
\end{equation}

\medskip

\begin{remark}
By using Laplace's method for the Cauchy integrals that give the Maclaurin coefficients $a_n$, Moser and Wyman \cite[Eq.~(3.40)]{MR68564} obtained the expansion \eqref{eq:I_n_asympt}
truncated at order $O(n^{-3/2})$. Knuth \cite[pp.~62--64]{Knuth3} gets the slightly less precise
\[
I_n  = \frac{e^{\sqrt{n}-1/4}}{\sqrt{2}} \left( \frac{n}{e}\right)^{n/2} \Big( 1+\frac{7}{24}n^{-1/2} + O\big(n^{-3/4}\big) \Big),
\]
by first applying Laplace's method to the sum (obtained from expanding $f(z)=e^z \cdot e^{z^2/2}$)
\[
I_n = \sum_{0\leq k \leq n/2} \frac{n!}{(n-2k)!\,2^k k!},
\]
and then using Euler--Maclaurin's formula to replace the remaining sums by integrals. 
\end{remark}

\section{Evidence for the Linear Form Hypothesis}\label{app:evidence}

\subsection{Reducing the number of non-zero terms} By using a meticulous change of variables, it is possible to significantly reduce the number of non-zero terms in the putative polynomial coefficients of \eqref{eq:hypoL}. In doing so we hugely improve the stability of the numerical method to be discussed in \ref{app:num_evidence}.

Consistent with the already observed shared sparsity patterns, the transformation $t=\psi_h(s)$ as defined in \cite[Lemma~3.5]{arxiv:2301.02022} does not only reduce the number of non-zero terms in the polynomial coefficients for $\beta=2$ (see \cite[Eqs.~(3.20a/c, 3.21b)]{arxiv:2301.02022}) but to exactly the same degree here. 
We recall that $t=\psi_h(s)$ maps $s\in\R$ analytically and monotonically increasing onto $-\infty < t < h^{-1}$ if $h>0$, and  expands around $sh = 0$ into the power series
\begin{equation}\label{eq:psiexpan}
t = \psi_h(s) = s-\frac{3 s^2}{10} h-\frac{ s^3}{350}h^2+\frac{479 
   s^4}{63000}h^3+\cdots\,.
\end{equation}
Now, if we plug \eqref{eq:psiexpan} into \eqref{eq:Gexpan} and take  $p_{+,jk}$ as shown in \eqref{eq:F_beta_j}, then a routine calculation with truncated power series (and their reversals) implies that the linear form hypothesis for up to $m=3$ is {\em equivalent} to the following statement: the expansion
\[
E_\pm(\psi_{h_\nu}(s);\nu) = F_{\pm}(s) + \tilde E_{\pm,1}(s) h_\nu + \tilde E_{\pm,2}(s) h_\nu^2 + \tilde E_{\pm,3}(s) h_\nu^3 + O(h_\nu^4),
\]
as $h_\nu\to 0^+$ with $s$ being any fixed real number, has expansion terms satisfying
\begin{subequations}\label{eq:hypoFconcrete}
\begin{align}
\tilde E_{\pm,1}(s) &= -\frac{2}{5} F_\pm''(s),\label{eq:Gz1}\\*[1mm]
\tilde E_{\pm,2}(s) &= \frac{9}{175} F_\pm'(s) - \frac{32s}{175}F_\pm''(s) + \frac{4}{50} F_\pm^{(4)}(s),\\*[1mm]
\tilde E_{\pm,3}(s) &= \frac{268s}{7875} F_\pm'(s) - \frac{48s^2}{875} F_\pm''(s) - \frac{578}{7875}  F_\pm^{'''}(s) + \frac{64s}{875} F_\pm^{(4)}(s) - \frac{8}{750} F_\pm^{(6)}(s).
\end{align}
\end{subequations}
We will check the validity of the latter equations numerically, thereby using the first one, which we know to be unconditionally true, as a sanity check. In doing so, $\tilde E_{\pm,1}, \tilde E_{\pm,2}, \tilde E_{\pm,3}$ have to be evaluated using Thm.~\ref{thm:detexpan}. To prepare, we transform Lemma~\ref{lem:Vkernexpan} accordingly:

\begin{lemma}\label{lem:auxtrans} For $h>0$, we consider the $m=3$ case of the kernel expansion in Lemma~\ref{lem:Vkernexpan},
\[
K_{(h)}(x,y):= V_{\Ai}(x,y) + K_1(x,y) h + K_2(x,y) h^2 + K_3(x,y) h^3.
\] 
Then $t=\psi_h(s)$ induces the symmetrically transformed kernel
\begin{subequations}\label{eq:auxkernelexpan}
\begin{align}
\tilde K_{(h)}(x,y) &:= \sqrt{\psi_h'(x)\psi_h'(y)} \, K_{(h)}(\psi_h(x),\psi_h(y))\\
\intertext{which expands as}
\tilde K_{(h)}(x,y) &= V_{\Ai}(x,y) + \tilde K_1(x,y)h + \tilde K_2(x,y)h^2 + \tilde K_3(x,y)h^3 + h^4 \cdot O\big(e^{-(x+y)/2}\big),
\end{align}
uniformly in $s_0 \leq x, y \leq 2h^{-1}$ as $h \to 0^+$, $s_0$ being a fixed real number. Written in terms of the (scaled) elementary symmetric polynomials $u=(x+y)/2$, $v=xy$, the kernels are
\begin{align*}
\tilde K_1(x,y) &= -\frac{u}{5}\Ai(u) + \frac{u^2-v}{10}\Ai'(u),\\*[1mm]
\tilde K_2(x,y) &= \frac{7 u^5-14 u^3 v+7 u v^2-88 u^2+52 v}{700} \Ai(u)
 + \frac{-12 u^3+12 u v+5}{350}\Ai'(u),\\*[1mm]
\tilde K_3(x,y) &= \frac{-69 u^6+117 u^4 v-27 u^2 v^2 -990 u^3-21v^3 +1110 u v-140}{31500}\Ai(u)\\*[1mm]
& \quad + \frac{21 u^7-63 u^5 v+63 u^3 v^2 -1378 u^4-21 uv^3+1976 u^2 v-598 v^2+100 u}{31500}\Ai'(u).
\end{align*}
\end{subequations}
Preserving uniformity, the kernel expansion can repeatedly be differentiated w.r.t. $x$, $y$.
\end{lemma}
\begin{proof} We argue as in the proof of \cite[Lemma~3.5]{arxiv:2301.02022}. Then, starting with \eqref{eq:K}, a routine calculation with truncated power series yields the asserted formulae for $\tilde K_1, \tilde K_2, \tilde K_3$. 
\end{proof}

\subsection{Numerical evidence}\label{app:num_evidence} 

We provide numerical evidence for \eqref{eq:hypoFconcrete} while at the same time reclaiming, as yet another token of consistency, the rational coefficients independently of the theory in Sect.~\ref{sect:functionalform}. That is, we check for
\begin{subequations}\label{eq:hypoF}
\begin{align}
\tilde E_{\pm,1}(s) &= a_{11} F_\pm''(s),\\*[1mm]
\tilde E_{\pm,2}(s) &= a_{21} F_\pm'(s) + a_{22} s F_\pm''(s) + a_{23} F_\pm^{(4)}(s),\\*[1mm]
\tilde E_{\pm,3}(s) &= a_{31} s F_\pm'(s) + a_{32} s^2 F_\pm''(s) + a_{33}  F_\pm^{'''}(s) + a_{34} s F_\pm^{(4)}(s) + a_{35} F_\pm^{(6)}(s),
\end{align}
\end{subequations}
with yet to be determined coefficients $a_{jk}\in\Q$, applying the following algorithm:

\smallskip

\begin{itemize}\itemsep=3pt
\item in Eqs.~\eqref{eq:hypoF}, insert some distinct points $s_{jk}$, $k=1,\ldots,n_j$, numerically into $\tilde E_{\pm,j}(s)$ and into the  functions of $s$ on the right hand side, up to a controlled error level $\epsilon$;
\item taking $n_1=1$, $n_2=3$ and $n_3=5$, solve the resulting $n_j\times n_j$ linear systems  for some unique  numerical values of the coefficients $a_{jk}$ ($k=1,\ldots,n_j$);
\item compute a rational best approximation of the $a_{jk}$ with sufficiently small error;
\item using these rational values, check whether both sides of Eqs.~\eqref{eq:hypoF} agree numerically on a sufficiently large interval $[s_0, s_1]$ of $s$, up to the error level $\epsilon$.
\end{itemize}

\smallskip

\noindent
By analyticity an agreement on any interval $[s_0,s_1]$ implies the agreement on the real line.

Actually, we apply the algorithm to the more general case which interpolates between the two sign choices in 
the transformed variant of \eqref{eq:Gexpan}: i.e., for a fixed $-1\leq z \leq 1$ we consider\footnote{Cf. \cite[Eq.~(3.22)]{arxiv:2301.02022} for the upper bound $2h^{-1}$ in the Fredholm determinant on the left.}
\begin{equation}\label{eq:Gz_tilde}
\begin{aligned}
\det(I-z\tilde K_{(h)})\big|_{L^2(s,2h^{-1})} &= F_z(s) + \tilde E_{z,1}(s) h_\nu + \tilde E_{z,2}(s) h_\nu^2 + \tilde E_{z,3}(s) h_\nu^3 + O(h_\nu^4), \\*[2mm]
F_z(s) &= \det(I-zV_{\Ai})\big|_{L^2(s,\infty)},
\end{aligned}
\end{equation}
and check (numerically) the expressions \eqref{eq:hypoF} to be valid with $z$ replacing $\pm$. As it turns out, we get compelling numerical evidence that the rational coefficients $a_{jk}$ are {\em independent} of $z$.
Note that, when $0\leq z\leq 1$, the quantity $\theta:=1-z$ corresponds to a thinning rate  in the corresponding point processes as in \cite[§3.3]{MR3647807}; see also Rem.~\ref{rem:thinning}.

\subsubsection{Implementation}
The algorithm can straightforwardly be implemented in our Matlab toolbox \cite{MR2895091} (with the extension to resolvents and traces developed in \cite{MR3647807}). Note the symbolic ``look and feel,'' the formulae of Thm.~\ref{thm:detexpan} and Lemma~\ref{lem:auxtrans} are, basically, taken just as stated. 

\medskip

\noindent
We walk the reader through the code by commenting on it block-wise:

\begin{itemize}
\item[I)]  the linear form hypothesis is checked for up to $m=3$
\end{itemize}

\begin{lstlisting}[firstnumber=1,numbers=left,basicstyle=\ttfamily\footnotesize]
m = 3; 
\end{lstlisting}

\begin{itemize}
\item[II)]  implementation of the kernel formulae of Lemma~\ref{lem:auxtrans} with a prefactor $-1 \leq z \leq 1$
\end{itemize}

\begin{lstlisting}[firstnumber=2,numbers=left,basicstyle=\ttfamily\footnotesize]
V = @(x,y) z*airy((x+y)/2)/2;
p{1} = @(u,v) -u/5;
q{1} = @(u,v) (u.^2 - v)/10;
p{2} = @(u,v) (7*u.^5 - 14*u.^3.*v + 7*u.*v.^2 - 88*u.^2 + 52*v)/700;
q{2} = @(u,v) (- 12*u.^3 + 12*u.*v + 5)/350;
p{3} = @(u,v) (- 69*u.^6 + 117*u.^4.*v - 27*u.^2.*v.^2 - 990*u.^3 ...
    - 21*v.^3 + 1110*u.*v - 140)/31500;
q{3} = @(u,v) (21*u.^7 - 63*u.^5.*v + 63*u.^3.*v.^2 - 1378*u.^4 ...
    - 21*u.*v.^3 + 1976*u.^2.*v - 598*v.^2 + 100*u)/31500;
for j=1:m
    Kuv{j} = @(u,v) p{j}(u,v).*airy(u) + q{j}(u,v).*airy(1,u);
    K{j} = @(x,y) z*Kuv{j}((x+y)/2, x.*y);
end
\end{lstlisting}

\pagebreak 
\begin{itemize}
\item[III)]  implementation of the general formulae of Thm.~\ref{thm:detexpan} 
\end{itemize}

\begin{lstlisting}[firstnumber=15,numbers=left,basicstyle=\ttfamily\footnotesize]
tr = @(T) trace(T);
I = @(K) eye(size(K));
E = @(K,L) (I(K)-K)\L;
d{1} = @(K0,K1) - tr(E(K0,K1));
d{2} = @(K0,K1,K2) tr(E(K0,K1))^2/2 - tr(E(K0,K1)^2)/2 - tr(E(K0,K2));
d{3} = @(K0,K1,K2,K3) - tr(E(K0,K1))^3/6 + tr(E(K0,K1))*tr(E(K0,K1)^2)/2 ...
    - tr(E(K0,K1)*E(K0,K2) + E(K0,K2)*E(K0,K1))/2 - tr(E(K0,K1)^3)/3 ...
    + tr(E(K0,K1))*tr(E(K0,K2)) - tr(E(K0,K3));
\end{lstlisting}

\begin{itemize}
\item[IV)]  implementation of Thm.~\ref{thm:detexpan} applied to Lemma~\ref{lem:auxtrans}, yielding $\tilde E_{z,j}(s)$ in \eqref{eq:Gz_tilde}
\end{itemize}

\begin{lstlisting}[firstnumber=23,numbers=left,basicstyle=\ttfamily\footnotesize]
function val = Ez_tilde(j,s) % \tilde E_{z,j}(s)
    ops = {op(K{1},[s,inf]), op(K{2},[s,inf]), op(K{3},[s,inf])};
    val = det1m(op(V,[s,inf]))*OperatorTerm(d{j}, op(V,[s,inf]), ops{1:j});
end
\end{lstlisting}

\begin{itemize}
\item[V)]  implementation of $F_z(s)$ as defined in \eqref{eq:Gz_tilde}
\end{itemize}

\begin{lstlisting}[firstnumber=27,numbers=left,basicstyle=\ttfamily\footnotesize]
function val = Fz(s) % F_z(s)
    val = det1m(op(V,[s,inf]));
end
\end{lstlisting}

\begin{itemize}
\item[VI)]  polynomial approximation of $F_z(s)$ on the interval $[-7,4]$
\end{itemize}

\begin{lstlisting}[firstnumber=30,numbers=left,basicstyle=\ttfamily\footnotesize]
ival = [-7, 4];
F = chebfun(vec(@(s) Fz(s)), ival);
\end{lstlisting}

\begin{itemize}
\item[VII)]  corresponding polynomial approximations of the derivatives $F_z^{(k)}(s)$ on that interval for $k=1,\ldots,2m$; good to just an error of about $10^{-14+3k/2}$ for $1/8 \leq |z| \leq 1$
\end{itemize}

\begin{lstlisting}[firstnumber=32,numbers=left,basicstyle=\ttfamily\footnotesize]
for k=1:2*m
    f{k} = diff(F,k); % accuracy < 2*10^{-14+1.5*k} (0.125<=|z|<=1)
end
\end{lstlisting}

\begin{itemize}
\item[VIII)]  list of the ansatz functions whose linear combinations give the RHS of \eqref{eq:hypoF}
\end{itemize}

\begin{lstlisting}[firstnumber=35,numbers=left,basicstyle=\ttfamily\footnotesize]
ff = @(s) {[f{2}(s)], [f{1}(s) s*f{2}(s) f{4}(s)], ...
    [s*f{1}(s) s^2*f{2}(s) f{3}(s) s*f{4}(s) f{6}(s)]};
\end{lstlisting}  

\begin{itemize}
\item[IX)]  choice of sample points to set up the linear system
\end{itemize}
    
\begin{lstlisting}[firstnumber=37,numbers=left,basicstyle=\ttfamily\footnotesize]
ss = {[-3.5], [-4.5, -3.5, -2.5], [-4.5 -4 -3.5 -3 -2.5]}; % sample points
\end{lstlisting}  

\begin{itemize}
\item[X)]  finally, the algorithm is started, looping through $j=1,2,3$
\end{itemize}

\begin{lstlisting}[firstnumber=38,numbers=left,basicstyle=\ttfamily\footnotesize]
ind = @(j) struct('type','{}','subs',{{j}});
for j=1:m
\end{lstlisting} 

\begin{itemize}
\item[XI)]  first, create the linear system for the $a_{jk}$ from \eqref{eq:hypoF} evaluated in the sample points
\end{itemize}

\begin{lstlisting}[firstnumber=40,numbers=left,basicstyle=\ttfamily\footnotesize]
    E = zeros(length(ss{j}),1); 
    M = zeros(length(ss{j}));   
    for k=1:length(E)
        s = ss{j}(k);
        E(k) = Ez_tilde(j,s);  
        M(k,:) = subsref(ff(s), ind(j)); % k-th row of matrix
    end
\end{lstlisting} 

\begin{itemize}
\item[XII)]  second, solve that linear system and perform a rational reconstruction of the $a_{jk}$
\end{itemize}

\begin{lstlisting}[firstnumber=47,numbers=left,basicstyle=\ttfamily\footnotesize]
    D = 10000; % bound of denominator
    [num, denom] = rat(M\E, 1/(2*D^2));
\end{lstlisting} 

\begin{itemize}
\item[XIII)]  third, take the values of the $a_{jk}$ and compute polynomial approximations of the LHS and RHS in \eqref{eq:hypoF} on the interval $[-7,4]$; compute the $L^\infty$-norm of their difference
\end{itemize}

\begin{lstlisting}[firstnumber=49,numbers=left,basicstyle=\ttfamily\footnotesize]
    LHS = chebfun(vec(@(s) Ez_tilde(j,s)), ival, 48);
    RHS = chebfun(vec(@(s) subsref(ff(s), ind(j))*(num./denom)), ival, 48);
    err = norm(LHS-RHS, inf);
\end{lstlisting} 

\begin{itemize}
\item[XIV)]   if the numerical tolerances are met, print out the rational values of the $a_{jk}$
\end{itemize}

\begin{lstlisting}[firstnumber=52,numbers=left,basicstyle=\ttfamily\footnotesize]
    if err < 2*10^(-14+3*j) % bound = numerical error level of RHS
        for k=1:length(E)
            fprintf("a_{%i%i} = %i/%i, ", j, k, num(k), denom(k));
        end
        fprintf("\n");
\end{lstlisting} 

\begin{itemize}
\item[XV)]  otherwise, tell that \eqref{eq:hypoF} is not met within the numerical tolerances
\end{itemize}

\begin{lstlisting}[firstnumber=57,numbers=left,basicstyle=\ttfamily\footnotesize]
    else
        fprintf("j = %i: LHS = RHS does NOT meet numerical tolerances\n", j)
    end
end
\end{lstlisting}

\noindent
When run on $z=\pm n/8$ ($n=1,\ldots 8$) the output is consistently
\medskip

{\footnotesize
\begin{verbatim}
a_{11} = -2/5, 
a_{21} = 9/175, a_{22} = -32/175, a_{23} = 2/25, 
a_{31} = 268/7875, a_{32} = -48/875, a_{33} = -578/7875, a_{34} = 64/875, a_{35} = -4/375,
\end{verbatim}}

\medskip

\noindent
which reclaims the coefficients in \eqref{eq:hypoFconcrete}. Clearly, by applying the algorithm adapted to the case $\beta=2$ as a further sanity-check, we were able to reproduce the rational coefficients of the rigorously established formulae \cite[Eqs.~(3.20a/c, 3.21b)]{arxiv:2301.02022}.

\begin{remark}
For $j=3$, where the rather inaccurate polynomial approximation of $F_z^{(6)}(s)$ enters the calculations, the numerical errors in evaluating the RHS of \eqref{eq:hypoF} are just about matching the tolerance required for the rational reconstruction of the  $a_{3k}$. Still, there are no problems of robustness here: we get consistently the same results for a variety of $z$ (and a variety of different choices of the sample points).
However, it should be clear that extending the numerical approach to  $j=4$, or even larger, would require significantly more accurate numerical representations of the higher order derivatives of $F_z(s)$. 
\end{remark}

\subsection{Further insight into the observed \boldmath$z$-independence of form\unboldmath}

\subsubsection{Why it had to be expected}\label{app:thinning}
The $z$-independence of the coefficients that was observed in \ref{app:num_evidence} is actually less surprising than one might think prima facie. In fact, we will show here (using the original $t$-variable) that the straightforward generalization of the linear form hypothesis,
\[
E_{z,j} = \sum_{k=1}^{2j} p_{z,jk} F_z^{(k)},\quad p_{z,jk} \in \Q[t],
\]
which is clearly supported by the numerical experiments of \ref{app:num_evidence}, suffices to prove such a $z$-independence of the coefficient polynomials $p_{z,jk}$. (For the relation of the parameter $z$ to introducing a thinning rate into the underlying point process as in \cite[§3.3]{MR3647807}, see
the discussion below Eq.~\eqref{eq:Gz_tilde} and Rem.~\ref{rem:thinning}.)

For brevity we call the {\em $\xi$-variant} of an operator expression that is given in terms of a kernel to be the corresponding operator expression obtained by multiplying the kernel by $\xi$. In random matrix theory this often corresponds to generating functions of higher-order gap probabilities and can be given a probabilistic meaning. To begin with, the $z$-variant of \eqref{eq:Gexpan} is
\[
E_z(t;\nu) = F_z(t) + \sum_{j=1}^m E_{z,j}(t) h_\nu^j + O(h_\nu^{m+1})
\]
with the functions $F_z$ that were introduced in \eqref{eq:Gz_tilde}. We write the  $z^2$-variant of \eqref{eq:hard2soft2} as
\[
E_2^\text{hard}(\phi_\nu(t);\nu;z) = F_{2,z}(t) + \sum_{j=1}^m E_{2,z,j}(t) h_\nu^j + O(h_\nu^{m+1}),
\]
where $F_{2,z}$ is the $z^2$-variant of the Tracy--Widom distribution $F_2$. By \cite[Eqs.~(4.4--5)]{MR2229797} the factorizations
\eqref{eq:detfactor} generalize in the form
\[
E_2^\text{hard}(\phi_\nu(t);\nu;z) = E_z(t;\nu)\cdot E_{-z}(t;\nu), \quad F_{2,z}(t) = F_{z}(t) \cdot F_{-z}(t).
\]
Finally, \cite[Eqs.~(4.7--9) and (4.12)]{MR2229797} allow us to generalize \eqref{eq:TWtheory} in the form
\[
\frac{F_{2,z}'(t)}{F_{2,z}(t)} = q_z'(t)^2 - t q_z(t)^2 + q_z(t)^4,\qquad F_{\pm z}(t) = e^{\mp \frac{1}{2} \int_t^\infty q_z(x)\,dx}\sqrt{F_{2,z}(t)}.
\]
Here $q_z$ is the solution of the Painlevé II equation \eqref{eq:PII}, subject to the boundary condition
\[
q_z(s) \sim z \Ai(s)\quad(s\to\infty),
\]
which is the only place where $z$ explicitly enters this set of formulae. We note the antisymmetry $q_{-z} = -q_z$, so that $q_z$ corresponds, for $z=\pm 1$, to the Hastings--McLeod solution  and, for $-1<z<1$, to the Ablowitz--Segur solutions  of Painlevé~II.

At this stage we observe that the {\em algebra} of all these formulae remains just the same as for the case $z=\pm 1$ which we had considered in Sect.~\ref{sect:functionalform} and, for $\beta=2$, in \cite[Appendix~B]{arxiv:2301.02022}. Neither there (with the exception of the super-exponential decay of $q(t)$ as $t\to\infty$, which is all the more true for $q_z(t)$) nor in the present paper we used any specifics of the boundary condition but just that $q$ is a solution of Painlevé II. Therefore, with literally the same purely algebraic proof and calculations as for \eqref{eq:E2j_Structure} we get
\[
E_{2,z,j} = \sum_{k=1}^{2j} p_{2,jk} F_{2,z}^{(k)}
\]
where the $p_{2,jk}$ are {\em exactly} the same polynomials as defined in \eqref{eq:E2j_Structure}, independently of $z$. Based on the generalized linear form hypothesis, by performing literally the same purely algebraic calculations as in Sect.~\ref{sect:higherorder} we get to express the $p_{\pm z,jk}$ in terms of the $p_{2,jk}$ with exactly the same algebraic relations as in \eqref{eq:hyposol}. We thus have, as claimed,
\[
p_{z,jk} = p_{+,jk}.
\]

\subsubsection{Partial analytical evidence}\label{app:ana_evidence} 

With the exception of the boundary cases $z=\pm 1$, for which we provided proof in Sect.~\ref{sect:firstorder}, already the first order formula suggested by the numerical experiments of \ref{app:num_evidence} is lacking proof:
\[
\tilde E_{z,1}(s) = \frac{2}{5} F_z''(s).
\]
In lieu of any better ideas, we now show that at least, as $z\to 0$, 
\[
\tilde E_{z,1}(s) = \frac{2}{5} F_z''(s) + O(z^3).
\]
In fact, using the notation and results of \cite[Rem.~3.1]{arxiv.2206.09411}, we can represent $F_z''(s)$ in the form
\[
F_z''(s) = F_z(s)\Big(z^2\big(\tr((I-zV_{\Ai})^{-1} V_{\Ai}')\big)^2-z^2\tr\big(((I-zV_{\Ai})^{-1} V_{\Ai}')^2 \big)-z \tr\big((I-zV_{\Ai})^{-1} V_{\Ai}''\big)\Big),
\]
where the resolvent and traces are understood to act on $L^2(s,\infty)$ and the derived kernels are
\[
V_{\Ai}'(x,y) =  \frac{1}{2}\Ai'\left(\frac{x+y}{2}\right),\qquad V_{\Ai}''(x,y) =  \frac{x+y}{4}\Ai\left(\frac{x+y}{2}\right).
\]
On the other hand, from \eqref{eq:Gz_tilde}, Thm.~\ref{thm:detexpan} and Lemma~\ref{lem:auxtrans} we get the representation
\[
\tilde E_{z,1}(s) = F_z(s)\cdot z \tr\big((I-z V_{\Ai})^{-1} \tilde K_1 \big),
\]
where we rewrite the kernel $\tilde K_1(x,y)$ in the equivalent form
\[
\tilde K_1(x,y)  = \frac{2}{5}\big(L(x,y) - V_{\Ai}''(x,y)\big), \quad L(x,y) := \frac{(x-y)^2}{16} \Ai'\left(\frac{x+y}{2}\right).
\]
We thus have to show that
\[
\tr\big((I-z V_{\Ai})^{-1} L \big) = z \big(\tr((I-zV_{\Ai})^{-1} V_{\Ai}')\big)^2-z \tr\big(((I-zV_{\Ai})^{-1} V_{\Ai}')^2 \big) + O(z^2),
\]
or equivalently, by expanding the resolvent into a Neumann series,
\begin{subequations}
\begin{align}
\tr L &= 0,\label{eq:trace2nd}\\*[1mm]
\tr(V_{\Ai} L) &=( \tr V_{\Ai}')^2 - \tr(V_{\Ai}'^2\,),\label{eq:trace3rd}
\end{align}
\end{subequations}
where products of kernels are understood as products of integral operators acting on $L^2(s,\infty)$. Since the diagonal of the integral operator $L$ vanishes, \eqref{eq:trace2nd} is
trivially true:
\[
\tr L = \int_s^\infty L(x,x)\,dx = 0,
\]
whereas \eqref{eq:trace3rd} can be shown to be true as follows. Upon writing briefly $f(x)=\Ai(x)$, we get by symmetry and the change of coordinates $\xi = (x+y)/2$, $\eta = y-x$ in the double integral:
\begin{multline*}
\tr(V_{\Ai}L) = \frac{1}{32}\int_s^\infty \int_s^\infty (x-y)^2 f((x+y)/2) f'((x+y)/2) \,dy\,dx\\*[1mm]
=  \frac{1}{16} \int_s^\infty \int_x^\infty (x-y)^2  f((x+y)/2) f'((x+y)/2) \,dy \, dx\\*[1mm]
= \frac{1}{16} \int_s^\infty \int_0^{2\xi-2s} \eta^2 f(\xi)f'(\xi) \,d\eta\,d\xi = \frac{1}{6} \int_s^\infty (\xi-s)^3 f(\xi)f'(\xi) \,d\xi,
\end{multline*}
the fourth derivative of which is $D_s^4 \tr(V_{\Ai} L) = f(s) f'(s)$.
Analogously we get
\begin{align*}
\tr V_{\Ai}' &= \frac{1}{2}\int_s^\infty f'(x)\,dx = -\frac{1}{2} f(s),\\*[1mm] 
\tr(V_{\Ai}'^2\,) &= \frac{1}{4}\int_s^\infty  \int_s^\infty f'((x+y)/2)^2\,dy\,dx = \int_s^\infty (\xi-s) f'(\xi)^2\,d\xi,\
\end{align*}
and thus, by using the Airy differential equation $f''(x) = x f(x)$, the same fourth derivative:
\[
D_s^4 \big( (\tr V_{\Ai}')^2 - \tr(V_{\Ai}'^2\,) \big) = f(s)f'(s).
 \]
Because of the super-exponential decay of both sides of \eqref{eq:trace3rd} as $s\to \infty$, and that of all of their derivatives,  the equality of the fourth derivatives can be lifted to the original expressions.

\bibliographystyle{spmpsci}
\bibliography{paper}

\end{document}